\documentclass[reqno,11pt, psfig]{amsart}
\usepackage{cases}
\usepackage{mathrsfs}
\usepackage[T1]{fontenc}
\usepackage{mathrsfs}
\usepackage{amsmath,latexsym,amssymb,amsfonts,amsbsy, amsthm}
\usepackage{bm}
\usepackage[usenames]{color}
\usepackage{xspace,colortbl}
\usepackage{epsfig}
\usepackage{graphicx}
\usepackage{subfigure}
\usepackage{amsmath,amsfonts,amsthm,amssymb,amscd}
\input amssym.def
\input amssym.tex
\usepackage{color}
\usepackage[title,titletoc,toc]{appendix}
\usepackage{bigints}

\allowdisplaybreaks

\newcommand{\be}{\begin{equation} }
\newcommand{\ee}{\end{equation}}
\newcommand{\bee}{\begin{equation*} }
\newcommand{\eee}{\end{equation*}}
\newcommand{\bse}{\begin{subequations}}
\newcommand{\ese}{\end{subequations}}

\newcommand{\p}{\partial}
\newcommand{\R}{\mathbb{R}}

\usepackage {caption}
\usepackage{xcolor}
\usepackage{cases}

\usepackage{amsmath}
\usepackage{amsthm}

\DeclareMathOperator*{\res}{Res}

\usepackage{todonotes}
\usepackage{float}
\usepackage{hyperref}

\newtheorem{theorem}{Theorem}[section]

\newtheorem{lemma}{Lemma}[section]
\newtheorem{Proposition}{Proposition}[section]
\newtheorem{RHP}{The RH problem}

\theoremstyle{remark}
\newtheorem{remark}{Remark}[section]

\theoremstyle{definition}

\numberwithin{equation}{section}

\title[Global  solution for the  mCH equation]
{Existence of global solutions  for the  modified Camassa-Holm equation with  a  nonzero background}

\author[]{Yiling YANG}
\address[Yiling Yang]{School of Mathematical Sciences  and Key Laboratory   for Nonlinear Science, Fudan   University, Shanghai 200433, P. R. China.}\email{19110180006@fudan.edu.cn}
\author[]{Engui Fan}
\address[Engui Fan]{School of Mathematical Sciences  and Key Laboratory   for Nonlinear Science, Fudan   University, Shanghai 200433, P. R. China.}\email{faneg@fudan.edu.cn}
\author[]{Yue Liu}
\address[Yue Liu]{Department of Mathematics, University of Texas at Arlington, TX 76019}\email{yliu@uta.edu}

\date{}                                           
\begin{document}
\thispagestyle{empty}

\begin{abstract}

Consideration in the present paper is the   existence of global solutions  for  the modified Camassa-Holm (mCH) equation with  a  nonzero background initial value.
The mCH  equation is completely integrable and can be considered as  a model for the unidirectional propagation of
shallow-water waves. By applying  the inverse scattering  transform with an application of the
 Cauchy projection operator,  the existence  of a unique global solution to the mCH equation in the line with  a  nonzero background initial value is established in the weighted Sobolev space $  H^{2, 1} (\mathbb{R})\cap H^{1, 2} (\mathbb{R})$ based on the representation of a  Riemann-Hilbert (RH) problem associated with
the Cauchy problem to the mCH equation. A crucial technique used is to  derive the boundedness of the solution in the Sobolev space $  W^{1,\infty}(\mathbb{R}),$
then reconstruct a new RH problem for the   Cauchy  projection operator of reflection coefficients. The regularity of the global solution  is achieved by the refined estimate arguments on those solutions of the corresponding RH problem. 

\end{abstract}

\maketitle

\noindent {\sl Keywords\/}:    Modified Camassa-Holm   equation;  global solution;  Riemann-Hilbert problem;     Cauchy projection operator.

\vskip 0.2cm

\noindent {\bf AMS Subject Classification 2020:} 35P25; 35Q15; 35Q35; 35A01.

\tableofcontents
\section{Introduction}
\quad
The present paper is concerned with   the existence  of global solutions  for   the modified Camassa-Holm (mCH) equation \cite {BF1996, PP1996} , namely,
\begin{align}
	&m_{t}+\left(m\left(u^{2}-u_{x}^{2}\right)\right)_{x}=0, \quad m=u-u_{xx}, \quad (t,x)\in\mathbb{R}^+\times\mathbb{R}, \label{mcho2}
\end{align} with  nonzero background $ u(t, x) \to 1 $ as $ |x| \to \infty. $
It is known that the mCH equation \eqref{mcho2}
is an integrable modification to   the well-known  Camassa-Holm (CH) equation
\begin{align} \label{ch}
	&m_t+ (um )_x+  u_x m=0, \quad m=u-u_{x x},
\end{align}	
which   was appeared in \cite{Fuchssteiner1} as a integrable equation proposed by   Fuchssteiner
and Fokas and first introduced by Camassa and Holm    as a
model for the unidirectional propagation of  shallow-water waves \cite{Holm1} (see also \cite{cola}  for a rigorous justification in shallow-water approximation).
The CH equation  \eqref{ch}  has  been  attracted considerable interest and been studied extensively  due to its  rich mathematical structures and
remarkable properties, such as  Painleve-type asymptotics, peakon and multi-peakon solutions,  bi-Hamiltonian structure,   algebro-geometric solutions,   wave-breaking phenomena   \cite{Monvel22, Holm1, Constantin1, C-Ep, C-E, Constantin2, Eckhardt0,
Fritz2, M,  mis}.

It is observed that all the nonlinear terms in the CH equation  \eqref{ch} are quadratic. 
Different from the CH equation \eqref{ch}, the mCH equation \eqref{mcho2} contains the cubic nonlinearity  which was introduced by Fuchssteiner \cite{BF1996}, Olver and
Rosenau \cite{PP1996} by applying the general method of tri-Hamiltonian duality to the bi-Hamiltonian
representation of the modified Korteweg-de Vries equation (see also \cite{Qiao}).  Applying the scaling transformation and taking parameter limit,  the  mCH  equation \eqref{mcho2} can be reduced  a   short pulse equation \cite{sp1}
\begin{align}
	&u_{xt}=u+\frac{1}{6}(u^3)_{xx}.\nonumber
\end{align}	
More recently, the mCH equation \eqref{mcho2} was considered as a model for the unidirectional propagation for
shallow-water waves of mild amplitude over a flat bottom \cite{CHL}, where the solution $ u $ is related to  the horizontal velocity in certain level of water.

It was shown that the mCH  equation \eqref{mcho2}
 admits a Lax pair \cite{Qiao, sch}
\begin{equation}
	\Phi_x = U \Phi,\hspace{0.5cm}\Phi_t =V \Phi, \label{lax0}
\end{equation}
where
\begin{align}
	&U=\frac{1}{2}\left(\begin{array}{cc}
		-1 &  {\lambda} m \\
		-\lambda  m & 1
	\end{array}\right), \ \ \lambda=- \frac{1}{2}(z+z^{-1}),\nonumber\\
	&V=\left(\begin{array}{cc}
		\lambda ^{-2}+\frac{1}{2}(u^2-u_x^2) & -\lambda ^{-1} (u-u_x)-\frac{ 1}{2}\lambda (u^2-u_x^2)m \\[5pt]
		\lambda ^{-1} (u+u_x)+\frac{ 1}{2}\lambda (u^2-u_x^2)m  & -\lambda ^{-2}-\frac{1}{2}(u^2-u_x^2)
	\end{array}\right),\nonumber
\end{align}
$\Phi $ is a complex matrix valued function of $x, t$ and $\lambda$  with the  spectral parameter  $\lambda\in \mathbb{C}$ independent of $x$ and $t$,  and the momentum potential $m$ is the
solution of the mCH  equation \eqref{mcho2}.


It is noted that the global smooth one-soliton solution of the mCH  equation \eqref{mcho2} with nonzero background data  was obtained by using the RH method \cite{Mon}.
On the other hand,  the soliton of the mCH equation \eqref{mcho2} with zero background data  is a weak solution in the form of peaked wave, that is, $ \varphi_c(x-ct) = \sqrt{\frac{3c}{2}} e^{-|x-ct|},  \; c > 0. $ In addition, the  quasi-periodic solutions with periodic background data were constructed by using algebro-geometric method \cite{THFan}.

The  wave-breaking and those peakons for the mCH equation \eqref{mcho2} with zero background data were also  investigated in  \cite{qu7}.
With the aid of reciprocal transformation, the B\"acklund transformation  and
nonlinear superposition formula  for the mCH equation \eqref{mcho2} were given \cite{WLM}.
The existence of the global  peakon solutions and  the large time asymptotic behaviour of these kind of non-smooth solitons  were investigated in \cite {CS2}.
Recently     the     long time   asymptotic behavior of the mCH equation  \eqref{mcho2}      was established  by using $\bar\partial$-steepest decadent  analysis in \cite{YYLmch}.

It is known that  the Cauchy problem associated with the mCH equation \eqref{mcho2} is the locally well-posed in the Sobolev space $ H^{s}(\mathbb{R}), s > 5/2 $ \cite{qu7}.
However, to the best of our knowledge, the global well-posedness to the    mCH equation \eqref{mcho2} on the line   seems  still  unknown.

The main purpose in the present paper is to   establish   the global existence of solutions to the  mCH equation \eqref{mcho2}  with a nonzero background
by applying inverse scattering theory.   It is found,  however,   that  directly using this approach to the mCH  equation \eqref{mcho2}
will  confront some substantial difficulties, which need to be overcome and are much  different from the NLS, AKNS, derivative NLS  equations   \cite{p3,Deift2011,RN10,LiuDNLS,pelinovskey1,zhou1}.
For example,
\begin{itemize}
	\item[1)] In the RH problem associated with the  mCH  equation \eqref{mcho2}, there are two  singularities  $z=\pm1$, which give rise to difficulty
	in directly  using Zhou vanishing lemma.  However, to remove the singularity at $z=\pm1$  by   change of variable,
	the   function $\eta=1-1/m$ will appear in the condition of the RH problem.
	The    boundness and space  estimate  of  $\eta$ must  be obtained to show that $m-1\in H^{2,1}(\mathbb{R})\cap H^{1, 2} (\mathbb{R})$.

	\item[2)]    In the estimates of the Volterra integral and the Cauchy projector operators,  there is an oscillating term $e^{i(z-1/z)x}$,  which
	is  difficult to directly apply the  Fourier transform in real axis $ \mathbb{R} $. This  technical  difficulty could be conquered by splitting
	the original  spectral problem   into two new spectral problems  by  two transformations.
	\item[3)]   In the jump matrix,  the reflection coefficient $|r(z)|=1, \ z=\pm1$, which
	 prevents from establishing the  Lipschitz continuous for solutions.  To deal with this issue,  we  can take an advantage of the
	nice properties of the time evolution for scattering data  to directly
	obtain estimates of the global solutions from  the time dependent RH problems without a prior estimate on the potential function $ m.$
\end{itemize}


In what follows,  we  address our  main result on   the existence  of global solutions for  the Cauchy problem  associated with  the mCH equation, namely,
\begin{align}\label{mch1}
\left\{ \begin{array}{ll} m_{t}+\left(m\left(u^{2}-u_{x}^{2}\right)\right)_{x}=0, \quad m=u-u_{xx},\\[4pt] m(0,x)=m_0(x),\\\end{array}\right.
\end{align}
where $ m_0 = (1- \partial_x^2) u_0 $  with $ u_0=(1-\partial_x^2)^{-1}m_0 = p\ast m_0,  $  $ p(x)=\frac {1}{2}\,e^{-|x|}, $
and the symbol `$\ast$' denotes the convolution product on $\R$, defined by
\begin{equation*}
(f\ast g)(x)=\int_{\mathbb R}f(y)g(x-y)dy.
\end{equation*}
The principal result   of the present paper  is now    stated  as follows.
\begin{theorem}\label{last}
  Assume that $ m_0(x)  > 0, \forall  x\in \mathbb{R} $ and $m_0 -1\in  H^{2,1}(\mathbb{R})\cap H^{1, 2} (\mathbb{R})$. Then
 there exists a unique global solution $ m$ of  the Cauchy  problem   \eqref{mch1}  such that $m -1\in C([0, +\infty);  H^{2,1}(\mathbb{R})\cap H^{1, 2} (\mathbb{R})). $
\end{theorem}

A crucial  approach for proving this  theorem    is   to  establish a   $L^2$-bijection map
between  solution $m$  and scattering data  by applying  the inverse scattering transform theory to the Cauchy problem \eqref{mch1},
see   Figure \ref{Figure1}.   Starting from the given  initial value $m_0 -1\in  H^{2,1}(\mathbb{R})\cap H^{1, 2} (\mathbb{R})$,
   the direct scattering transform
 gives rise to the scattering data and  a RH problem  associated with the Cauchy problem    \eqref{mch1}.
Then the inverse scattering transform  goes  back  to
  the global  solution $m(t) $  satisfying  $m -1 \in C ([0, +\infty);  H^{2,1}(\mathbb{R} )\cap H^{1, 2} (\mathbb{R})  )$
   to the   Cauchy problem   \eqref{mch1} via the  RH problem
with associated with time-dependent scattering data.


\begin{remark}  Note that if  $m_0(x)>0$,  $\forall x\in\mathbb{R}$, then  the solution $ m $ of \eqref{mch1} satisfies $\inf_{x\in\mathbb{R}} m(t, x)>0$, for $t \in [0, T)$  with any given $ T> 0 $ \cite{qu7}.
Using the blow-up criterion in Theorem 4.2 \cite{qu7} with the energy estimates, it follows from Theorem \ref{last} that for any $ s \ge 2, $  there exists a unique global solution $ m $ of the Cauchy problem  \eqref{mch1} satisfying  $ m -1 \in  C([0, +\infty);  H^{s,1}(\mathbb{R})\cap H^{s-1, 2} (\mathbb{R})), $ with the initial data $  m_0(x)  > 0, \forall  x\in \mathbb{R} $ and $m_0 -1\in  H^{s,1}(\mathbb{R})\cap H^{s-1, 2} (\mathbb{R}),  \, \forall s \ge 2. $
\end{remark}

\begin{figure}
\begin{center}
\begin{tikzpicture}
\node at (-4.0, 1.0 ){\fontsize{8pt}{8pt} $m_0 -1\in  H^{2,1}(\mathbb{R})\cap H^{1,2} (\mathbb{R}) $};
\draw [->] (-1.9,1.0)--(1.4,1.0);
\node at (3.8, 1.0 ){\fontsize{8pt}{8pt} $(r (0;z), \{z_j,c_j(0)\}_{j=1}^{2N_0})\in Y $};

\node at (-4.0, -1.0 ){\fontsize{8pt}{8pt} $m  -1 \in C([0,\infty),H^{2,1}(\mathbb{R})\cap H^{1,2} (\mathbb{R}) ) $};
\draw [<-] (-1.3,-1.0)--(1.4,-1.0);
\node at (3.8, -1.0 ){\fontsize{8pt}{8pt} $(r (t;z),\{z_j,c_j(t)\}_{j=1}^{2N_0})\in  Y $};

\draw [-] (2.9,0.8)--(2.9,-0.6);\draw [->] (2.9,-0.6)--(2.9,-0.8);
\draw [dashed] (-4,0.8)--(-4,-0.6);
 \draw [->] (-4,-0.6)--(-4,-0.8);
\node at (4.1, 0 ){\footnotesize Time \ evolution};
 \node at (-0.4, 1.3 ){\footnotesize Direct scattering};
 \node at (0, -0.7){\footnotesize Inverse scattering};
\end{tikzpicture}
 \caption{\footnotesize   Bijective relation between potential $m$ and reflection coefficients }
 \label{Figure1}
\end{center}
 \end{figure}
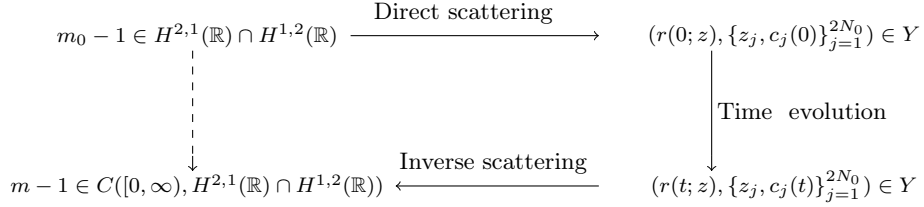

The rest of the paper is organized as follows.   In Section \ref{sec2}, we  give construction of the basic RH problems associated with the Cauchy problem \eqref{mch1} and establish
a direct inverse transform from the   initial value $m_0$ to the scattering data at $t=0$.
In Section \ref{sec3},  we carry out the inverse scattering transform from  time-dependent scattering data and solvability of the regular RH problem.
In Section \ref{sec6},
we show the existence  of   a unique  global  solution  $ m -1 \in C ([0, +\infty);  W^{1,\infty}(\mathbb{R} )) $  to the  Cauchy problem   \eqref{mch1}.
In the last section,  Section \ref{sec7},   we reconstruct   a new RH problem   to arrive at the regularity of the solution $m-1$.
Finally,  we demonstrate the existence of a unique  global  solution  satisfying   $m  -1 \in C ([0, +\infty);  H^{2,1}(\mathbb{R} )\cap H^{1, 2} (\mathbb{R})  )$
  via  the new RH problem  and refined   estimates on the solution  $m$.
\vskip 0.1cm
\noindent {\bf Notations.}
Throughout, the norm in the classical space $ L^p (\mathbb{R})$ for any $ 1 \le p \le \infty $  will be written as $ \|\cdot \|_{L^p}. $
We use the normed spaces defined in the following.
 A weighted $L^p(\mathbb{R})$ space for any $ 1 \le p \le \infty $ is specified by
	$$L^{p,s}(\mathbb{R})  =  \left\lbrace f \in L^p(\mathbb{R})\hspace{0.2cm} |  \  \hspace{0.1cm} \langle \cdot\rangle^sf\in L^p(\mathbb{R}) \right\rbrace, \quad s \ge 0 $$ with $ \langle x \rangle = (1 + x^2)^{\frac{1}{2}} $ equipped with the norm
$$
\|f\|_{p, s} =  \| \langle \cdot\rangle^s f \|_{L^p}. \quad
$$
The Sobolev space $ W^{k,p} $ is defined by
	$$W^{k,p}(\mathbb{R})  =  \left\lbrace f\in L^p(\mathbb{R}) \hspace{0.2cm}| \  \hspace{0.1cm} \partial_x^jf \in L^p(\mathbb{R}), \; j=1,2,...,k \right\rbrace, $$ equipped with the norm
$$
\|f\|_{W^{k, p}} = \left(\sum_{j =0}^{k} \| \partial_x^jf \|^p_{L^p}\right)^{1/p},
$$
and
a weighted Sobolev space $H^{k,s}$ is defined by
	$$H^{k,s}(\mathbb{R})   =   \left\lbrace f\in L^2(\mathbb{R})  \hspace{0.2cm}| \  \hspace{0.1cm} \langle \cdot\rangle^{s}\partial_x^jf  \in L^2(\mathbb{R}),  \; j=0,1,...,k \right\rbrace,  \; s \ge 0$$ equipped with the norm
$$
\|f\|_{H^{k, s}}  =   \sum_{j =0}^{k} \| \langle \cdot\rangle^{s} \partial_x^j f \|_{L^2},  \; s \ge 0.
$$
And  the Fourier transform is normalized  by
\begin{align*}
	\hat{f}(z)=(\mathcal{F}f)(z)=\int_{\mathbb{R}}f(s)e^{-2\pi isz}ds,\ 	\check{g}(z)=(\mathcal{F}^{-1}g)(z)=\int_{\mathbb{R}}g(s)e^{2\pi isz}ds.
\end{align*}
Let $I$ be an interval on the real line $\mathbb{R}$ and $X$ be a  Banach space. We denote  $C_B(I;X)$  the space of bounded continuous functions on $I$ taking values in $X$  equipped with the norm
\begin{equation*}
	\|f\|_{C_B(I; X)}=\sup _{x \in I}\|f(\cdot, x)\|_{X}.
\end{equation*}
Define a Banach space
\begin{align}
Y =\left\{\left( r,\{z_j,c_j\}_{j=1}^{2{N_0}}\right) \in \left( H^{1,2}(\mathbb{R})\cap H^{2,1}(\mathbb{R})\right) \otimes \mathbb{C}^{2{N_0}}\otimes \mathbb{C}^{2{N_0}}  \right\}, \label{banch}
\end{align}
  equipped with the norm
\begin{align}
	\Big\| \left( r,\{z_j,c_j\}_{j=1}^{2{N_0}}\right)\Big\|_{Y}=\|r\|_{H^{1,2}\cap H^{2,1}}+\sum_{j=1}^{2{N_0}}|z_j|+\sum_{j=1}^{2{N_0}}|c_j|.\label{normD}
\end{align}
In addition, for any vector or matrix $A$, the superscript
$A^T$ denotes the transpose of $A$.  Denote  $f^*(z)=\overline{f(\bar{z})}$.
Finally,  the letter $C$ will be used to denote universal positive constants which may vary from line to line. We also use the
notation $A\lesssim B$ to denote the bound of the form $A \leqslant CB$. To emphasize the implied constant
to depend  on some parameter $\alpha$, we shall indicate this by $C(\alpha)$.


\section {Direct scattering transform  }\label{sec2}

Our attention in this section is  to consider the  direct scattering transform for  the spatial spectral problem in the Lax pair  \eqref{lax0} as $t=0$ with the initial value $ m_0 - 1 \in H^{1,2} \cap H^{2,1} $.
The variable $t$ is omitted as usual, for example, $\Phi(z;0,x)$ is just written as $\Phi(z;x)$.
\subsection{The Jost function}\label{subs2.1}

Following the approach on  the scheme of construction of  a basic  RH problem
\cite { p2,    Mon, CL, CL2,pelinovskey1,  YYLmch, ZHOU},
we define  a   transformation for $m_0-1\in L^1(\mathbb{R})$
\begin{align}
		\mu(z;x)=D(z)\Phi(z;x)e^{  p(z;x)\sigma_3}
,\label{transmu}
\end{align}
where
\begin{align}
	&D(z)=\left(\begin{array}{cc}
		1 & -\frac{2\lambda}{2+  ik}\\
		-\frac{2\lambda}{2+ ik} & 1
	\end{array}\right), \ \ k= z-z^{-1},\ \lambda=- \frac{1}{2}(z+z^{-1}),\label{D}\\
&	p(z;x)= \frac{ik }{4}( x+ \int_{+\infty}^{x}(m_0(s)-1)ds). \nonumber
\end{align}
Then the Jost function $\mu(z;x) $ satisfies the following spectral problem
\begin{align}
	& \mu_{x}(z;x) = - \frac{ ik}{4} m_0 [\sigma_3,\mu(z;x)]+P\mu(z;x),\label{lax1.1} \ 	P=\dfrac{m_0-1}{ ik}\left(\begin{array}{cc}
		1 & \lambda \\
		-\lambda & -1
	\end{array}\right),
\end{align}
where the Lie bracket $ [A, B] $ is defined by $ [A, B] = AB -BA. $
We now denote the matrix in column by
$\mu^\pm(z;x)=\left(  \mu^{\pm }_1(z;x),  \mu^{\pm }_2(z;x)\right), $
where  the subscripts  $1$ and $ 2$ denote
the first and second columns of $\mu^\pm(z;x)$, respectively. Then \eqref{lax1.1} leads to   two  Volterra type integrals
\begin{align}
	&\mu^\pm_1(z;x)=e_1+\int^{x}_{\pm \infty}\text{diag} ( 1,
	e^{  2(p(z;s)-p(z;x)) }) P(z;s)\mu^\pm_1(z;s) ds,\label{intmu}\\
	&\mu^\pm_2(z;x)=e_2+\int^{x}_{\pm \infty}\text{diag} (
		e^{  -2(p(z;s)-p(z;x)) }, 1) P(z;s)\mu_2^\pm(z;s)  ds\label{intmu2},
\end{align}
where $e_1=(1,0)^T$, $e_2=(0,1)^T$. Then the existence of $\mu^\pm$ and their analytic properties as functions
of $z$ are the subject of the following proposition.
\begin{Proposition}\label{pro1.1}
	Let $m_0-1\in L^{1}(\mathbb{R})$. Then there  exists unique solution $\mu^\pm(z;\cdot)\in L^\infty(\mathbb{R})$ for every  $z\in\mathbb{R}\setminus\{ \pm1\} $. Moreover,
	for  every $x\in\mathbb{R}$,  $z\in\mathbb{C}^+$, $\mu^-_1 (z;x)$ and $\mu^+_2(z;x)$ extend analytically into solutions of \eqref{intmu} and \eqref{intmu2}  respectively, while 	for  every $x\in\mathbb{R}$,   $z\in\mathbb{C}^-$, $\mu^+_1(z;x)$ and $\mu^-_2(z;x)$ extend analytically into solutions of \eqref{intmu} and \eqref{intmu2}  respectively.
\end{Proposition}
\begin{proof}
It suffices to prove the statement by taking $\mu^-_2$ as example. Rewrite \eqref{intmu2} as
	\begin{align*}
	\mu^\pm_2(z;x)=e_2+\mathcal{T}\mu^\pm_2(z;x),
	\end{align*}
where $\mathcal{T}$ is  a linear integral  operator defined by
\begin{align}
	\mathcal{T} f (z;x)&= \int^{x}_{+\infty}K(z;x,s)f(z;s)ds, \label{wef}
\end{align}
for any $f(z;\cdot)\in L^\infty(\mathbb{R})$ with  integral kernel $K(z;x,s)={\rm diag} (
e^{\frac{i}{2}(z-1/z)\int^{s}_{x}m_0(l)dl}, 1) P(z;s).$ Then $\mathcal{T}$ is a bounded integral  operator on $L^\infty$ for every $z\in\overline{\mathbb{C}^-}\setminus\{\pm1\}$ with
\begin{align*}
	\|\mathcal{T}\|\leq C(z)\|m_0-1\|_{L^1}.
\end{align*}
For every $z\in\overline{\mathbb{C}^-}\setminus\{\pm1\}$,  any $n=2,3,...$, $f(z;\cdot)\in L^\infty(\mathbb{R})$
	\begin{align}
	\mathcal{T}^n f(z;x)&=\int^{x}_{+\infty}\int^{y_1}_{+\infty}\dots \int^{y_{n-1}}_{+\infty}K(z;x,y_1)\dots K(z;x,y_n)f(z;y_n)dy_n\dots dy_1\nonumber\\
	&:=\int^{x}_{+\infty}K_n(z;x,y_n)f(z;y_n)dy_n,\label{Tn}
\end{align}
with
\begin{align*}
	|K(z;x,y_n)|\lesssim\frac{ \left( \int_{y_n}^{+\infty} |m_0(\zeta)-1|d\zeta\right) ^{n-1}}{(n-1)!} |m_0({y_n})-1|.
\end{align*}
Then $\mathcal{T}^n$ admits
\begin{align*}
		\|\mathcal{T}^n\|\leq C(z)\frac{\|m_0-1\|_{L^1}^n}{n!}.
\end{align*}
Therefore, the series $\sum_{n =1}^\infty\mathcal{T}^n$  is absolute convergence such that $	\mu^\pm_2(z;\cdot)$ exists in $L^\infty$ with
\begin{align*}
		\mu^\pm_2(z;x)=\sum_{n =0}^\infty\mathcal{T}^ne_2(z;x),
\end{align*}
for every $z\in\overline{\mathbb{C}^-}\setminus\{\pm1\}$. Analyticity of  $	\mu^\pm_2(x; \cdot)$ in $\mathbb{C}^-$ for every $x \in\mathbb{R} $ follows from the absolute and uniform convergence of the Neumann series of analytic functions in $z$. This completes the proof of Proposition \ref{pro1.1}.
\end{proof}
\begin{remark}
	Similar to Lemma 1 in \cite{pelinovskey1}, although the estimate in \eqref{normT} requires  $\parallel m_0-1\parallel_{L^1}<1$, real axis $\mathbb{R}$ can be split  into a finite number of subintervals $I_n$ such that $\int_{I_n}|m_0(s)-1|ds< 1$ on each subinterval. Thus, through that $\mathcal{T}$ is  a linear integral  operator, Proposition \ref{lemma1} can be accomplished without any restriction on  $\|m_0-1\|_{L^1}<1$.
\end{remark}
Applying \eqref{intmu} and \eqref{intmu2} also implies that as $|z|\to\infty$,
$$	\mu^\pm(z;x)=I+\mathcal{O}(1/z),$$
with 	\begin{align}
	\lim_{z\to \infty}z \mu^\pm_{12}(z)=\eta,\ 	\eta=1-1/m_0. \label{asymuinf}
\end{align}
On the other hand, the Jost functions exhibit singular behavior for $z$ near $\pm1$.  The singular behavior of these solutions at $z = \pm1$ plays a non-trivial
and unavoidable role in our analysis.
\begin{Proposition}\label{promu1}
	Let $m_0-1\in L^{1,1}(\mathbb{R})$. Then  $	\mu^\pm$  have singularities on 	$z=\pm1$ with a special structure
	\begin{align*}
		&\mu^\pm(z;x)=\frac{i\alpha_\pm(x)}{2(z-1)}\left(\begin{array}{cc}
			-1 & 1 \\
			-1 & 1
		\end{array}\right)+\mathcal{O}(1),  \quad \text{as } z\to 1,\\
		&\mu^\pm(z;x)=\frac{i\alpha_\pm(x)}{2(z+1)}\left(\begin{array}{cc}
			-1 & -1 \\
			1 & 1
		\end{array}\right)+\mathcal{O}(1),  \quad \text{as }z\to -1,
	\end{align*}
	where $\alpha_\pm(x)\in\mathbb{R}$.
\end{Proposition}
\begin{proof}
	 We take $z$ near $1$ as an example.  The property of $z$ near $-1$ can be obtained from the symmetry immediately.
	Denote
	\begin{align}
		\upsilon^\pm=(z^2-1)D^{-1}\mu^\pm\label{dev}
	\end{align}
	with $D$ defined in \eqref{D}. Then  the integral representations of $\mu^\pm$ in \eqref{intmu} and \eqref{intmu2} give that the two columns of $\upsilon^\pm$ admit the following integral equations,
	\begin{align}
		&\upsilon^\pm_1(z;x)=\left(\begin{array}{cc}
			(z-i)^2/2	   \\[4pt]
			i(z^2+1)/2
		\end{array}\right)+\int^{x}_{\pm \infty} \frac{m_0(s)-1}{2}\Gamma (z,x,s)\upsilon^\pm_1(z;s)ds,\label{v1}\\
		&\upsilon^\pm_2(z;x)=\left(\begin{array}{cc}
			i(z^2+1)/2  \\[4pt]
			(z-i)^2/2	
		\end{array}\right)+\int^{x}_{\pm \infty} \frac{(m_0(s)-1)\gamma^2(z,x,s)}{2}\Gamma (z,x,s)\upsilon^\pm_2(z;s)ds\label{v2},
	\end{align}
	where  $\Gamma$ is a bounded  analytic function with respect to $z$ in a neighbourhood of $z=\pm1$,
	\begin{align}
		\Gamma (z,x,s)&=D^{-1}(z)\left(\begin{array}{cc}
			1 & 0\\
			0 & \gamma(z,x,s)
		\end{array}\right)D(z)\sigma_3\\
		&=\frac{(z-i)^2}{2(z^2-1)}\left(\begin{array}{cc}
			1-\gamma(z,x,s)d(z)^2 & -d(z)(1-\gamma(z,x,s))\\
			-d(z)(1-\gamma(z,x,s)) & d(z)^2 -\gamma(z,x,s)
		\end{array}\right),
	\end{align}
	and
	\begin{align}
		\gamma(z,x,s)=e^{\frac{1}{2}(z-1/z)\int^s_xm_0(\eta)d\eta},\ d(z)=\frac{z^2+1}{i(z-i)^2}.
	\end{align}
	
	The proof also  resembles  the analysis of the Jost function in \cite{YYLmch} by taking $\upsilon^+_1$ as an example. For any matrix function $f(z;x)\in L^\infty((1/2,2)\times\mathbb{R}^+)$, define a integral operator $\tilde{\mathcal{T}}$
	\begin{align*}
		\tilde{\mathcal{T}}	f(z;x)=\int^{x}_{+ \infty} \frac{m_0(s)-1}{2}\Gamma (z,x,s)\sigma_3f(z;s) ds.
	\end{align*}
	Noting that
	\begin{align*}
		\Big\vert	\frac{1}{z^2-1}\left( 1-\gamma(z,x,s)\right)\Big\vert =\Big\vert\int_{x}^{s}\frac{e^{\frac{1}{2}(z-1/z)\int^\xi_xm_0(\eta)d\eta}}{m_0(\xi)}d\xi\Big\vert\lesssim|s-x|,
	\end{align*}
	it is ascertained that
	$$\|(m_0(\cdot)-1)\Gamma(z;\cdot)\|_{L^1}\leq \|m_0-1\|_{L^1}+2\|m_0-1\|_{L^1}^2+4\|m_0-1\|_{L^{1,1}}.$$
	Then the similar process as one in  Proposition \ref{lemma1} implies that  $(I-\tilde{\mathcal{T}})^{-1}$ exists with $\|(I-\tilde{\mathcal{T}})^{-1}\|\leq \exp\{\|m_0-1\|_{L^1}+2\|m_0-1\|_{L^1}^2+4\|m_0-1\|_{L^{1,1}}\}$. This in turn implies that
	\begin{align*}
		\upsilon^+_1=(I-\tilde{\mathcal{T}})^{-1}(e_1(z-i)^2/2+e_2i(z^2+1)/2).
	\end{align*}
	
	Furthermore, when  $z=1$,
	combining \eqref{dev}, $D(1)=\left(\begin{array}{cc}
		1 & 1\\
		1 & 1
	\end{array}\right)$, and $\det D(z)=-\frac{z^2-2z-1}{2(z^2-1)}$, it then follows  that
	\begin{align*}
		\mu^\pm_1(z)=\frac{(\upsilon^\pm )_{11}(1)+(\upsilon^\pm)_{21}(1)}{2(z-1)}\left(\begin{array}{cc}
			1 \\
			1
		\end{array}\right)+\mathcal{O}(1).
	\end{align*}
	Together with the symmetry \eqref{symPhi1}, it is concluded that through denoting $\alpha_\pm=i(\upsilon^\pm )_{11}(1)+i(\upsilon^\pm)_{21}(1)$, we obtain the result of Proposition \ref{promu1}.
\end{proof}
In view of tr$U=0$ in the spectral problem \eqref {lax0}, it follows from  Abel's formula that the determinant  $\det(\Phi^\pm(z;x))=\det(\mu^\pm(z;x)) =1$.
And from the symmetry of the function $p$ and the matrix function $P$, the Jost functions 	$ \mu^\pm (z;x)$ admit  three kinds of  reduction conditions
		\begin{equation}
			\mu^\pm(z;x)=\sigma_1\overline{\mu^\pm(\bar{z};x)}\sigma_1=\sigma_2\mu^\pm(-z;x)\sigma_2=\sigma_1\mu^\pm(z^{-1};x)\sigma_1.\label{symPhi1}
		\end{equation}		
On the other hand, $	\mu^\pm(z;x)$ satisfy
a same special asymptotic as $z\to i$.
Based on the approach in \cite{Mon}, if $m_0-1\in L^1(\mathbb{R})$, then
 $	\mu^\pm(z;x)$ have the following expansions as $z\to i$,
\begin{align}
	&\mu^\pm(z;x)=e^{\frac{1}{2}\int_x^{\pm\infty}m_0(s)-1ds\sigma_3}\nonumber\\
	&+(z-i)e^{-\frac{1}{2}\int_x^{\pm\infty}m_0(s)-1ds\sigma_3}\left(\begin{array}{cc}
	0 & -\frac{u_0(x)+\partial_xu_0(x)}{2} \\
	\frac{\partial_xu_0(x)-u_0(x)}{2} & 0
\end{array}\right)+\mathcal{O}\left( (z-i)^2\right). \label{asyM}
\end{align}

\subsection{Scattering map} \label{map}
In this  subsection,   we will derive  basic properties on the scattering data
of the Lax pair  for the  general  initial data $ m_0 $  in the weighted Sobolev spaces with  $m_0 -1\in  H^{2,1}(\mathbb{R})\cap H^{1, 2} (\mathbb{R})$.

The Jost functions $\mu^\pm(z;x)$  satisfy a scattering  relation on   $\mathbb{R}$ by
\begin{equation}
	\mu^+(z;x)=\mu^-(z;x)e^{  p(z;x)\widehat{\sigma}_3}S(z), \label{scattering}
\end{equation}
where $S(z)$ is the   scattering matrix independent of $x$ and
\begin{equation}
	S(z) =\left(\begin{array}{cc}
		a(z) &b(z)   \\[4pt]
		\overline{b(z)} & \overline{a(z)}
	\end{array}\right),\hspace{0.5cm}\det (S(z)) =|a(z)|^2-|b(z)|^2=1. \nonumber
\end{equation}
In addition,   the relation formula \eqref{scattering} also gives that
\begin{align}
	a(z)=\det(\mu_1^+,\mu_2^-),\ b(z)=e^{-2p}\det(\mu_2^+,\mu_2^-).\label{deta}
\end{align}
This thus implies that  $a(z)$ is  analytical on $\mathbb{C}^-$, and continuous up to $ \mathbb{R} $  except at $z=\pm1$. And $|z|\to\infty$,  $a(z)\to 1$.
And $b(z)$ is continuous on $\mathbb{R}$ except at $z=\pm1$, where $z=\pm1$ are the only poles of $a(z)$ and $b(z)$ on $\mathbb{R}$. Specially, at $z=\pm 1$, in the generic case, $a(z)$ and $b(z)$ have the same order singularities. While in the non-generic case, since $a(z)$ and $b(z)$ are both continuous at $z=\pm 1$, which can be dealt with in a standard approach \cite {pelinovskey1}, we only consider the generic case  in the present  paper.
\vskip 0.1cm
\noindent \textbf{The discrete spectrum.}\quad
The function $a(z)$ may have zeros on $\mathbb{C}^-$. A point  $z_0\in\mathbb{C}^-$ satisfies  $a(z_0)=0$ if and only if  $z_0$ is an eigenvalue of the spatial part of the Lax pair  \eqref{lax0}. This $z_0$ is called the discrete spectrum. In addition, the zeros of $a(z)$ are important component of scattering data.  The following proposition demonstrates
that the  zeros of $a(z)$ take
a very restricted form.
\begin{Proposition} \label {prop2.2}
	The zeros of $a(z)$ in $\mathbb{C}^-$ are simple and finite, and they are restricted to	the unit circle.
\end{Proposition}
\begin{proof}
	
	Since $a(z)$ is analytical  and $a(z)\to 1$ as $|z|\to\infty$,  then
 $a(z)$  has only a finite number of zeros on $\mathbb{C}^-$. Rewrite   the spatial part in  \eqref{lax0} as
	\begin{align*}
		-\frac{i}{m_0}\sigma_2( 2\partial_x+\sigma_3) \Phi=\lambda\Phi,
	\end{align*}
	which is  self-adjoint. So if $z_0$ is a zero of $a(z)$, the eigenvalue $\lambda(z_0)$ must be real. Together with $|a(z)|\geq1$ for $z\in\mathbb{R}$, the zero $z_0$ must locate on the unit circle in $\mathcal{C}^-$. Moreover, it is observed from  \eqref{asyM} and \eqref{scattering} that
	\begin{align}
		a(-i)=\exp\{\frac{1}{2}\int_{\mathbb{R}}(m_0-1)ds\}\neq0.\label{ai}
	\end{align}
  Thus $z_0\neq -i$.
	It is given by
	\eqref{scattering}  that
	\begin{align}
		a(z)=\det (D) \det(\Phi^+_{1},\Phi^-_{2}).\label{a1}
	\end{align}
	So there exists a constant $b_0$ such that
	\begin{align}
		\Phi^+_1(z_0)=b_0\Phi^-_2(z_0).\label{b0}
	\end{align}
	Differentiating \eqref{a1} with respect to $z$ at $z_0$, it appears that
	\begin{align*}
		\partial_za(z_0)=\det(D(z_0))( \det(\partial_z\Phi^+_{1},\Phi^-_{2})(z_0)+\det(\Phi^+_{1},\partial_z\Phi^-_{2})(z_0))\triangleq A_1+A_2.
	\end{align*}
	It is inferred from \eqref{lax0} that
	\begin{align*}
		\partial_xA_1=\det (D(z_0)) \det(\partial_zU\cdot\Phi^+_{1},\Phi^-_{2})(z_0),\hspace*{0.3cm}\partial_xA_2=\det (D(z_0)) \det(\Phi^+_{1},\partial_zU\cdot\Phi^-_{2})(z_0),
	\end{align*}
	where $\partial_zU=-\frac{im_0}{4}(1-\frac{1}{z^2})\sigma_2.$
	Together with \eqref{b0}, it is deduced that
	\begin{align*}
		&A_1=\det (D(z_0))(1-z_0^{-2})b_0\int_x^{+\infty}\frac{m_0}{4}((\Phi^-_{22})^2+(\Phi^-_{12})^2)dy,\\&A_2=\det (D(z_0)) (1-z_0^{-2})b_0\int^x_{-\infty}\frac{m_0}{4}((\Phi^-_{22})^2+(\Phi^-_{12})^2)dy.
	\end{align*}
	In the above expression, $D(z_0)$ is a real matrix. Via the symmetry of $\Psi$, it is ascertained that $\Psi(z_0),\ b_0\in\mathbb{R}$. It is thereby inferred that $\Phi(z_0)\in\mathbb{R}$. Hence,
	\begin{align*}
		\partial_za(z_0)=\det(D(z_0))(1-\frac{1}{z_0^2})b_0\int_{\mathbb{R}}\frac{m_0}{4}((\Phi^-_{22})^2+(\Phi^-_{12})^2)dy\neq0.
	\end{align*}
	Simplicity of the zeros of $a$ follows immediately. This completes the proof of Proposition \ref{prop2.2}.
\end{proof}
The zeros of $a(z)$ in the fourth quadrant  are denoted as $z_1,...,z_{N_0}$.
Thus we define
$Z_0=\{z_j,-\bar{z}_j\}_{j=1}^{N_0}$ as the set of ${2N_0}$ zeros of $a(z)$ on $\mathbb{C}^-$. For convenience, we denote $z_j=-\bar{z}_{j-N_0}$, $j={N_0}+1,...,2{N_0}$. Then $Z_0=\{z_j\}_{j=1}^{2N_0}$.

\vskip 0.1cm
\noindent \textbf{The reflection coefficient.}\quad
We now define the reflection coefficient  by
\begin{equation}
	r(z)=\dfrac{b(z)}{a(z)},\ z\in\mathbb{R}.\label{symr}
\end{equation}
From  its definition, $r(z)$ is continuous on $\mathbb{R}$ with $r(1)=-r(-1)=-1$. And $r(z)$ admits symmetries
\begin{align}
	r(z)=\overline{r(z^{-1})}=-\overline{r(-z)}.\label{syr}
\end{align}
This in turn implies from det$(S(z))=1$ that
\begin{align*}
	|a(z)|^{-2}+|r(z)|^2=1,\ z\in\mathbb{R}.
\end{align*}
So $|r(z)|^2,\  |a(z)|^{-1}\leqslant 1$. Furthermore, $|a(z)|^{-2} =1-|r(z)|^2\neq 0 $  for $z\neq \pm1$.

We
analyze  the  boundedness and integrability of the reflection coefficient $r(z)$ when $t=0$, which is one of determining factors for  the RH problem.
The next  lemma gives a property of  the direct scattering map: $m_0\to r$.
\begin{lemma}\label{r1}
If  initial data  $m_0-1\in  H^{2,1}(\mathbb{R} )\cap H^{1, 2} (\mathbb{R}) $,   then the reflection coefficient $r\in  H^{2,1}(\mathbb{R} )\cap H^{1, 2} (\mathbb{R}) $.
\end{lemma}
The proof of Lemma \ref{r1} is approached via a series of  propositions  and follows a similar method to the Section 3 in \cite{YYLmch}.
By the symmetry of $r$ in \eqref{syr}, it is adduced that
\begin{align*}
	&	\int_{\mathbb{R}}|r(z)|^2dz=\int_{\mathbb{R}_-}|r(z)|^2dz+\int_{\mathbb{R}_+}|r(z)|^2dz=2\int_{\mathbb{R}_+}|r(z)|^2dz\\
	&=2\int_{0}^1|r(z)|^2dz+2\int_1^{+\infty}|r(z)|^2dz=2\int_1^{+\infty}|r(z)|^2(1+\frac{1}{z^2} ) dz\leqslant 4\int_1^{+\infty}|r(z)|^2dz.
\end{align*}
Thus  it suffices to consider the property of $r(z)$ at the interval $[1,+\infty)$.
We divide the interval $[1,+\infty)$ into two parts  $[1,C]$ and $(C,\infty)$ for any constant $C>1$. And we first
 analyze the property of $r(z)$, namely, the function $\mu^\pm(z;x)$ at the interval $z\in(C,\infty)$ for any constant $C>1$. To simplify notation,  without loss of generality, we take $C=2$. In the following demonstration, our attention is now focused on three function spaces $C _B\left(\mathbb{R}^\pm; L^{2}(2,\infty)  \right)  $,  $ L^{\infty}\left( (2,\infty)  \times\mathbb{R}^\pm \right) $ and $ L^{2}\left( (2,\infty)  \times\mathbb{R}^\pm \right) $ with the norms $
 \parallel \cdot\parallel_{C_B}$,$ \parallel \cdot\parallel_{L^\infty}$  and $ \parallel \cdot\parallel_{L^2}$ respectively. We start with estimates obtained with the Fourier theory.
 \begin{Proposition}\label{proest}
 	For any $f\in L^{2,1/2}(\mathbb{R})$,
 	\begin{align}
 	&\big\|\int_x^{+\infty}f(s)e^{2ik\int^s_xm_0(l)dl}ds	\big\|_{C_B}\lesssim \|f\|_{L^{2}}\label{estf1}\\
 	&\big\|\int_x^{+\infty}f(s)e^{2ik\int^s_xm_0(l)dl}ds	\big\|_{L^2\left( (C,\infty)  \times\mathbb{R}^\pm \right)}\lesssim \|f\|_{L^{2,1/2}}\label{estf2}.
 	\end{align}
 \end{Proposition}
\begin{proof}
	For any $g\in L^2$, it transpires that
	\begin{align*}
		&\big|\int_\mathbb{R}\int_x^{+\infty}X_{(C,+\infty)}(k)g(k)f(s)e^{2ik\int^s_xm_0(l)dl}ds	dk\big|=\big|\int_x^{+\infty}\int_\mathbb{R}X_{(C,+\infty)}(k)g(k)f(s)e^{2ik\int^s_xm_0(l)dl}	dkds\big|\\
		&=\big|\int_x^{+\infty}f(s)(X_{(C,+\infty)}g)^{\vee}(\frac{\int^s_xm_0(l)dl}{\pi})ds\big|\leqslant ( \int_x^{+\infty}|f(s)|^2ds) ^{1/2}\|X_{(C,+\infty)}g\|_{L^2},
	\end{align*}
where $C$ is a real constant.
where $X_{(C,+\infty)}(k)$ is the truncation function on interval $(C,+\infty)$. Then \eqref{estf1} follows directly. It also gives that
\begin{align*}
	\int_{\mathbb{R}^+} \int_{C}^{+\infty}\big|\int_\mathbb{R}\int_x^{+\infty}f(s)e^{2ik\int^s_xm_0(l)dl}ds	\big|dkdx\leq	\int_{\mathbb{R}^+}\int_x^{+\infty}|f(s)|^2dsdx\lesssim\|f\|_{L^{2,1/2}}^2.
\end{align*}
Thus we complete the proof of Proposition \ref{proest}.
\end{proof}
\begin{Proposition}\label{lemma1}
	If  initial data  $m_0-1\in  H^{2,1}(\mathbb{R} )\cap H^{1, 2} (\mathbb{R}) $, then there exists a unique the Jost function
	$\mu^\pm(z;x)$ with  $\mu^\pm(z;x)-I\in C_B \left(\mathbb{R}^\pm; L^{2}(2,\infty) \right) \cap L^{\infty}\left( (2,\infty) \times\mathbb{R}^\pm\right)\cap L^{2}\left( (2,\infty) \times\mathbb{R}^\pm\right) $ satisfying
	\begin{align*}
		&\parallel \mu^\pm(z;x)-I \parallel_{L^\infty\cap C_B }\leqslant \frac{C(\|m_0-1\|_{H^{2,1}\cap H^{1,2}})}{\underset{x\in\mathbb{R}}{\inf}m_0} ,\\
		&	\parallel \mu^\pm(z;x)-I \parallel_{L^2}\leqslant C(\|m_0-1\|_{H^{2,1}\cap H^{1,2}}) .
	\end{align*}
\end{Proposition}
\begin{proof}
	The proof basically follows the approach in Proposition \ref{pro1.1}. We prove this proposition only  by taking the second column  $\mu_{ 2}^+(z;x)$ as an example.
	Denote
	$$\mu_{ 2}^+(z;x)-e_2 \triangleq h(z;x)=(h_1(z;x),h_2(z;x))^T.$$
	The second column of the Volterra  integral equation \eqref{intmu2}   can be written as   an   operator equation
	\begin{align}
		h(z;x)=h^{(0)}(z;x)+ \mathcal{T} h (z;x),\label{dn}
	\end{align}
	where  $\mathcal{T}$ is defined in \eqref{wef}.
	And
	\begin{align*}
		h^{(0)}(z;x)&= Te_2  =\int^{x}_{+\infty}i\dfrac{m_0(s)-1}{z^2-1}\left(\begin{array}{cc}
			\frac{z^2+1}{2}e^{\frac{i}{2}(z-1/z)\int^{s}_{x}m_0(l)dl}	   \\[4pt]
			z
		\end{array}\right)ds\triangleq (h_1^{(0)}(z;x),h_2^{(0)}(z;x))^T.
	\end{align*}
	It is obvious that $h^{(0)} \in L^{\infty} ( (2,\infty)\times\mathbb{R}^+)$.
	Considering  the first component  of $h^{(0)}(z;x)$ and integrating  by parts yields
	\begin{align}
		h^{(0)}_1(z;x)=&\dfrac{(z^2+1)z}{(z^2-1)^2}\frac{m_0-1}{m_0}  - \dfrac{(z^2+1)z}{(z^2-1)^2}\int^{x}_{+\infty}	e^{\frac{i}{2}(z-1/z)\int^{s}_{x}m_0(l)dl}	\partial_s\left( \frac{m_0-1}{m_0}\right) ds,
	\end{align}
	which implies that
	\begin{align*}
		\sup_{x\in\mathbb{R}^+}\int_{\mathbb{R}}|h^{(0)}_1(z;x)|^2dz\leqslant  \frac{\|m_0-1\|_{H^{2,1}}}{\underset{x\in\mathbb{R}}{\inf}m_0}.
	\end{align*}
	Obviously,  $h^{(0)}_2$ has the same estimate.
	Thus,
	$$h^{(0)}(z;x) \in  C_B\left(\mathbb{R}^+;  L^{2}(2,\infty) \right) \cap L^{\infty}\left( (2,\infty)\times\mathbb{R}^+\right),$$
	with $\|h^{(0)}\|_{L^\infty\cap C_B }\leqslant \frac{2\|m_0-1\|_{H^{2,1}}}{\underset{x\in\mathbb{R}}{\inf}m_0} $. Recall the change of variable
	$k=\frac{1}{4}\left(z-\frac{1}{z} \right) $. When $z\in(2,+\infty)$, we have
	$k\in(3/8,+\infty)$. Note that $\frac{1}{4} |dz|\leqslant |dk|\leqslant \frac{5}{16} |dz|$. Then for any function  $f(z)$ well-defined in $f(z(k))$, it follows that
	$$ f(z)\in L^2(2,+\infty) \Longleftrightarrow  f(z(k))\in L^2(3/8,+\infty). $$
	Under this change of variable, $h^{(0)}\in  L^{2}\left( (2,\infty)\times\mathbb{R}^+\right)$ follows from Proposition \ref{proest} with
 $\|h^{(0)}\|_{L^2}\leq C(\|m_0-1\|_{H^{2,1}\cap H^{1,2}})$.
	
	The next step is to   prove that $\mathcal{T}$ is a bounded linear operator on $C_B\left(\mathbb{R}^+; L^{2}(2,\infty)\right) \cap L^{\infty}\left( (2,\infty)\times\mathbb{R}^+\right)\cap L^{2}\left( (2,\infty)\times\mathbb{R}^+\right) $.
	In fact, for any $f=(f_1,f_2)^T \in  L^{\infty}\left( (2,\infty)\times\mathbb{R}^+\right) $, it appears that
	\begin{align}
		|\mathcal{T}f (z;x)|&=\Big|\int^{x}_{+\infty}\frac{m_0(s)-1}{z^2-1}\left(\begin{array}{c}
			e^{\frac{i}{2}(z-1/z)\int^{s}_{x}m_0(l)dl}	(-zf_1(z;s)+\frac{z^2+1}{2}f_2(z;s) )   \\[4pt]
			zf_2(z;s)-\frac{z^2+1}{2}f_1(z;s)
		\end{array}\right)ds\Big|\nonumber\\
		&\leqslant \parallel f\parallel_{L^{\infty}} \int^{+\infty}_{x}|m_0-1|ds\leqslant  \parallel m_0-1\parallel_{L^1}\parallel f\parallel_{L^{\infty}}.
	\end{align}
	For $f \in C_B\left(\mathbb{R}^+; L^{2}(2,\infty)\right)$, it is inferred that
	\begin{align}
		&\left( \int_{2}^{+\infty}|\mathcal{T} f (z;x)|^2dz\right) ^{1/2}\nonumber\\
		&\leqslant\int^{+\infty}_{x}\left(  \int_{2}^{+\infty}\Big|\frac{m_0(s)-1}{z^2-1}\left(\begin{array}{c}
			e^{\frac{i}{2}(z-1/z)\int^{s}_{x}m_0(l)dl}	(-zf_1(z;s)+\frac{z^2+1}{2}f_2(z;s) )   \\[4pt]
			zf_2(z;s)-\frac{z^2+1}{2}f_1(z;s)
		\end{array}\right)\Big|^2dz\right)   ^{1/2} ds\nonumber\\
		&\leqslant \int^{+\infty}_{x}|m_0(s)-1|\left(   \int_{2}^{+\infty}(|f_1(z;s)|+|f_2(z;s)|)^2dz\right)  ^{1/2} ds\leqslant   \parallel m_0-1\parallel_{L^1}\parallel f\parallel_{C_B}.
	\end{align}
	For $f \in L^{2}\left( (2,\infty)\times\mathbb{R}^+\right) $, it follows that
		\begin{align}
		&\left(\int_{\mathbb{R}^+} \int_{2}^{+\infty}|\mathcal{T} f (z;x)|^2dzdx\right) ^{1/2}\nonumber\\
		&\lesssim\int_{\mathbb{R}^+}\left(  \int^{+\infty}_{x}|m_0(s)-1| \int_{2}^{+\infty}(|f_1(z;s)|+|f_2(z;s)|)^2dz) ^{1/2}ds\right) dx\nonumber\\
		&\leqslant \left(\int_{\mathbb{R}^+}(\int^{s}_{0}|m_0(s)-1| \int_{2}^{+\infty}(|f_1(z;s)|+|f_2(z;s)|)^2dzdx) ^{1/2}ds\right) ^2\nonumber\\
		&
		\leqslant   \parallel m_0-1\parallel_{L^{2,1/2}}\parallel f\parallel_{L^{2}}.
	\end{align}
	In general,	for $\mathcal{T}^n $ defined in \eqref{Tn}, $n=2,3,...$, it admits
	\begin{align}
		&|\mathcal{T}^n f (z;x)|\leqslant\parallel f\parallel_{L^{\infty}}\int^{+\infty}_{x}\frac{ \left( \int_s^{+\infty} |m_0(\zeta)-1|d\zeta\right) ^{n-1}}{(n-1)!} |m_0(s)-1|ds\leqslant\frac{ \parallel m_0-1\parallel_{L^1}^{n }}{n!}\parallel f\parallel_{L^{\infty}},\nonumber
	\end{align}
	and
	\begin{align}
		&\left(  \int_{2}^{+\infty}|\mathcal{T}^n f (z;x)|^2dz\right) ^{1/2}\leqslant \int^{+\infty}_{x}\frac{  \left( \int_s^{+\infty} |m_0(\zeta)-1|d\zeta\right) ^{n-1}}{(n-1)!}|m_0(s)-1|ds\parallel f\parallel_{C_B}\leqslant\frac{ \parallel m_0-1\parallel_{L^1}^{n }}{n!}\parallel f\parallel_{C_B}.\nonumber
	\end{align}
	So  the operator $(I-\mathcal{T})^{-1}$ exists on $C_B \left(\mathbb{R}^\pm; L^{2}(2,\infty) \right) \cap L^{\infty}\left( (2,\infty) \times\mathbb{R}^\pm\right)$ and admits
	\begin{align}
		\parallel(I-\mathcal{T})^{-1}\parallel\leqslant e^{\parallel m_0-1\parallel_{L^1}}.\label{normT}
	\end{align}
And for the space $L^{2}\left( (2,\infty) \times\mathbb{R}^\pm\right)$, a similar estimate gives that
\begin{align}
	\|\mathcal{T}^n\|\leqslant\frac{ \parallel m_0-1\parallel_{L^1}^{n-1 }}{(n-1)!}\|m_0-1\|{L^{2,1/2}},
\end{align}
which leads to
\begin{align*}
		\parallel(I-\mathcal{T})^{-1}\parallel\leqslant 1+e^{\parallel m_0-1\parallel_{L^1}}\|m_0-1\|{L^{2,1/2}}.
\end{align*}
	Finally, it follows from   the operator equation in \eqref{dn}  that
	$$h(z;x)=(I-\mathcal{T})^{-1}h^{(0)} \  (z;x)\in C_B\left(\mathbb{R}^+; L^{2}(2,\infty)\right) \cap L^{\infty}\left( (2,\infty)\times\mathbb{R}^+\right)\cap L^{2}\left( (2,\infty)\times\mathbb{R}^+\right). $$
	This completes the proof of Proposition \ref{lemma1}.
\end{proof}

Next, in order to obtain the property of $z$-derivative, we establish the following estimate  of the Jost function  $\mu^\pm(z;x)$ with initial data $m_0(x)-1$.
\begin{Proposition} \label{prop1}
	Let   $m_0-1\in  H^{2,1}(\mathbb{R} )\cap H^{1, 2} (\mathbb{R}) $. Then the  $z$-derivatives of $\mu^\pm(z;x)$  admit the  estimates
	\begin{align}
		\big\Vert  \partial_z \mu^\pm(z;x)  \big\Vert _{  C_B},\ \big\Vert  \partial_z^2 \mu^\pm(z;x)  \big\Vert _{ C_B}\leqslant C(\|m_0-1\|_{H^{2,1}\cap H^{1,2}},\underset{x\in\mathbb{R}}{\inf}m_0).
	\end{align}
\end{Proposition}
\begin{proof}
	By  taking derivative
	in \eqref{dn} with respect to $z$,  it is found that
	\begin{align}
		\partial_z h(z;x)=h^{(1)}(z;x)+\mathcal{T}(\partial_z h )(z;x),\ 	\partial_z^2 h(z;x)=h^{(2)}(z;x)+\mathcal{T}(\partial_z^2 h )(z;x)
	\end{align}
	where
	\begin{align*}
		&	h^{(1)}(z;x)=\partial_z h^{(0)}(z;x)+(\partial_z \mathcal{T})h(z;x),\\
		& h^{(2)}(z;x)=\partial_z^2 h^{(0)}(z;x)+(\partial_z \mathcal{T}) (\partial_zh)(z;x)+(\partial_z^2 \mathcal{T}) h(z;x).
	\end{align*}
For $h^{(1)}$, $\partial_z\mathcal{T}$ is  a linear integral  operator from $L^{2}\left( (2,\infty)\times\mathbb{R}^+\right)$ to  $ C_B\left(\mathbb{R}^+; L^{2}(2,\infty)\right) $ defined by
\begin{align}
	\partial_z\mathcal{T} f (z;x)&= \int^{x}_{+\infty}\partial_zK(z;x,s)f(z;s)ds,
\end{align}
where the concrete expression of $\partial_zK$ is
\begin{align*}
	\partial_zK(z;x,s)=&-\frac{1}{2}(1+1/z^2)(\int^{s}_{x}m_0(l)dl)e^{\frac{i}{2}(z-1/z)\int^{s}_{x}m_0(l)dl}(m_0(s)-1)\left(\begin{array}{cc}
		-\frac{z}{z^2-1} & \frac{z^2+1}{2(z^2-1)}\\
		0 & 0
	\end{array}\right)\\
	&+i(m_0(s)-1){\rm diag} (
	e^{\frac{i}{2}(z-1/z)\int^{s}_{x}m_0(l)dl}, 1) (i\partial_z(\frac{z^2+1}{2(z^2-1)})\sigma_2-\partial_z(\frac{z}{z^2-1})\sigma_3).
\end{align*}
For $z\in(2,\infty)$, direct calculation gives that
\begin{align*}
	|\partial_zK(z;x,s)|\lesssim |\int^{s}_{x}m_0(l)dl||m_0(s)-1|+|m_0(s)-1|.
\end{align*}
Thus for $f\in L^{2}\left( (2,\infty)\times\mathbb{R}^+\right)$, it appears that
\begin{align*}
	&\|	\partial_z\mathcal{T} f\|_{C_B}=\sup_{x\in\mathbb{R}^+}(\int_2^{+\infty}\big|\int^{x}_{+\infty}\partial_zK(z;x,s)f(z;s)ds\big|^2dz)^{1/2}\\
	&\leqslant\sup_{x\in\mathbb{R}^+}\int^{x}_{+\infty}(\int_2^{+\infty}|\partial_zK(z;x,s)f(z;s)|^2dz)^{1/2}ds\\
	&\lesssim \sup_{x\in\mathbb{R}^+}\int^{x}_{+\infty}(|\int^{s}_{x}m_0(l)dl||m_0(s)-1|+|m_0(s)-1|)(\int_2^{+\infty}|f(z;s)|^2dz)^{1/2}ds\\
	&\leq (\|m_0(s)-1\|_{L^{2,1}}+\|m_0(s)-1\|_{L^{2}})\|f\|_{L^2},
\end{align*}
which gives the boundedness of $	\partial_z\mathcal{T} $. Thus it follows that $h^{(1)}\in C_B\left(\mathbb{R}^+; L^{2}(2,\infty)\right) $. By following the similar proof of  Proposition \ref{lemma1}, the estimate of $\partial_z\mu^\pm$ in Proposition \ref{prop1}
is obtained. And the proof of the estimate of  $\partial_z^2\mu^\pm$ can be completed in the same way.
\end{proof}
By virtue of \eqref{intmu} and \eqref{scattering},   the  scattering data  $b(z)$ is expressed with
\begin{align}
	b(z)=&\int_{\mathbb{R}}e^{-\frac{i}{2}(z-1/z)(\int_{s}^{+\infty}(m_0(l)-1)dl-s)}\frac{i(m_0(s)-1)}{2(z^2-1)} \left( 2z\mu^+_{12}(z;s)-(z^2+1)\mu^+_{22}(z;s)\right)  ds,\label{pb1}
\end{align}
which reveals the following  result.
\begin{Proposition}\label{prob}
	Let     $m_0(x)-1  \in H^{2,1}(\mathbb{R} )\cap H^{1, 2} (\mathbb{R}) $.
	Then  $b(z) \in  H^{1,2} (2,\infty)\cap H^{2,1} (2,\infty)$ and
	$$\| b \| _{H^{1,2} (2,\infty)\cap H^{2,1} (2,\infty)} \leqslant C(\|m_0-1\|_{H^{2,1}\cap H^{1,2}},\underset{x\in\mathbb{R}}{\inf}m_0).$$
\end{Proposition}
\begin{proof}
	\eqref{pb1} can be rewritten in the form
	\begin{align}
		b(z)=&\int_{\mathbb{R}}e^{-\frac{i}{2}(z-1/z)(\int_{s}^{+\infty}(m_0(l)-1)dl-s)}\frac{i(m_0(s)-1)}{2(z^2-1)}\left(  2z\mu^+_{12}(z;s)-(z^2+1)(\mu^+_{22}(z;s)-1)\right)   ds\nonumber\\
		&-\int_{\mathbb{R}}e^{-\frac{i}{2}(z-1/z)(\int_{s}^{+\infty}(m_0(l)-1)dl-s)}\frac{z^2+1}{2(z^2-1)}i(m_0(s)-1) ds.
	\end{align}
	Applying Proposition \ref{lemma1} thus yields the $L^2$-property of the first integral. As for the second integral, it  suffices to  give estimate of the integral $\int_{\mathbb{R}}e^{-\frac{i}{2}(z-1/z)(\int_{s}^{+\infty}(m_0(l)-1)dl-s)}(m_0-1)ds$.
	
	Recall the change of variable
	$k=\frac{1}{4}\left(z-\frac{1}{z} \right) $.
	Consequently,   it suffices  to show
	$$\int_{\mathbb{R}}e^{-2i\cdot(\int_{s}^{+\infty}(m_0(l)-1)dl-s)}(m_0-1)ds\in L^2(3/8,+\infty).$$ The proof resemble the one in Proposition \ref{proest}.
	For any $f(k) \in L^2(3/8,+\infty)$,  it transpires that
	\begin{align*}
		&\Big|\int_{3/8}^{+\infty}f(k)\int_{\mathbb{R}}e^{-2ik(\int_{s}^{+\infty}(m_0(l)-1)dl-s)}(m_0(s)-1)dsdk	\Big|\nonumber\\
		&=	\Big|\int_{\mathbb{R}}(m_0(s)-1)\widehat{f}(\frac{1}{\pi} \int_{s}^{+\infty}(m_0(l)-1)dl-\frac{s}{\pi}) ds\Big|\\
		&\leqslant \parallel m_0-1\parallel_{L^2} \Big\Vert   \widehat{f}( \frac{1}{\pi}\int_{s}^{+\infty}(m_0(l)-1)dl-\frac{s}{\pi}) \Big\Vert  _{L^2},
\lesssim \parallel m_0-1\parallel_{L^2} \parallel \widehat{f}\parallel_{L^2}^2.
	\end{align*}
	This in turn implies that $b\in L^2(2,+\infty)$.
	
	Our attention is now
	turned  to the estimate of   weighted integral property of the function $b$. Applying  integration by parts, $zb(z)$ is rewritten as
	\begin{align}
		zb(z)=&-\frac{z^3}{(z^2-1)^2}\int_{\mathbb{R}}e^{-\frac{i}{2}(z-1/z)(\int_{s}^{+\infty}(m_0(l)-1)dl-s)}\partial_s( \frac{2(m_0-1)}{m_0}\mu^+_{12})ds\nonumber\\
		&+\frac{(z^2+1)z^2}{(z^2-1)^2}\int_{\mathbb{R}}e^{-\frac{i}{2}(z-1/z)(\int_{s}^{+\infty}(m_0(l)-1)dl-s)}\partial_s(\frac{m_0-1}{m_0}(\mu^+_{22}-1))ds\nonumber\\
		&+\frac{(z^2+1)z^2}{(z^2-1)^2}\int_{\mathbb{R}}e^{-\frac{i}{2}(z-1/z)(\int_{s}^{+\infty}(m_0(l)-1)dl-s)}\partial_s( \frac{m_0-1}{m_0})ds.\label{pb}
	\end{align}
	By substituting \eqref{lax1.1} into the partial derivative in above expression, it is deduced that
	\begin{align*}
		&\partial_s( \frac{2(m_0-1)}{m_0}\mu^+_{12})=-2\partial_s( \frac{1}{m_0})\mu^+_{12}+\frac{i(m_0-1)}{2}(z-\frac{1}{z})\mu^+_{12}+\frac{2i(m_0-1)^2}{m(z^2-1)}( z\mu^+_{12}-\frac{z^2+1}{2}\mu^+_{22}) ,\nonumber\\
		&\partial_s(\frac{m_0-1}{m_0}(\mu^+_{22}-1))=\partial_s(\frac{m_0-1}{m_0})(\mu^+_{22}-1)+\frac{m_0-1}{m_0}(P\mu^+)_{22},
	\end{align*}
	which follows from Proposition \ref{lemma1} that  the  both expressions in the above belong to $L^2$, where  $(P\mu^+)_{22}$ is the $(2,2)$-entry element of matrix $P\mu^+$. Specially, 
	equations above give that the $\mu^+_{12}$ item in the integral representation  \eqref{pb} of $ zb(z)$ has
	\begin{align*}
		&-\frac{3(z^2-1)^2}{2z^2}\int_{\mathbb{R}}e^{-\frac{i}{2}(z-1/z)(\int_{s}^{+\infty}(m_0(l)-1)dl-s)}\frac{i(m_0-1)z\mu^+_{12}}{z^2-1} ds\nonumber\\
		=&\int_{\mathbb{R}}e^{-\frac{i}{2}(z-1/z)(\int_{s}^{+\infty}(m_0(l)-1)dl-s)}(\frac{2i(m_0-1)^2}{m_0(z^2-1)}( z\mu^+_{12}-\frac{z^2+1}{2}\mu^+_{22})-2\partial_s( \frac{1}{m})\mu^+_{12})ds,
	\end{align*}
	which belongs to $L^2(2, \infty)$. Hence in view of \eqref{pb}, it  is accomplished that $\langle \cdot\rangle b\in L^2(2, \infty)$. Similarly, continuously using  integration by parts and the weight transform to $x$-derivative
	of $m_0$ again  leads to $\langle \cdot\rangle^2b\in L^2(\mathbb{R})$. For the $z$-derivative of $b(z)$, the estimate is obtained in the same way. Here, the essential tool is also that the norm of $z$-derivative of $b(z)$ can be controlled by the weighted integral norm of $m_0-1$ from integration by parts and Proposition \ref{proest}. In consequence,  we obtain  the desired result as indicated in Proposition \ref{prob}.
\end{proof}

Similarly,  from  \eqref{intmu} and \eqref{scattering},  the scattering data  $a(z)$ could be   expressed  by
\begin{align}
	a(z)=&1-\int_{\mathbb{R}}\frac{i(m_0(s)-1)}{2(z^2-1)} (2z\mu^+_{11}(z;s)-(z^2+1)\mu^+_{21}(z;s))ds.\label{a}
\end{align}
According to  Propositions \ref{lemma1} and  \ref{prop1}, we have  the following result.
\begin{Proposition}\label{proa}
	Let $m_0-1\in  H^{2,1}(\mathbb{R} )\cap H^{1, 2} (\mathbb{R}) $.
	Then   $a(z)$ is continuous up to the real line $ \mathbb{R}$, except at $z=\pm1$.
	Moreover, $a\in W^{2,\infty}(2,+\infty)$ with $|a(z)|\geq1$ on $\mathbb{R}$.
\end{Proposition}

The following proposition gives the  property of $r(z)$ near $z=\pm1$.
\begin{Proposition}\label{boundr}
	Let     $m_0 -1 \in H^{2,1}(\mathbb{R} )\cap H^{1, 2} (\mathbb{R}) $. Then  $r$, $\partial_zr$ and  $\partial_z^2r$ are bounded near $z=\pm1$.
\end{Proposition}
\begin{proof}
	The proof of this proposition resemble as the one in the case $(2,+\infty)$.
	In view of the definition of $r$ in \eqref{symr},  it is necessary to analyze the integral property of the function $b$.
	Since $b$ can be represented by the Jost function $\mu^\pm$, the strategy is to first estimate $\mu^\pm$ in this proof. 	It is sufficient to prove this proposition  near $z=1$.
	As shown in Proposition \ref{promu1}, $\mu^\pm$  have singularities on $z=\pm1$. Furthermore, by the symmtry of $\mu^\pm$ in \eqref{symPhi1},  as $z\to1$,
	\begin{align}
		\mu^\pm(z;x)=\frac{i\alpha_\pm(x)}{2(z-1)}\left(\begin{array}{cc}
			-1 & 1 \\
			-1 & 1
		\end{array}\right)+\mu^{\pm,(0)}+\mathcal{O}(z-1),\label{z1mu1}
	\end{align}
	with $\mu^{\pm,(0)}_{11}=\mu^{\pm,(0)}_{22}\in\mathbb{R}$, $\mu^{\pm,(0)}_{21}=\mu^{\pm,(0)}_{12}\in\mathbb{R}$.
	From \eqref{deta}, it appears that
	\begin{align*}
		a=\mu^+_{11}\mu^-_{22}-\mu^+_{21}\mu^-_{12},\ \  b=e^{-2p}(\mu^+_{12}\mu^-_{22}-\mu^+_{22}\mu^-_{12}).
	\end{align*} Then in view of  \eqref{z1mu1}, as $z\to1$,
	\begin{align*}
		&	a(z)=\frac{ic^{(0)}}{2(z-1)}+\mathcal{O}(1),\ 	b(z)=\frac{-ic^{(0)}}{2(z-1)}+\mathcal{O}(1),
	\end{align*}
	where $c^{(0)}=\alpha_-(\mu^{+,(0)}_{11}-\mu^{+,(0)}_{21})+\alpha_+(\mu^{-,(0)}_{12}-\mu^{-,(0)}_{22})$. Therefore, the definition of $r$ gives that $r$ is bounded at $z=1$.
	To analyze the $z$-derivative property of $\mu^\pm$ near $z=\pm1$, the  $z$-derivative of $\upsilon^\pm$ needs to be considered. To see this, it is found from \eqref{v1} that
	\begin{align}
		\partial_z\upsilon_1^\pm(z;x)=&\left(\begin{array}{cc}
			z-i	   \\[4pt]
			iz
		\end{array}\right)+\int^{x}_{\pm \infty} \frac{m_0(s)-1}{2}\partial_z\Gamma (z,x,s)\sigma_3\upsilon^\pm(z;s) ds+\tilde{\mathcal{T}}(\partial_z\upsilon^\pm)(z;x),
	\end{align}
	where the  calculation gives that $\partial_z\Gamma $ can be divided into two parts by
	\begin{align*}
		\partial_z\Gamma (z,x,s)=\frac{1}{z^2-1}\Gamma^{(1)} (z,x,s)+\Gamma^{(0)} (z,x,s).
	\end{align*}
	$\Gamma^{(1)}$ and $\Gamma^{(0)}$  are both analytic and bounded matrix functions with respect to $z$ near $z=\pm1$ whose  expressions are omitted for brevity. Apply a similar approach  as $z\to1$ to get
	\begin{align}
		(z-1)\partial_z\upsilon_1^\pm=\frac{\upsilon_{z,1}^{\pm,(-1)}(x)}{2}+\mathcal{O}(z-1).
	\end{align}
	Therefore, it thereby follows  again from  the symmetry \eqref{symPhi1} and
	\begin{align*}
		\partial_z\mu^\pm(z;x)=-\frac{2z}{(z^2-1)^2}D(z)\upsilon^\pm(z;x)+	\frac{1}{z^2-1}D(z)\partial_z \upsilon^\pm(z;x)+\frac{1}{z^2-1}\partial_z D(z)\upsilon^\pm(z;x),
	\end{align*}
	 that as $z\to1$
	\begin{align}
		\partial_z\mu^\pm(z;x)=\frac{i\alpha_\pm^{(1)}(x)}{4(z-1)^2}\left(\begin{array}{cc}
			1 & -1 \\
			1 & -1
		\end{array}\right)+\frac{\mu_z^{\pm,(1)}(x)}{z-1}+\mathcal{O}(1),\label{z1mu2}
	\end{align}
	with
	\begin{align}
	&\alpha_\pm^{(1)}=-i(2\upsilon^\pm_{11}+2\upsilon^\pm_{21}+\upsilon_{z,11}^{\pm,(-1)}+\upsilon_{z,21}^{\pm,(-1)})\in \mathbb{R}, \nonumber\\
&	\mu_{z,11}^{\pm,(-1)}=\mu_{z,22}^{\pm,(-1)}\in i\mathbb{R},  \ \ \mu_{z,12}^{\pm,(-1)}=\mu_{z,21}^{\pm,(-1)}\in i\mathbb{R}. \nonumber
	\end{align}
	
	Taking $z$-derivative  of $a$ and $b$ in \eqref{deta}, it is deduced that
	\begin{align*}
		&\partial_za=\partial_z\mu^+_{11}\mu^-_{22}+\mu^+_{11}\partial_z\mu^-_{22}-\partial_z\mu^+_{21}\mu^-_{12}-\mu^+_{21}\partial_z\mu^-_{12},\\  &\partial_zb=\partial_ze^{-2p}(\mu^+_{12}\mu^-_{22}-\mu^+_{22}\mu^-_{12})+e^{-2p}(\partial_z\mu^+_{12}\mu^-_{22}+\mu^+_{12}\partial_z\mu^-_{22}-\partial_z\mu^+_{22}\mu^-_{12}-\mu^+_{22}\partial_z\mu^-_{12}).
	\end{align*}
	Thus by \eqref{z1mu1} and \eqref{z1mu2}, as $z\to1$, we have
	\begin{align}
		&\partial_za=\frac{ic^{(1)}}{4(z-1)^2}+\mathcal{O}\left( (z-1)^{-1}\right) ,\ \partial_zb=\frac{-ic^{(1)}}{4(z-1)^2}+\mathcal{O}\left( (z-1)^{-1}\right),
	\end{align}
	where $c^{(1)}=\alpha_-^{(1)}(\mu^{+,(0)}_{11}-\mu^{+,(0)}_{21})+\alpha_+^{(1)}(\mu^{-,(0)}_{12}-\mu^{-,(0)}_{22})+2\alpha_-(\mu^{+,(-1)}_{11}-\mu^{+,(-1)}_{21})+2\alpha_+(\mu^{-,(-1)}_{12}-\mu^{-,(-1)}_{22})$.  Then  $a\partial_zb-b\partial_z a=\mathcal{O}(z-1)^{-2}$ as $z\to1$.
	This  also gives the boundedness of \begin{align*}
		\partial_z r=\frac{1}{a^2}(a\partial_zb-b\partial_z a).
	\end{align*}
The analysis of $\partial_z^2r$ resembles the methods shown above, which is omitted here. This completes the proof of Proposition \ref{boundr}.
	\end{proof}
\begin {proof} [Proof of Lemma \ref{r1}]
In view of the definition of $r$ in (\ref{symr}), the result of Lemma \ref{r1} is obtained directly from Proposition \ref{proa} and Proposition \ref{prob}.
\end{proof}
\vskip 0.1cm
\noindent \textbf{The norming constant.}\quad
We denote the norming constant $c_j$ by
\begin{align}
	c_j=\frac{1}{b_ja'(z_j)},\ j=1,...,{2N_0},
\end{align}
where $b_j$ is  the ratio of $\mu^+_1$ and $\mu^-_2$ such that
\begin{align}
	\mu^+_1(z_j)=b_je^{-2p(z_j)}\mu^-_2(z_j).\label{bj}
\end{align}
The symmetry of $\mu^\pm$ implies that $b_j\in\mathbb{R}$ and $\mu^+_1(-\bar{z}_j)=b_je^{-2p(-\bar{z}_j)}\mu^-_2(-\bar{z}_j)$. Thus $c_j=c_{j+{N_0}}$.

\subsection{The RH problems }\label{secRH}
By the Jost functions $\mu^\pm(z;x)$ and the function $a(z)$,  two  piecewise  analytical  matrices are defined as follows,
\begin{align}
	M_l(z;x)=\left\{ \begin{array}{ll}
		\left( \frac{ \mu^-_1  (z;x) } {a^*(z)},  \mu^+_2  (z;x)\right),   &\text{as } z\in \mathbb{C}^+,\\[12pt]
		\left(  \mu^+_1(z;x),\frac{ \mu^-_2(z;x)}{a(z)}\right)  , &\text{as }z\in \mathbb{C}^-,\\
	\end{array}\right.\label{ml}\\
	M_r(z;x)=\left\{ \begin{array}{ll}
		\left( \mu^-_1  (z;x), \frac{\mu^+_2  (z;x)  } {a^*(z)} \right),   &\text{as } z\in \mathbb{C}^+,\\[12pt]
		\left( \frac{\mu^+_1(z;x) }{a(z)} ,\mu^-_2(z;x)\right)  , &\text{as }z\in \mathbb{C}^-.\\
	\end{array}\right.\label{mr}
\end{align}
In order to construct the RH problem only depending explicitly on the scattering data, we introduce the space variable
\begin{equation}
	y(x)=x+  \int_{+\infty}^{x}(m(s)-1)ds.
\end{equation}
Under this new scaling, when $t=0$, the functions $M_l$ and $M_r$ denoting in \eqref{ml} and \eqref{mr} respectively become
\begin{align}
	M_l(z) := M_l(z;y(x)),\ \
	M_r(z):=  M_r(z;y(x))
\end{align}
Obviously,
$$M_l(z;y)\to I, \ \ y\to\infty, \ \ {\rm and} \ \  M_r(z;y)\to I, \ \ y\to-\infty, $$
which  inspires us to  consider  $M_l(z;y)$ on $y\in[0,+\infty)$ and $M_r(z;y)$ on $y\in(-\infty,0]$. And from the property of $\mu^\pm(z;x)$, \eqref{scattering} and \eqref{bj},  it then is inferred  that
$M_l(z)$ solves the following RH problem,
\begin{RHP}\label{RHP1}
	Find a matrix-valued function $M_l(z)=M_l(z;y)$ which satisfies
	
	$\bullet$ Analyticity: $M_l(z)$ is meromorphic  in $\mathbb{C}\setminus \mathbb{R}$;
	
	$\bullet$ Symmetry: $M_l(z)=\sigma_1\overline{M_l(\bar{z})}\sigma_1=\sigma_2M_l(-z)\sigma_2=\sigma_1M_l(z^{-1})\sigma_1$;
	
	$\bullet$ Jump condition: $M_l(z)$ has continuous boundary values $[M_l]_\pm(z)$ on $\mathbb{R}$ and
	\begin{equation}
		[M_l]_+(z)=[M_l]_-(z)V_l(z),\hspace{0.5cm}z \in \mathbb{R},
	\end{equation}
	where
	\begin{equation}
		V_l(z)=\left(\begin{array}{cc}
			1-|r(z)|^2 & r(z)e^{2\theta(z)}\\
			-\bar{r}(z)e^{-2\theta(z)} & 1
		\end{array}\right),   \label{V}
	\end{equation}
	with
	\begin{align}
		\theta(z )=p(z;x(y))=-\frac{1}{4}i \left(z-z^{-1} \right) y;
	\end{align}
	
	$\bullet$ Asymptotic behavior:
	\begin{align}
		&M_l(z) = I+\mathcal{O}(z^{-1}),\hspace{0.5cm}|z| \rightarrow \infty;
	\end{align}
	
	$\bullet$ Singularity: $M_l(z)$ has singularities at $z=\pm1$ with
	\begin{align}
		&M_l(z)= \frac{i\alpha_+(x)}{2(z-1)}\left(\begin{array}{cc}
			-c & 1 \\
			-c & 1
		\end{array}\right)+\mathcal{O}(1),\ z\to  1\ \text{ in }\ \mathbb{C}^+,\\
		&M_l(z)= \frac{i\alpha_+(x)}{2(z+1)}\left(\begin{array}{cc}
			-c & -1 \\
			c & 1
		\end{array}\right)+\mathcal{O}(1),\ z\to  -1\ \text{ in }\ \mathbb{C}^+,
	\end{align}
	where  $c=0$ in generic case, while $c\neq0$ in non-generic case;
	
	$\bullet$
	Residue condition: $M_l(z)$ has simple poles at each point in $ Z_0\cup \bar{Z_0}$ with
	\begin{align}
		&\res_{z=z_j}M_l(z)=\lim_{z\to z_j}M_l(z)\left(\begin{array}{cc}
			0 & c_je^{2\theta(z_j)}\\
			0 & 0
		\end{array}\right),\\
		&\res_{z=\bar{z}_j}M_l(z)=\lim_{z\to \bar{z}_j}M_l(z)\left(\begin{array}{cc}
			0 & 0\\
			\bar{c}_je^{-2\theta(\bar{z}_j)} & 0
		\end{array}\right),
	\end{align}
	where $2\theta(z_j)=2\theta(\bar{z}_j)=y$Im$z_j$.
\end{RHP}
The set $( r,\{z_j,c_j\}_{j=1}^{2N_0}) $ as the scattering data   is  crucial for construction of the RH problem \ref{RHP1}.
Analogously, it is also found that
$M_r(z)$ solves the following RH problem,
\begin{RHP}\label{RHPr1}
	Find a matrix-valued function $M_r(z)=M_r(z;y)$ which satisfies the same analyticity, symmetry and asymptotic behavior as the RH problem \ref{RHP1}, and

	$\bullet$ Jump condition: $M_r(z)$ has continuous boundary values $[M_r]_\pm(z)$ on $\mathbb{R}$ and
	\begin{equation}
		[M_r]_+(z)=[M_r]_-(z)V_r(z),\hspace{0.5cm}z \in \mathbb{R},
	\end{equation}
	where
	\begin{equation}
		V_r(z)=\left(\begin{array}{cc}
			1 & \tilde{r} (z)e^{2\theta(z)}\\
			-\bar {\tilde{r}}(z)e^{-2\theta(z)} & 1-|\tilde{r}(z)|^2
		\end{array}\right), \ \  \tilde{r} (z)=\dfrac{b(z)}{a^*(z)}.   \label{V}
	\end{equation}

	$\bullet$ Singularity: $M_r(z)$ has singularities at $z=\pm1$ with
	\begin{align}
		&M_r(z)= \frac{i\alpha_+(x)}{2(z-1)}\left(\begin{array}{cc}
			-1 & c \\
			-1 & c
		\end{array}\right)+\mathcal{O}(1),\ z\to  1\ \text{ in }\ \mathbb{C}^+,\\
		&M_r(z)= \frac{i\alpha_+(x)}{2(z+1)}\left(\begin{array}{cc}
			-1 & -c \\
			1 & c
		\end{array}\right)+\mathcal{O}(1),\ z\to  -1\ \text{ in }\ \mathbb{C}^+,
	\end{align}
	where $c=0$ in the generic case, while $c\neq0$ in the non-generic case;
	
	$\bullet$
	Residue condition: $M_r(z)$ has simple poles at each point in $ Z_0\cup \bar{Z_0}$ with
	\begin{align}
		&\res_{z=z_j}M_r(z)=\lim_{z\to z_j}M_r(z)\left(\begin{array}{cc}
			0 & 	0\\
			\tilde{c}_je^{-2\theta(z_j)} & 0
		\end{array}\right),\\
		&\res_{z=\bar{z}_j}M_r(z)=\lim_{z\to \bar{z}_j}M_r(z)\left(\begin{array}{cc}
			0 & 	\bar{\tilde{c}}_je^{2\theta(\bar{z}_j)} \\
			0	& 0
		\end{array}\right).
	\end{align}
\end{RHP}
 In the residue condition,
 \begin{align}
 \tilde{c}_j=\tilde{c}_{j+{N_0}}=\frac{b_j}{a'(z_j)},\ j=1,...,{N_0},
 \end{align}
 where $b_j$ is defined in \eqref{bj}. Obviously, $\tilde {r}(z)$ has the same properties as $r(z)$.

If the RH problems \ref{RHP1} and \ref{RHPr1} have solutions, it must be unique. In fact,
If $M_l$ and $\tilde{M}_l$ both are solution of  the RH problems \ref{RHP1}, then it is verifies that $\tilde{M}_l^{-1}M_l$ is a  analytic function on $\mathbb{C}$ and tends to $I$ as $|z|\to\infty$. So by Liouvill's   theorem, it  follows that  $M_l=\tilde{M}_l$.

To overcome difficulties from the  singularity conditions of the RH problems \ref{RHP1} and \ref{RHPr1} on $z=\pm1$, we consider  a matrix function $M_l^{(1)}(z)$ which is the solution of the following RH problem without singularity on $z=\pm1$.
\begin{RHP}\label{RHP2}
	Find a matrix-valued function $M_l^{(1)}(z)=M_l^{(1)}(z;y)$ having no  singularity on $z=\pm1$ and satisfying same
	analyticity,  jump condition, Asymptotic behavior and residue condition as the RH problem \ref{RHP1}, and

	$\bullet$ Symmetry: $M_l^{(1)}(z)=\sigma_1\overline{M_l^{(1)}(\bar{z})}\sigma_1=\sigma_2M_l^{(1)}(-z)\sigma_2$.
\end{RHP}
And we will show in Section \ref{sec4} that
this RH problem exists a unique solution.

Then it can easily verify that
$M_l(z;y)$ defined in the following transformation is the solution of the original RH problem \ref{RHP1}:
\begin{align}
	\left( I-\frac{\sigma_1}{z}\right) M_l(z;y)=\left( I-\frac{\sigma_1M_l^{(1)}(0;y)^{-1}}{z} \right) M_l^{(1)}(z;y).\label{tran1}
\end{align}

Similarly, for $y\in(-\infty,0]$,  by means of the following transformation
\begin{align}
\left( I-\frac{\sigma_1}{z}\right) M_r(z;y)=\left( I-\frac{\sigma_1M_r^{(1)}(0;y)^{-1}}{z} \right) M_r^{(1)}(z;y),\label{tran2}
\end{align}
 it is easy to check that $M_r (z)$   is the  solution of the RH problem  \ref{RHPr1}, where  the  matrix function
$M_r^{(1)}(z;y)$ satisfies the  above  similar type  RH problem \ref{RHP2} under $\left( \tilde {r},\{z_j,\tilde {c}_j\}_{j=1}^{2N_0}\right) $ by replacing
  $M_l^{(1)}(z;y)$ with $M_r^{(1)}(z;y)$.

\subsection{Reconstruction  of the  potential $m$}\label{secrec}
The aim  in  this subsection is to give the reconstruction formula
of $m$.
By  the construction of $M_l(z;y)$ on $y\in[0,+\infty)$ and $M_r(z;y)$ on $y\in(-\infty,0]$  in (\ref{ml}) and  (\ref{mr}) respectively, $m_0$ can be reconstructed from the  behavior of the analytic continuations at $z=\infty$ and $z=i$ of $M_l(z;y)$ and  $M_r(z;y)$ from \eqref{asymuinf} and \eqref{asyM}.
\begin{Proposition}\label{lemmaM}
	If $M(z;y)$ is the solution of the RH problem \ref{RHP1} for $y\in[0,+\infty)$ or the RH problem \ref{RHPr1} for $y\in(-\infty,0]$, then the potential function $m$ can be recovered  from $M(z;y)$ with
	\begin{align}
		m=\left( 1-\eta\right) ^{-1},\label{rescm}\ \eta=\lim_{z\to \infty}z M_{12}(z).
	\end{align}
	Moreover,  as $z\to i$, $M(z;y)$ has the following expansion:
	\begin{align}
		M(z)=\left(\begin{array}{cc}
			a_1 & 0\\
			0 & a_1^{-1}
		\end{array}\right)+\left(\begin{array}{cc}
			0 & a_2\\
			a_3 & 0
		\end{array}\right)(z-i)+\mathcal{O}\left( (z-i)^2\right),
	\end{align}
	where $a_j$ is a real function independent of $z$.
	In addition,  $u=(1-\partial_x^2)^{-1}m$  is  obtained from $M$ directly, that is
	\begin{align}
		&	u(y)=	1-M_{11}(i)\lim_{z\to i}\left(\frac{M_{12}(z)-M_{12}(i)}{z-i} \right) -M_{22}(i)\lim_{z\to i}\left(\frac{M_{21}(z)-M_{21}(i)}{z-i} \right),\nonumber\\
		&	x=y+2\log \left( M_{11}(i)\right). \label{u}
	\end{align}
\end{Proposition}
On the other hand, the following proposition reveals that  the solutions reconstructed from two RH problems in Proposition \ref{lemmaM}  are actually same at $ y = 0. $
\begin{Proposition}\label{Proeq}
	For $y\in[0,+\infty)$, denote $u_l$ and  $m_l$ as those  recovered  from   $M_l(z;y)$, and for  $y\in(-\infty,0]$  denote $u_r$, $m_r$ as those  recovered  from   $M_r(z;y)$. Then  $m_l(0)=m_r(0)$ and  $u_l(0)=u_r(0)$.
\end{Proposition}
\begin{proof}
	Denote
	\begin{align}
		\tilde{a}(z)=\left\{ \begin{array}{ll}
			a^*(z),   &\text{as } z\in \mathbb{C}^+,\\[12pt]
			a^{-1}(z) , &\text{as }z\in \mathbb{C}^-,\\
		\end{array}\right. \hspace{0.3cm}	\tilde{a}\to I,\ |z|\to\infty.
	\end{align}
	Then $M_l(z;y)=M_r(z;y)	\tilde{a}(z)^{-\sigma_3}$ with $\lim_{z\to \infty}\left(z M_{r,12}(z;y)\right)=\lim_{z\to \infty}\left(z M_{l,12}(z;y)\right)$ for $y=0$. And  it is deduced from  \eqref{u} that  $u_l(y)=u_r(y)$.
\end{proof}
Therefore, we define
\begin{align}
	M^{(1)}(z;y)=\left\{ \begin{array}{ll}
		M^{(1)}_l(z;y),   &\text{as } y\in [0,+\infty),\\[12pt]
			M^{(1)}_r(z;y) , &\text{as }y\in (-\infty,0],\\
	\end{array}\right.\label{M1}
\end{align}
which is used  to recover the solution $m_0(x)$ through the following lemma, which is a directly result from Proposition \ref{lemmaM} and \eqref{transM}.
\begin{lemma} \label{lemm3.2}
	Suppose that $M^{(1)}(z;y)$ is the  solution of the RH problem \ref{RHP2}. Then
	\begin{align*}
		M^{(1)}(0 ;y)=\left(\begin{array}{cc}
			\alpha(y) & i\beta(y) \\
			-i\beta(y) & \alpha(y)
		\end{array}\right),\ \alpha^2-\beta^2=1,
	\end{align*}
	where $\alpha(y)$ and $\beta(y)$ are two real functions.
	And when $\beta\neq0$, as $z\to i$, $M^{(1)}(z;y)$ has the expansion
	\begin{align*}
		M^{(1)}(z;y)=&\left(\begin{array}{cc}
			f_1(y) &\frac{ i\beta}{\alpha+1}f_2(y) \\
			-\frac{ i\beta}{\alpha+1}f_1(y) & f_2(y)
		\end{array}\right)\nonumber\\
		&+\left(\begin{array}{cc}
			\frac{ i\beta}{\alpha+1}g_1(y) & g_2(y) \\
			g_1(y) & -\frac{ i\beta}{\alpha+1}g_2(y)
		\end{array}\right)(z-i)+\mathcal{O}((z-i)^2),
	\end{align*}
	where $g_1(y)$, $g_2(y)$, $f_1(y)$ and  $f_2(y)$ are real functions.
	Then the following  formula  gives  $m_0(x)=m(x(y))$ with
	\begin{align}
		u_0(y)=	&1-\alpha_2(y)\alpha_1(y)-\alpha_3(y)\alpha_1(y) ^{-1},\label{res1}\\
		x(y)=&y+2\log \left( \alpha_1(y)\right) ,\label{res2}\\
		m_0(y)=&u_0(y)-u_{0, xx}(y),
	\end{align}
	where
	\begin{align}
		&\alpha_1(y)=\left( 1-\frac{\beta}{\alpha+1}\right) f_1,\ \ 	\alpha_2(y)=\frac{\beta}{\alpha+1}f_2+\left( 1-\frac{\beta}{\alpha+1}\right) g_2,\label{al1}\\
		&\alpha_3(y)=\frac{-\beta}{\alpha+1}f_1+\left( 1-\frac{\beta}{\alpha+1}\right) g_1.\label{al2}
	\end{align}
\end{lemma}

\begin{remark}
	In the case $\beta=0$, then $\alpha=\pm 1$. When $\alpha=1$, it only needs to take $\frac{\beta}{\alpha+1}=0$ in the above  formula. But when  $\alpha=-1$, by the symmetry of $M^{(1)}(z;y)$, it appears  that  as $z\to i$, $	 M^{(1)}(z;y)$ has the expansion
	\begin{align*}
		M^{(1)}(z;y)=&\left(\begin{array}{cc}
			0 &f_1(y)i \\
			f_1(y)^{-1}i  & 0
		\end{array}\right)+\left(\begin{array}{cc}
			g_1(y)i  & 0\\
			0	& g_2(y)i
		\end{array}\right)(z-i)+\mathcal{O}\left( (z-i)^2\right) .
	\end{align*}
	Under this case, the function $\alpha_j$ in the reconstruction formula become
	\begin{align*}
		&\alpha_1(y)= f_1^{-1}(y),\ 	\alpha_2(y)=f_1(y)+ g_2(y),\ \alpha_3(y)=f_1^{-1}(y)+g_1(y).
	\end{align*}
	On the other hand, we declare that $\alpha_1(y)\neq0$, which is from {\rm det}$(M^{(1)}(z;y))\equiv1$.
\end{remark}

\section{Inverse scattering transform}\label{sec3}
In this section, we shall provide a rigorous framework to study
the inverse scattering transform based on the RH problem of complex
analysis.
\subsection{The  Cauchy projection operator}\label{sec4}
For  a given  function $f(z) \in L^p(\mathbb{R})$, $1\leqslant p<\infty$, the Cauchy operator $\mathcal{C}$ is defined  by
\begin{align}
	\mathcal{C}f(z)=\frac{1}{2\pi i}\int_{\mathbb{R}}\frac{f(s)}{s-z}ds, \ z\in\mathbb{C}\setminus\mathbb{R}.
\end{align}
The function $\mathcal{C}f$ is analytic off the real line such that $\mathcal{C}f(\cdot+i\gamma)$ is in $L^p(\mathbb{R})$ for each $0\neq\gamma\in\mathbb{R}$.
When $z$ approaches to a point on the real line transversely from the upper and lower
half planes, that is, if $\gamma\to \pm0$ in the following expression, the Cauchy operator $\mathcal{C}$ becomes the  Cauchy projection operator $\mathcal{C}_\pm$:
\begin{align}
	&\mathcal{C}_\pm f(z)=\lim_{\gamma\to 0}\frac{1}{2\pi i}\int_{\mathbb{R}}\frac{f(s)}{s-(z\pm i\gamma )}ds, \ z\in\mathbb{R},\\
	&(\mathcal{C}_\pm f)^\wedge (z)= X_\pm(z)	 \hat{f}(z),
\end{align}
where  $ X_\pm$ denotes the characteristic function on $\mathbb{R}^\pm$.
The following proposition summarizes the basic properties of the Cauchy  projection
operators \cite{book}.
\begin{Proposition} \cite{book}
	For any $ f\in L^p(\mathbb{R})$, $1\leqslant p<\infty$, the Cauchy integral  $\mathcal{C} f$ is analytical  off
	the real line, decays to zero as $|z|\to\infty$, and approaches to $\mathcal{C}_\pm f $ almost everywhere
	when a point $z \in\mathbb{C}^\pm$ approaches to a point on the real axis by any non-tangential contour
	from $\mathbb{C}^\pm$. If $1< p<\infty$, then there exists a positive constant $C_p$   such that
	\begin{align}
		\parallel \mathcal{C}_\pm f\parallel_{L^p}\leqslant C_p\parallel f \parallel_{L^p}.
	\end{align}
	When $f\in L^1(\mathbb{R})$, as $z\to\infty$, $\mathcal{C} f(z)=\mathcal{O}(z^{-1})$.
\end{Proposition}

\subsection{Solvability of the RH problem}
In this subsection, we prove  solvability of  the RH problem \ref{RHP2} about $M^{(1)}_l$ for $y\in[0,+\infty)$  under
the given scattering data: $$R=\left( r(\cdot),\{z_j,c_j\}_{j=1}^{2{N_0}}\right)\in  Y,$$
where $Y$ is a Banach space    defined in \eqref{banch},   and     $r(z)$ admits the symmetries
\begin{align}
	&	 	r(z)=\overline{r(z^{-1})}=-\overline{r(-z)},\ r(1)=-r(-1)=-1.\label{spr}
\end{align}
For $z\in\mathbb{R}$, denote $w=w_++w_-$ with
\begin{align}
	&   w_+(z;y)=\left(\begin{array}{cc}
		0  & 0\\
		-\bar{r}(z)e^{-2\theta(z;y)}	& 0
	\end{array}\right),\ w_-(z;y)=\left(\begin{array}{cc}
		0  & r(z)e^{2\theta(z;y)}\\
		0	& 0
	\end{array}\right).
\end{align}
Then in view of  the Beals-Coifman  theorem in \cite{p2}, the RH problem \ref{RHP2} has a solution  given by the Cauchy operator
\begin{align}
	M^{(1)}_l(z;y)=I+\mathcal{C}(\mu w)(z;y)+\sum_{j=1}^{2{N_0}} \left(\frac{\bar{c}_je^{y\text{Im}z_j}}{z-\bar{z}_j}\mu_2(\bar{z}_j;y),\frac{c_je^{y\text{Im}z_j}}{z-z_j}\mu_1(z_j;y)  \right), \label{intM}
\end{align}
if and only if there is  a solution $\mu(z;y)$ of the Fredholm integral equation
\begin{align}
	\mu(z;y)=I+\mathcal{C}_w	\mu(z;y) +\sum_{j=1}^{2{N_0}} \left(\frac{\bar{c}_je^{y\text{Im}z_j}}{z-\bar{z}_j}\mu_2(\bar{z}_j;y),\frac{c_je^{y\text{Im}z_j}}{z-z_j}\mu_1(z_j;y)  \right).\label{mu}
\end{align}
And $\mu_1(z;y)$ and $\mu_2(z;y)$ denote the first and second columns of $\mu(z;y)$ respectively,
and $$\mathcal{C}_wf(z;y)=\mathcal{C}_+(fw_-)(z;y)+\mathcal{C}_-(fw_+)(z;y).$$
Thus the attention turns to  the estimate and  solvability of the Fredholm integral equation \eqref{mu}. The proof basically follows the idea in
\cite{ZHOU,zhou1}, which was later used in \cite{Deift2011,liu} and so on.  To this end, 
 we  take the first row of $\mu(z)$   as an instance. A  simple calculation gives that $\|\mathcal{C}_wI\|_{L^2}\leqslant2 \|R\|_{Y}$ for $\|R\|_{Y}$ in \eqref{normD}.
We denote a new  column vector function by
\begin{align}
	N(z;y):=\left( N_{11}(z;y), N_{12}(z;y) \right)^T,\hspace*{0.3cm} N_{11}(z;y)= \mu_{11}(z;y)-1,\hspace*{0.3cm} N_{12}(z;y)=\mu_{12}(z;y)  .\label{dN}
\end{align}
Via  \eqref{mu}, it admits that for $z\in\mathbb{R}$:
\begin{align}
&	N_{11}(z;y)=\mathcal{C}_-(-\bar{r}e^{-2\theta(\cdot;y)}N_{12}(\cdot;y))(z;y)+ \sum_{j=1}^{2{N_0}}\frac{N_{12}(\bar{z}_j;y)\bar{c}_je^{y\text{Im}z_j}}{z-\bar{z}_j}, \\
		&N_{12}(z;y)=N_{0,2}(z;y)+\mathcal{C}_+(re^{2\theta(\cdot;y)}	N_{11}(\cdot;y))(z;y)+\sum_{j=1}^{2{N_0}}\frac{N_{11}(z_j;y)c_je^{\text{Im}z_jy}}{z-z_j},
\end{align}
and for $j=1,...,2N_0$,
\begin{align}
	&	N_{11}(z_j;y)=\mathcal{C}_-(-\bar{r}e^{-2\theta(\cdot;y)}N_{12}(\cdot;y))(z_j;y)+ \sum_{j=1}^{2{N_0}}\frac{N_{12}(\bar{z}_j;y)\bar{c}_je^{y\text{Im}z_j}}{z_j-\bar{z}_j}, \\
	&N_{12}(\bar{z}_j;y)=N_{0,2}(\bar{z}_j;y)+\mathcal{C}_+(re^{2\theta(\cdot;y)}	N_{11}(\cdot;y))(\bar{z}_j;y)+\sum_{j=1}^{2{N_0}}\frac{N_{11}(z_j;y)c_je^{\text{Im}z_jy}}{\bar{z}_j-z_j},
\end{align}
where
\begin{align}
	N_{0,2}(z;y)=\mathcal{C}_+(re^{2\theta})(z;y)+\sum_{j=1}^{2{N_0}}\frac{c_je^{\text{Im}z_jy}}{z-z_j}.
\end{align}
Combining above equations, it immediately follows that
\begin{align}
	N_{11}(z_j;y)=\mathcal{C}(	N_{11})(z_j;y),\hspace*{0.3cm} N_{12}(\bar{z}_j;y)=\mathcal{C}(	N_{12})(\bar{z}_j;y),\label{Nzj}
\end{align}
which inspires us to only consider $N_{11},\ N_{12}\in  L^2(\mathbb{R}).$
The integral equation  \eqref{mu} can be rewritten as
\begin{align}
	&N(z;y)= N_0(z;y)+\mathcal{K}N(z;y), \label{N}\\
	&N_0(z;y)=\left(0, N_{0,2}(z;y)\right)^T,\hspace*{0.3cm}\mathcal{K}=\left(\begin{array}{cc}
		0  & \mathcal{K}^{(1)}\\
		\mathcal{K}^{(2)}	& 0
	\end{array}\right),\label{K}
\end{align}
where  $\mathcal{K}^{(1)}$, $\mathcal{K}^{(2)}:L^2(\mathbb{R})\to L^2(\mathbb{R})$ are two operators  acting on scalar functions defined as follows. Denote  $h$ as a generic function in $L^2(\mathbb{R})$. Then
\begin{align}
	&\mathcal{K}^{(1)}=\mathcal{K}^{(1)}_{1}+\mathcal{K}^{(1)}_{2},\hspace*{0.3cm} \mathcal{K}^{(2)}=\mathcal{K}^{(2)}_{1}+\mathcal{K}^{(2)}_{2},\\
	&\mathcal{K}^{(1)}_{1}h(z;y)=\mathcal{C}_-(-\bar{r}e^{-2\theta(\cdot;y)}h)(z;y),\ \mathcal{K}^{(1)}_{2}h(z;y)=\sum_{j=1}^{2{N_0}}\frac{\mathcal{C}(	h)(\bar{z}_j;y)\bar{c}_je^{y\text{Im}z_j}}{z-\bar{z}_j},\label{A11}\\
	&\mathcal{K}^{(2)}_{1}h(z;y)=\mathcal{C}_+(re^{2\theta(\cdot;y)}h)(z;y),\ \mathcal{K}^{(2)}_{2}h(z)=\sum_{j=1}^{2{N_0}}\frac{\mathcal{C}(	h)(z_j;y)c_je^{y\text{Im}z_j}}{z-z_j}.\label{B22}
\end{align}
 It is adduced that $N_0\in L^2(\mathbb{R})$. On the other hand, on account of $\|r\|_{L^\infty}=1$, there exists a positive constant $C(\|R\|_{Y})$ such that $\|\mathcal{K}\|=\|\mathcal{K}^{(1)}\|+\|\mathcal{K}^{(2)}\|\leqslant C(\|R\|_{Y})$. Our expectation  is to demonstrate $N\in L^2(\mathbb{R})$ through \eqref{N}. Owing to the special structure of $\mathcal{K}$, we consider the operator
\begin{align*}
	\mathcal{K}^2=\left(\begin{array}{cc}
		\mathcal{K}^{(1)}\mathcal{K}^{(2)} & 0\\
		0	& \mathcal{K}^{(2)}\mathcal{K}^{(1)}
	\end{array}\right).
\end{align*}
For simplicity of notations, we drop the dependence of the operator $\mathcal{K}$, the functions $\mu$, $M^{(1)}_l$ and $N$ from the variable of $y$.

The following proposition is given to show the invertibility of $I-\mathcal{K}$.
\begin{Proposition}\label{lemmaG}
	If   $R=\left( r,\{z_j,c_j\}_{j=1}^{2{N_0}}\right)  \in Y$,  then  $I-\mathcal{K}$ is a bounded Fredholm operator: $ L^2(\mathbb{R})\to  L^2(\mathbb{R})$ for $y\in[0,+\infty)$ with index zero. Moreover, $(I-\mathcal{K})^{-1}$ exists and is also a bounded linear operator: $ L^2(\mathbb{R})\to  L^2(\mathbb{R})$. And   $\|\mathcal{K}^2\|\to 0$, as $y\to\infty$.
\end{Proposition}
\begin{proof}
	Note that  operators $\mathcal{K}^{(1)}_{2}$, $\mathcal{K}^{(2)}_{2}$ are both  finite-rank. So it suffices  to analyze the property of $\mathcal{K}^{(1)}_{1}\mathcal{K}^{(2)}_{1}$ and  $\mathcal{K}^{(2)}_{1}\mathcal{K}^{(1)}_{1}$.  Take $\mathcal{K}^{(2)}_{1}\mathcal{K}^{(1)}_{1}$ as an example. A simple calculation gives that $\|\mathcal{K}^{(2)}_{1}\mathcal{K}^{(1)}_{1}\|\leqslant1$  as $|r|\leqslant1$.
	Recall the definition of $X_\pm, $  that is,  $X_\pm$ is the characteristic function
	on $\mathbb{R}^\pm$. For any $f\in L^2(\mathbb{R})$,
	\begin{align*}
	{(\mathcal{K}^{(2)}_{1}\mathcal{K}^{(1)}_{1}\check{f})}^\wedge (z) =\int_{\mathbb{R}}K(z,\lambda;y)f(\lambda)d\lambda,
	\end{align*}
	where the integral kernel is
	\begin{align*}
		K(z,\lambda;y)=X_+(z)\int_{\mathbb{R}}X_-(s)	{(\bar{r}e^{-2\theta(\cdot;y)})}^\wedge(s-\lambda)(re^{2\theta(\cdot;y)})^\wedge(z-s)ds.
	\end{align*}
	Since   $$\|K(\cdot,\cdot; y)\|_{L^2(\mathbb{R}\otimes\mathbb{R})}\leqslant \|(\bar{r}e^{-2\theta(\cdot;y)})^\wedge\|_{L^{2,1}}\|(re^{2\theta(\cdot;y)})^\wedge\|_{L^{2,1}}\leqslant C(y)\|r\|_{W^{1,2}}^2$$ for a positive constant $C(y)$ which is finite for every $y\in[0,+\infty)$, it follows that $\mathcal{K}^{(2)}_{1}\mathcal{K}^{(1)}_{1}$ is the Hilbert-Schmidt operator. It in turn implies that $\mathcal{K}^2$ is compact. Therefore, it is inferred in  \cite{simon} that $I-\mathcal{K}^2$ is a Fredholm operator with index zero, and so is $I-\mathcal{K}$.
We now only need to prove that the operator $I-\mathcal{K}:  L^2\to  L^2$ is  injective. For convenience, take ${N_0}=1$. Suppose that  $\vec{f} =\left(f_1,f_2\right) ^T\in  L^2$ is the solution of the homogeneous equation: $(I-\mathcal{K})\vec{f}=0$.  Denote $\vec{g}(z)=(g_1^*(z),g_2(z))$ with
	\begin{align*}
		&g_1(z)=\mathcal{C}(-\bar{r}e^{-2\theta(\cdot;y)}f_2)(z)+\mathcal{K}^{(1)}_{2}(f_2)(z),\ g_2(z)=\mathcal{C}(re^{2\theta(\cdot;y)}f_1)(z)+\mathcal{K}^{(2)}_{2}((f_1)^T)(z),
	\end{align*}
	where $g_1(z)$ and $g_2(z)$ are analytic in $\mathbb{C}^-$ and  $\mathbb{C}^+$ respectively, and meromorphic in $\mathbb{C}\setminus\mathbb{R}$.  Thus, $[g_1]_-=f_1$, $[g_2]_+=f_2$. Using the result in the Cauchy-Goursat theorem gives that
	\begin{align*}
		0=\int_{\mathbb{R}^+}\vec{g}(z)[\vec{g}^*(z)]^Tdz.
	\end{align*}
	On the other hand
	\begin{align*}
		0=\int_{\mathbb{R}^+}\vec{g}(z)[\vec{g}^*(z)]^Tdz&=\int_{\mathbb{R}}|f_1|^2+|f_2|^2+2\text{Re}\left( f_1(z)f_2^*(z)r(z)e^{2\theta(z;y)}\right)dz\\
		&\geq \int_{\mathbb{R}}\left( |f_1(z)|^2+|f_2(z)|^2\right) (1-|r|)dz.
	\end{align*}
	Note that $1-|r(z)|> 0, z\not=\pm 1$, and $1-|r(z)|=0$ if and only if $z=\pm1$. So the above  inequality implies that   $f_1(z)=f_2(z)=0$ for $z\neq\pm1$, namely,  $f(z)=0, \ a. e. $
	So  $\|f\|_{L^2}=0$. The final step is to prove $\|\mathcal{K}^2\|\to 0$ as $y\to\infty$.
	Because $e^{y\text{Im}z_j}$ decays to zero  exponentially
	 as $y\to\infty$, it is sufficient to demonstrate
	  $\|\mathcal{K}^{(1)}_{1}\mathcal{K}^{(2)}_{1}\|$,  $\|\mathcal{K}^{(2)}_{1}\mathcal{K}^{(1)}_{1}\|\to0$ as $y\to\infty$.
	  Taking $\mathcal{K}^{(1)}_{1}\mathcal{K}^{(2)}_{1}$ to give details, like in \cite{ZHOU}, $\bar{r}$ can be approximated in $L^\infty$ norm by a series of functions $\{r_n\}_1^\infty$. Here,  $\bar{r}_n\in L^2$ is a retinal function with simple poles $\{\xi^{(n)}_j\}$ on $\mathbb{C}^+$.Then  $\mathcal{K}^{(1)}_{1}\mathcal{K}^{(2)}_{1}f$ is approximated in $L^2$  norm by $\mathcal{C}_-(-\bar{r}_ne^{-2\theta(\cdot;y)}\mathcal{C}_+(re^{2\theta(\cdot;y)}f))$ uniformly for $\|f\|_{L^2}\leqslant1$. We denote the residue of $r_n$ at $\xi^{(n)}_j$ by $r_{n,j}$.
Note that
\begin{align*}
		(-(\cdot-\xi^{(n)}_j)^{-1}e^{-2\theta(\cdot;y)})^\wedge(z)=\int_{\mathbb{R}}\frac{e^{-2\pi izs}e^{i(s-1/s)y/2}}{s-\xi^{(n)}_j}ds=X_-(z)e^{-2\pi iz\xi^{(n)}_j+i(\xi^{(n)}_j-1/\xi^{(n)}_j)y/2}.
\end{align*}
There appears that
\begin{align*}
	(-\bar{r}_ne^{-2\theta(\cdot;y)})^\wedge(z)=\sum_{\xi^{(n)}_j\in\mathbb{C}^+}X_-(z)\bar{r}_{n,j}e^{-2\pi iz\xi^{(n)}_j+i(\xi^{(n)}_j-1/\xi^{(n)}_j)y/2}.
\end{align*}
Therefore,
\begin{align*}
	&\left( 	\mathcal{C}_-(-\bar{r}_ne^{-2\theta(\cdot;y)}\mathcal{C}_+(re^{2\theta(\cdot;y)}f))\right) ^\wedge(z)=X_-(z)\left( \mathcal{C}_+(re^{2\theta(\cdot;y)}f)\right) ^\wedge*(-\bar{r}_ne^{-2\theta(\cdot;y)})^\wedge(z)\\
	&=\sum_{\xi^{(n)}_j\in\mathbb{C}^+} X_-(z)e^{i(\xi^{(n)}_j-1/\xi^{(n)}_j)y/2}\bar{r}_{n,j}\left( \mathcal{C}_+(re^{2\theta(\cdot;y)}f)\right) ^\wedge*(X_-e^{-2\pi i(\cdot)\xi^{(n)}_j})(z),
\end{align*}
approaches zero  in $L^2$ as $y\to\infty$ uniformly for $\|f\|_{L^2}\leqslant1$.
This completes the proof of Proposition \ref {lemmaG}.
\end{proof}
 Proposition \ref {lemmaG} also gives that $I-\mathcal{K}$ is uniformly bounded for $y\in[0,+\infty),$ and it  transpires that there exists a constant $C(\|R\|_{Y})$ such that
\begin{align}
	\parallel\left( I-\mathcal{K}\right) ^{-1}\parallel\leqslant C( \|R\|_{Y}).
\end{align}
As a consequence of this proposition, the solution of \eqref{N} exsits with
\begin{align}
	N(z)=((I-\mathcal{K})^{-1}(N_0))(z),
\end{align}
which leads to the following proposition immediately.
\begin{Proposition}\label{pro3.1}
		If   $R=\left( r,\{z_j,c_j\}_{j=1}^{2{N_0}}\right)  \in Y$, then  there  exists a unique  solution $ M^{(1)}_l $ of the RH problem \ref{RHP2} with
		\begin{align}
			M^{(1)}_l(z)=I+\mathcal{C}(\mu w)(z)+\sum_{j=1}^{2{N_0}} \left(\frac{\bar{c}_je^{y\text{Im}z_j}}{z-\bar{z}_j}\mu_2(\bar{z}_j),\frac{c_je^{y\text{Im}z_j}}{z-z_j}\mu_1(z_j)  \right),
		\end{align}
		where $\mu$ comes from the definition of $N$ in \eqref{dN}.
\end{Proposition}
Furthermore, it is adduced from \eqref{Nzj} that
\begin{align*}
	|	N_{11}(z_j;y)|=|\mathcal{C}(	N_{11})(z_j;y)|\leq C(\text{Im}z_j)\|N_{11}\|_{L^2},\hspace*{0.3cm} |N_{12}(\bar{z}_j;y)|=|\mathcal{C}(	N_{12})(\bar{z}_j;y)|\leq C(\text{Im}z_j)\|N_{12}\|_{L^2}.
\end{align*}
	Together with  \eqref{N}, it  gives that
	\begin{align}
		&\parallel \mu-I\parallel_{L^2}
		\leqslant \parallel\left( I-\mathcal{K}\right) ^{-1}\parallel( \parallel C_w I\parallel_{L^2}+\sum_{j=1}^{2{N_0}}c_je^{y\text{Im}z_j}\Big\|\frac{1}{z-z_j}  \Big\|_{L^2} ),\label{esmu}\\
		&\sum_{j=1}^{2{N_0}}|\mu_2(\bar{z}_j)|+\sum_{j=1}^{2{N_0}}|\mu_1(z_j)-1 |\leqslant C(\text{Im}z_j)\parallel\left( I-\mathcal{K}\right) ^{-1}\parallel( \parallel C_w I\parallel_{L^2}+\sum_{j=1}^{2{N_0}}c_je^{y\text{Im}z_j}\Big\|\frac{1}{z-z_j}  \Big\|_{L^2} ).
	\end{align}
	Meanwhile, there exists a positive constant $C( \|R\|_{Y})$ such that
	\begin{align}
	\sum_{j=1}^{2{N_0}}|\mu_2(\bar{z}_j)|+\sum_{j=1}^{2{N_0}}|\mu_1(z_j)-1 |+	\parallel\mu-I\parallel_{L^2}\leqslant C(\|R\|_{Y}) .
	\end{align}
\begin{remark}
	Compared with derivative  NLS equation \cite{pelinovskey1},  where the  corresponding jump matrix    admits
     positive definite part
 with reflection coefficient $\|r\|_{L^\infty}<1$.
 So the norm of their corresponding  $(I-\mathcal{K})^{-1}$ can be directly controlled by   $(1-\|r\|_{L^\infty})^{-1}$.
  However, for the mCH equation, its corresponding reflection coefficient $r(z)$   satisfies  $|r(\pm1)|=1$.
	 So  	we  lose the Lipschitz continuity of the solution $M^{(1)}_l$ with respect to reflection coefficient $r(z)$.
\end{remark}
When $y\in(-\infty, 0)$, following the same process as $M_l^{(1)}(z)$, $M_r^{(1)}(z)$ admits similar property in the following proposition.
\begin{Proposition}\label{Pro3.2}
	If    $R   \in Y$, $y\in(-\infty, 0)$, then there exists a positive constant $C(\|R\|_{Y})$, such that
	\begin{align}
		&\parallel M_r^{(1)}-I\parallel_{L^2}+\sum_{j=1}^{2{N_0}}|M_{r,2}^{(1)}(\bar{z}_j)|+\sum_{j=1}^{2{N_0}}|M_{r,1}^{(1)}(z_j) |\leqslant C(\|R\|_{Y}),
	\end{align}
where $M_{r,2}^{(1)}$ denotes the $j$-column of $M_{r}^{(1)}$.
\end{Proposition}

\subsection{Time evolution}\label{subsec3.3}

The crucial argument  of inverse scattering theory is  how to recover the solution $m(t,x)$  to  the initial-value problem  \eqref{mch1} from scattering data. It is necessary to give the time evolution of the RH problem   from  the time evolution of scattering data.
According to the time  spectral problem in the Lax pair  \eqref{lax0}  and  scattering relation \eqref{scattering}, it is found  that the   time evolution scattering data $a(z;t )$ and  $b(z;t )$ satisfy the equations
\begin{align}
	\partial_t a(z;t )=0,\ 	\partial_t b(z;t )=-\dfrac{2iz(z^2-1)}{z^2+1}b(z;t),\ z\in\mathbb{R},
\end{align}
which yields
\begin{align}
	a(z;t )=a(z;0 ),\ b(z;t )=e^{-\frac{2iz(z^2-1)t}{z^2+1}}b(z ;0),\ z\in\mathbb{R}.
\end{align}
Therefore, we define the time-dependent scattering data
in the following lemma.
\begin{lemma}\label{pror}
	Define time-dependent  scattering data  ${R}(t)=\left( r(\cdot;t),\{z_j(t),c_j(t)\}_{j=1}^{2{N_0}}\right)$ with
	\begin{align*}
		r(z;t )=\exp\left( {-\frac{2iz(z^2-1)t}{(z^2+1)^2}}\right) r( z ;0 ),\ z\in\mathbb{R},\ z_j(t)=z_j(0),\ c_j(t)=\exp{\left( \frac{t\text{Im}z_j}{\text{Re}^2z_j}\right) }c_j(0).\label{ct}
	\end{align*}
For any $ T > 0, $ if   $R( 0)\in Y$, then for every $t\in [0, T)$,
$R( t)\in Y$. Moreover,
	\begin{align}
		\|R( t)\|_{Y} \leqslant C(T)	\|R( 0)\|_{Y},
	\end{align}
	with a constant $C(T)>0.$
\end{lemma}
\begin{proof}
	Since   $  \big| e^{-\frac{2iz(z^2-1)t}{(z^2+1)^2}}\big|=1 $,  for $z\in \mathbb{R}$, $j=0,1,2$, we have
	$$\|(\cdot)^jr(\cdot;t) \|_{L^2} =  \|(\cdot)^j e^{-\frac{2i(\cdot)((\cdot)^2-1)t}{((\cdot)^2+1)^2}} r(\cdot;0)  \|_{L^2} =\| r(\cdot;0)  \|_{L^{2,j}}. $$
Note that  for $j=0,1,2$,
	\begin{align*}
		z^j \partial_z r (z;t)=2it	\frac{z^j(z^4-6z^2+1)}{(z^2+1)^3}  e^{-\frac{2iz(z^2-1)t}{(z^2+1)^2}} r(z;0 )+ e^{-\frac{2iz(z^2-1)t}{(z^2+1)^2}}\partial_z r (z;0 ).
	\end{align*}
It   is thereby inferred that
	\begin{align*}
		&\|(\cdot)^j \partial_z r (\cdot;t)\|_{L^2}^2\leqslant 2t\sup_{z\in \mathbb{R}} \frac{| z^j(z^4-6z^2+1)| } { (z^2+1)^3} \|r(\cdot;0 )\|_{L^2}^2+ \|\partial_z r(\cdot;0 )\|_{L^2}^2.
	\end{align*}
	In the same way   for $j=0,1$,
	\begin{align*}
		z^j \partial_z^2 r (z;t)=& (it\frac{z^{j+1}(z^4+2z^2-7)}{(z^2+1)^4}- \frac{t^2z^j(z^4-6z^2+1)^2}{(z^2+1)^6} ) 4e^{-\frac{2iz(z^2-1)t}{(z^2+1)^2}} r(z;0 )\\
		&+4it	\frac{z^j(z^4-6z^2+1)}{(z^2+1)^3} e^{-\frac{2iz(z^2-1)t}{(z^2+1)^2}}\partial_z r (z;0 )+z^je^{-\frac{2iz(z^2-1)t}{(z^2+1)^2}}\partial_z^2 r (z;0 ).
	\end{align*}
	Consequently,
	\begin{align*}
		\|(\cdot)^j \partial_z^2 r (\cdot;t)\|_{L^2}^2\leqslant& 4 ( t \sup_{z\in \mathbb{R}} \frac{|z^{j+1}(-z^4+2z^2+7)|}{(z^2+1)^4} +t^2\sup_{z\in \mathbb{R}} \frac{|z|^j(z^4-6z^2+1)}{(z^2+1)^3})  \|r(\cdot;0 )\|_{L^2}^2\\
		& +4t\sup_{z\in \mathbb{R}}	\frac{|z^j(z^4-6z^2+1)|}{(z^2+1)^3} \|\partial_zr(\cdot;0 )\|_{L^2}^2+\|(\cdot)^j\partial_z^2r(\cdot;0 )\|_{L^2}^2.
	\end{align*}
This in turn implies  that  $r(\cdot;t) \in H^{1,2}(\mathbb{R} )\cap H^{2,1}(\mathbb{R} )$.
 The estimate of others scattering data can be accomplished  in the same way from the fact that $|z_j|=1$ with Im$z_j<0$ and $\exp{\left( \frac{t\text{Im}z_j}{\text{Re}^2z_j}\right) }<1$, which implies the desired result as indicated above.
\end{proof}
Therefore, invoking of Propositions \ref{pro3.1} and \ref{Pro3.2},  under time-dependent  scattering data  the  RH problems \ref{RHP2} exists unique solution.
We denote $M^{(1)}(z;t,y)$ as the  solution for  the  RH problem \ref{RHP2} under time-dependent  scattering data  $R(t) $.
And  $M^{(1)}_r(z;0,y)$ also has  a similar  time evolution denoted by $M^{(1)}_r(z;t,y)$ , which is omitted here. Thus  the  time evolution of $M^{(1)}$ defined in \eqref{M1} is given by
\begin{align}
	M^{(1)}(z;t,y)=\left\{ \begin{array}{ll}
		M^{(1)}_l(z;t,y),   &\text{as } y\in [0,+\infty),\\[12pt]
		M^{(1)}_r(z;t,y) , &\text{as }y\in (-\infty,0],\\
	\end{array}\right.\label{M1t}
\end{align}
The crux of the matter is whether the function $m(t,y)$   reconstructed  from the time-depended RH problem in \eqref{rescm} is the solution of the mCH equation \eqref{mcho2}.
The approach  is  to  demonstrate that $	M^{(1)}(z;t,y)$ satisfies some Lax pair equations with similar structure
as \eqref{lax0}  whose compatibility condition is the mCH equation \eqref{mcho2}. To achieve this aim, we first give the following proposition.
\begin{Proposition}
	$M^{(1)}(z;t,y)$ defined in \eqref{M1t} is differentiable with respect to $t>0$.
\end{Proposition}
\begin{proof}
	We give the details of the proof of the case $y\in[0,+\infty)$. In this case,  under time-dependent  scattering data  $R(t) $,	$M^{(1)}(z;t,y)=M^{(1)}_l(z;t,y)$. From the definition of $M^{(1)}_l(z;t,y)$ in \eqref{intM}, the existence of $	\partial_t M^{(1)}_l(z;t,y)$ is equivalent to  the existence of $	\partial_t \mu(z;t,y)$. By taking $t$-derivative of the operator equation \eqref{N}  under time-dependent  scattering data  $R(t) $, it follows that
	 \begin{align}
	 	\partial_t N(z;t,y)=\left( 	\partial_tN_0(z;t,y)+	(\partial_t \mathcal{K})(N)(z;t,y)\right) +\mathcal{K}	(\partial_tN)(z;t,y),
	 \end{align}
 where
 \begin{align*}
 	&	\partial_t N(z;t,y)=\left(0,\mathcal{C}_+(	\partial_tr(\cdot;t)e^{2\theta})(z;t,y)+\sum_{j=1}^{2N_0}\frac{	\partial_tc_j(t)e^{\text{Im}z_jy}}{z-z_j} \right)^T,\\
 	&\partial_t\mathcal{K}=\left(\begin{array}{cc}
 		0  & \partial_t\mathcal{K}^{(1)}\\
 		\partial_t\mathcal{K}^{(2)}	& 0
 	\end{array}\right).
 \end{align*}
For brevity, we omit the concrete expression of operators  $\partial_t\mathcal{K}^{(1)}$ and
$\partial_t	\mathcal{K}^{(2)}$. From Lemma \ref{pror},
\begin{align*}
&\partial_tr(z;t)=-\frac{2iz(z^2-1)}{(z^2+1)^2}\exp\left( {-\frac{2iz(z^2-1)t}{(z^2+1)^2}}\right) r( z ;0 ),\ z\in\mathbb{R}\\ &\partial_tc_j(t)=\frac{\text{Im}z_j}{\text{Re}^2z_j}\exp{\left( \frac{t\text{Im}z_j}{\text{Re}^2z_j}\right) }c_j(0),
\end{align*}
with $|\partial_tr(z;t)|\lesssim|r( z ;0 )|$, $|\partial_tc_j(t)|\lesssim|c_j(0)|$. Therefore, $\partial_tN_0(\cdot;t,y)+	(\partial_t \mathcal{K})(N)(\cdot;t,y)\in L^2(\mathbb{R})$.
 Then it is inferred from Proposition \ref{lemmaG}  that $	\partial_t N(z;t,y)$ exists with  \begin{align*}
 	\partial_t N(z;t,y)=\left( I-\mathcal{K}\right) ^{-1}\left( 	\partial_tN_0(z;t,y)+	(\partial_t \mathcal{K})(N)(z;t,y)\right)(z;t,y) .
 \end{align*}
It also implies the   existence of $	\partial_t \mu(z;t,y)$,  from which we obtain the result of this proposition.
\end{proof}
Therefore, by standard processing (see for example \cite{Mon,RN9}), the following proposition  can be verified.
\begin{Proposition}
	Assume that $M^{(1)}(z;t,y)$ is the  solution for  the RH problem \ref{RHP2} under time-dependent  scattering data
$R(t) $.  The  matrix function $M(z;t,y)$ is defined  by
	\begin{align}
		\left( I-\frac{\sigma_1}{z}\right) M(z;t,y)=\left( I-\frac{\sigma_1M^{(1)}( t,y;0)^{-1}}{z} \right) M^{(1)}(z;t,y).\label{transM}
	\end{align}
	Define
	\begin{align}
		\breve{M}(z;t,y)=M(z;t,y)e^{-\breve{p}(z;t,y)\sigma_3},
	\end{align}
	where $\breve{p}(z;t,y)=-\frac{i(z^2-1)}{4z}(-y+\frac{8z^2}{(z^2+1)^2}t)$. Then $\breve{M}(z;t,y)$ satisfies the system of linear differential equations
	\begin{align}
		\breve{M}_y(z;t,y)=\breve{U}\breve{M},\ \ \  \breve{M}_t(z;t,y)=\breve{V}\breve{M}, \label{poe}
	\end{align}
	where
	\begin{align*}
		&\breve{U}=-\frac{i(z^2-1)}{4z}\sigma_3+ \frac{ih_1(t,y)}{z-1}\left(\begin{array}{cc}
			1 & -1 \\
			1 & -1
		\end{array}\right)+\frac{ih_1(t,y)}{z+1}\left(\begin{array}{cc}
			1 & 1 \\
			-1 & -1
		\end{array}\right)+ih_1(t,y)\left(\begin{array}{cc}
			0 & -1 \\
			1 & 0
		\end{array}\right),\\
		&\breve{V}=\frac{2i(z^2-1)z}{(z^2+1)^2}\sigma_3+\frac{ih_2(t,y)}{z-1}\left(\begin{array}{cc}
			1 & -1 \\
			1 & -1
		\end{array}\right)+\frac{ih_2(t,y)}{z+1}\left(\begin{array}{cc}
			1 & 1 \\
			-1 & -1
		\end{array}\right)\nonumber\\
		&+\frac{1}{z-1}\left(\begin{array}{cc}
			0 & h_3(t,y) \\
			h_4(t,y) & 0
		\end{array}\right)+\frac{1}{z+1}\left(\begin{array}{cc}
			0 & h_4(t,y) \\
			h_3(t,y) & 0
		\end{array}\right),
	\end{align*}
	\begin{align*}
		&h_1(t,y)=-\frac{1}{2}\lim_{z\to \infty}zM_{12}(z;t,y),\\
		&h_2(t,y)=-i\lim_{z\to 1}(z-1)(\breve{M}_t(z;t,y)^{-1}\breve{M}(z;t,y))_{11},\nonumber\\
		&h_3(t,y)=2M_{11}(i ;t,y)\lim_{z\to i}\frac{M_{12}(z;t,y)}{z-i},\nonumber\\
		& h_4(t,y)=-2M_{22}( i ;t,y)\lim_{z\to i}\frac{M_{21}(z;t,y)}{z-i},\nonumber
	\end{align*}
	and the subscript $ij$ used above means the $(i,j)$-entry element of matrix. The compatibility condition between two equations in the system  \eqref{poe}, namely,
	$$\breve{U}_t-\breve{V}_y+[\breve{U},\breve{V}]=0$$
	yields the mCH equation \eqref{mcho2} with the variables $(t,y)$.
\end{Proposition}
Consequently, under the time evolution of the scattering data in Lemma \ref{pror}, it can be  verified that the reconstruction formula
remain accessible for $M^{(1)}(z;t,y)$.
\subsection{Estimates on solutions of  the RH problem } \label{sec5}
In this subsection,  to estimate the  solution of the Cauchy problem \eqref{mch1} through reconstruction formula, a key step is to give the proof of  estimates of $M_l^{(1)}(z)$ for $y\in[0,+\infty)$ which is the solution of the RH problem \ref{RHP2} under scattering data   $R(t) $.   \eqref{intM} implies that
as $z\to i$, $M_l^{(1)}(z)$  has the following expansion.
\begin{align}
	M_l^{(1)}(z)=M_l^{(1)}(i)+M_l^{(1),i}(z-i)+\mathcal{O}\left( (z-i)^2\right),
\end{align}
with
\begin{align*}
		&M^{(1)}_l(i)=I+\mathcal{C}(\mu w)(i)+\sum_{j=1}^{2{N_0}} \left(\frac{\bar{c}_je^{y\text{Im}z_j}}{i-\bar{z}_j}\mu_2(\bar{z}_j),\frac{c_je^{y\text{Im}z_j}}{i-z_j}\mu_1(z_j)  \right) ,\\
		&M_l^{(1),i}= \frac{1}{2\pi i} \int_{\mathbb{R}}\dfrac{\mu(s)w(s)}{(s-i)^2}ds-\sum_{j=1}^{2{N_0}} \left(\frac{\bar{c}_je^{y\text{Im}z_j}}{(i-\bar{z}_j)^2}\mu_2(\bar{z}_j),\frac{c_je^{y\text{Im}z_j}}{(i-z_j)^2}\mu_1(z_j)  \right).
\end{align*}
\begin{Proposition}\label{lemmai}
	  If
$R \in Y,$  $y\in[0,+\infty)$, there exists a positive constant $C(\|R\|_{Y})$ such that $M_l^{(1)}(i)$,  $M_l^{(1),i}$ and $M_l^{(1)}(0)$ are estimated by
	\begin{align}
		|M_l^{(1)}(i)-1|,\ |M_l^{(1),i}|,\ |M_l^{(1)}(0)|\leqslant  C( \|R\|_{Y}) .\label{377}
	\end{align}
\end{Proposition}
\begin{proof}
	From \eqref{esmu}, for $z=0$,
$\sum_{j=1}^{2{N_0}}| \frac{\bar{c}_je^{y\text{Im}z_j}}{\bar{z}_j}\mu_2(\bar{z}_j)| \; {\rm and} \;  \sum_{j=1}^{2{N_0}}|\frac{c_je^{y\text{Im}z_j}}{z_j}\mu_1(z_j)  |
$ are bounded. And the integral term becomes
	\begin{align}
		|\mathcal{C}(\mu w)(0)|&= \Big| \mathcal{C}((\mu-I)w)+\mathcal{C}w(0)\Big|\leqslant \frac{1}{2\pi}\left(\|\mu-I\|_{L^2}\|w/z\|_{L^2}+\|w/z\|_{L^1} \right) .
	\end{align}
	It remains to analyze $\int_{\mathbb{R}} | s^{-1} w(s)  |ds$ and $\int_{\mathbb{R}}| s^{-1} w(s)  |^2ds$. In view of  the definition of $w$, $|r|\leqslant1$ and   the symmetry of $r$,  it is sufficient to estimate $\int_{0}^{+\infty}|s^{-1}r(s)| ds $ and $\int_{0}^{+\infty}|s^{-1}r(s)|^2ds $.  Note that  $r(z) \in H^{1,2}(\mathbb{R})$.  In order to estimate these two  integrations  near $s=0$, namely, $\int_{0}^{1}|s^{-1}r(s)| ds $ and $\int_{0}^{1}|s^{-1}r(s)|^2ds $, we use  the symmetry of $r$, that is, $ r(z) = \overline {r(1/z)}, $ and  change of the variable $ z \to 1/z $ to obtain
	\begin{align*}
		&\int_{0}^{1}|s^{-1}r(s)|ds=\int_{1}^{+\infty} s^{-1} |r(s)| ds\lesssim \parallel r \parallel_{L^{2}},\\
 &\int_{0}^{1}|s^{-1}r(s)|^2ds=\int_{1}^{+\infty}|r(s)|^2ds\leqslant \parallel r \parallel_{L^{2}}^2.
	\end{align*}
Hence the estimate of $M_l^{(1)}(0)$ is obtained immediately. For $M_l^{(1)}(i)$ and $M_l^{(1),i}$, applying the H\"older inequality and the estimates in  \eqref{esmu} gives  the result in \eqref{377}.
\end{proof}
Now take derivative of the inhomogeneous equation \eqref{N} in $y$ to obtain
\begin{align}
	\partial_y N(z)=&\partial_yN_0(z)+\partial_y(\mathcal{K})N(z) +\mathcal{K} (\partial_yN)(z)\label{Nx1},\\
	\partial_y^2 N(z)=&\partial_y^2N_0(z)+2\partial_y(\mathcal{K})(\partial_yN)(z)+\partial_y^2(\mathcal{K}) N(z) +\mathcal{K} (\partial_y^2N)(z).
\end{align}
From the definition of $\mathcal{K}$ and $N_0$, it  transpires  that
\begin{align}
	&	\partial_yN_0(z)=\left(0,\mathcal{C}_+(-\frac{i}{2}(s-1/s)re^{2\theta})(z)+\sum_{j =0}^{2N_0}\frac{\text{Im}z_jc_je^{\text{Im}z_jy}}{z-z_j} \right)^T,\nonumber\\
	&\partial_y^2N_0(z)=\left(0,\mathcal{C}_+(-\frac{1}{4}(s-1/s)^2re^{2\theta})(z)+\sum_{j =0}^{2N_0}\frac{\text{Im}^2z_jc_je^{\text{Im}z_jy}}{z-z_j} \right)^T, \nonumber\\
		&	\partial_y\mathcal{K}=\left(\begin{array}{cc}
		0  & \partial_y\mathcal{K}^{(1)}\\
	\partial_y	\mathcal{K}^{(2)}	& 0
	\end{array}\right),\ \partial_y^2\mathcal{K}=\left(\begin{array}{cc}
	0  & \partial_y^2\mathcal{K}^{(1)}\\
	\partial_y^2		\mathcal{K}^{(2)}		& 0
\end{array}\right).\nonumber
\end{align}
For brevity, we omit the concrete expression of operators  $\partial_y\mathcal{K}^{(1)}$,
$\partial_y	\mathcal{K}^{(2)}$, $\partial_y^2\mathcal{K}^{(1)}$ and
$\partial_y^2	\mathcal{K}^{(2)}$. From $r(z) \in H^{1,2}(\mathbb{R} )\cap H^{2,1}(\mathbb{R} )$, $\partial_yK$ and $\partial_y^2K$ are also the bounded operators on $ L^2(\mathbb{R})$.  Thus,  $\partial_yN_0+(\partial_yK)N$  belongs to $ L^2(\mathbb{R})$, which implies that
\begin{align*}
		&\partial_yN(z)=\left( I-\mathcal{K}\right) ^{-1}( \partial_yN_0+(\partial_y\mathcal{K})N)(z)\in L^2(\mathbb{R}).
\end{align*}
Namely there exists  a positive constant $C(\|R\|_{Y})$ such that
\begin{align}
	\|\partial_y\mu\|_{L^2} \leqslant   C( \|R\|_{Y})\label{muy}.
\end{align}
So  $\partial_y^2N_0+2(\partial_y\mathcal{K})( \partial_yN)+(\partial_y^2\mathcal{K}) N$ also  belongs to $ L^2(\mathbb{R})$ such that
\begin{align}
 &\partial_y^2N=\left( I-\mathcal{K}\right) ^{-1}(\partial_y^2N_0+2(\partial_y\mathcal{K})( \partial_yN)+(\partial_y^2\mathcal{K}) N)\in L^2(\mathbb{R}).
\end{align}
Therefore there exists  a positive constant $C(\|R\|_{Y})$ such that
\begin{align}
	\|\partial_y^2\mu\|_{L^2}\leqslant  C( \|R\|_{Y})\label{muyy}.
\end{align}
We now consider the function $M_l^{(1)}$. A direct computation reveals that
 \begin{align*}
 	&\partial_yM^{(1)}_l(z)=\mathcal{C}(\partial_y\mu\  w)(z)+\mathcal{C}(\mu \partial_yw)(z)\\
 	&+\sum_{j=1}^{2{N_0}} \left(\frac{\bar{c}_je^{y\text{Im}z_j}}{z-\bar{z}_j}(\text{Im}z_j\mu_2(\bar{z}_j)+\partial_y\mu_2(\bar{z}_j)),\frac{c_j\text{Im}z_je^{y\text{Im}z_j}}{z-z_j}(\text{Im}z_j\mu_1(z_j)+\partial_y\mu_1(z_j) ) \right) ,\\
 	&\partial_y^2M^{(1)}_l(z)=\mathcal{C}(\partial_y^2\mu\  w)(z)+\mathcal{C} (\partial_y\mu \partial_yw)(z)+\mathcal{C}(\mu \partial_y^2w)(z)\\
 	&+\sum_{j=1}^{2{N_0}} \left(\frac{\bar{c}_je^{y\text{Im}z_j}}{z-\bar{z}_j}(\text{Im}z_j^2\mu_2(\bar{z}_j)+2\partial_y\text{Im}z_j\mu_2(\bar{z}_j)+\partial_y^2\mu_2(\bar{z}_j)),0\right) \\
 	&+\sum_{j=1}^{2{N_0}} \left(0,\frac{c_j\text{Im}z_je^{y\text{Im}z_j}}{z-z_j}(\text{Im}^2z_j\mu_1(z_j)+2\text{Im}z_j\partial_y\mu_1(z_j)+\partial_y^2\mu_1(z_j) ) \right) .
 \end{align*}
 As $z\to i$, they have the asymptotic expansions  in the following,
\begin{align}
	&\partial_yM^{(1)}_l(z)=\partial_yM^{(1)}_l(i)+\partial_yM^{(1),i}_l(z-i)+\mathcal{O}\left( (z-i)^2\right) ,\\
	&\partial_y^2M^{(1)}_l(z)=\partial_y^2M^{(1)}_l(i)+\partial_y^2M^{(1),i}_l(z-i)+\mathcal{O}\left( (z-i)^2\right) ,
\end{align}
with
\begin{align*}
	&\partial_yM^{(1),i}_l= \frac{1}{2\pi i} \int_{\mathbb{R}}\dfrac{\partial_y\mu(s) w(s)}{(s-i)^2}ds+\frac{1}{2\pi i}\int_{\mathbb{R}}\dfrac{\left( \mu(s)-I\right)  \partial_yw(s)}{(s-i)^2}ds+\frac{1}{2\pi i}\int_{\mathbb{R}}\dfrac{ \partial_yw(s)}{(s-i)^2}ds\nonumber\\
	&-\sum_{j=1}^{2{N_0}} \left(\frac{\bar{c}_je^{y\text{Im}z_j}}{(i-\bar{z}_j)^2}(\text{Im}z_j\mu_2(\bar{z}_j)+\partial_y\mu_2(\bar{z}_j)),\frac{c_j\text{Im}z_je^{y\text{Im}z_j}}{(i-z_j)^2}(\text{Im}z_j\mu_1(z_j)+\partial_y\mu_1(z_j) ) \right) ,\\
	&\partial_y^2M^{(1),i}_l=\frac{1}{2\pi i} \int_{\mathbb{R}}\dfrac{\partial_y^2\mu(s) w(s)}{(s-i)^2}ds+\frac{1}{2\pi i} \int_{\mathbb{R}}\dfrac{\partial_y\mu(s) \partial_yw(s)}{(s-i)^2}ds+\frac{1}{2\pi i} \int_{\mathbb{R}}\dfrac{\left( \mu(s)-I\right)  \partial_y^2w(s)}{(s-i)^2}ds\\
	&+\frac{1}{2\pi i} \int_{\mathbb{R}}\dfrac{ \partial_y^2w(s)}{(s-i)^2}ds-\sum_{j=1}^{2{N_0}} \left(\frac{\bar{c}_je^{y\text{Im}z_j}}{(i-\bar{z}_j)^2}(\text{Im}z_j^2\mu_2(\bar{z}_j)+2\partial_y\text{Im}z_j\mu_2(\bar{z}_j)+\partial_y^2\mu_2(\bar{z}_j)),0\right) \\
	&-\sum_{j=1}^{2{N_0}} \left(0,\frac{c_j\text{Im}z_je^{y\text{Im}z_j}}{(i-z_j)^2}(\text{Im}^2z_j\mu_1(z_j)+2\text{Im}z_j\partial_y\mu_1(z_j)+\partial_y^2\mu_1(z_j) ) \right) .
\end{align*}
Otherwise, the equation \eqref{intM} gives that for $|z|\to\infty$,
\begin{align}
	&\lim_{z\to \infty}z( 	M^{(1)}_l(z)-I) =-\frac{1}{2\pi i} \int_{\mathbb{R}}\left( \mu(s)-I\right)  w(s)ds-\frac{1}{2\pi i} \int_{\mathbb{R}} w(s)ds\nonumber\\
	&+\sum_{j=1}^{2{N_0}} \left(\bar{c}_je^{y\text{Im}z_j}\mu_2(\bar{z}_j),c_je^{y\text{Im}z_j}\mu_1(z_j)  \right)\\
	&\lim_{z\to \infty}z\partial_yM^{(1)}_l(z)=-\frac{1}{2\pi i}\int_{\mathbb{R}}\partial_y\mu(s) w(s)ds-\frac{1}{2\pi i}\int_{\mathbb{R}}\left( \mu(s)-I\right)  \partial_yw(s)ds-\frac{1}{2\pi i}\int_{\mathbb{R}} \partial_yw(s)ds\nonumber\\
	&+\sum_{j=1}^{2{N_0}} \left(\bar{c}_je^{y\text{Im}z_j}(\text{Im}z_j\mu_2(\bar{z}_j)+\partial_y\mu_2(\bar{z}_j)),c_j\text{Im}z_je^{y\text{Im}z_j}(\text{Im}z_j\mu_1(z_j)+\partial_y\mu_1(z_j) ) \right).
\end{align}
Similar to the proof in Lemma \ref{lemm3.2}, combining the H\"older inequality, 	\eqref{esmu}, \eqref{muy} and \eqref{muyy}, it transpires that for $j=1,2,$
	\begin{align}
		|\partial_y^jM^{(1)}_l(i)|,\ |\partial_y^jM^{(1),i}_l|,\ |\partial_y^jM^{(1)}_l(0)|,\ 	|\lim_{z\to \infty}z 	M^{(1)}_{l,12}(z)|,\ |\lim_{z\to \infty} z\partial_yM^{(1)}_{l,12}(z)|\leqslant C\left(\|R\|_{Y}\right),\label{LemmaN2}
	\end{align}
where the subscript $12$ means the $(1,2)$-entry element of matrix.

When $y\in(-\infty, 0)$, $M_r^{(1)}(z)$ has a similar property, which we omit its proof.
\begin{Proposition}\label{Pro1}
	Supposed that
$R  \in Y,$ $y\in(-\infty, 0)$.
Then as $z\to i$, $M^{(1)}_r$ has the following asymptotic expansion,
\begin{align*}
	M^{(1)}_r(z)=M^{(1)}_r(i)+M^{(1),i}_r(z-i)+\mathcal{O}\left( (z-i)^2\right) ,
\end{align*}
and there exists a positive constant $C(\|R\|_{Y})$  such that  for $j=0,1,2$, $\partial_y^jM^{(1)}_r(i)$,  $\partial_y^jM^{(1),i}_r$,  $\partial_y^jM^{(1)}_r(0)$,  $ \lim_{z\to \infty}z	M^{(1)}_{r,12}(z)$ and  $\lim_{z\to \infty} z\partial_yM^{(1)}_{r,12}(z)$ are bounded, that is,
	\begin{align*}
		&|\partial_y^jM^{(1)}_r(i)|,\ |\partial_y^jM^{(1),i}_r|,\ |\partial_y^jM^{(1)}_r(0)|,|\lim_{z\to \infty}z	M^{(1)}_{r,12}(z)|,
		|\lim_{z\to \infty} z\partial_yM^{(1)}_{r,12}(z)|\leqslant C( \|R\|_{Y}).
	\end{align*}
\end{Proposition}

\section {Existence of global solutions}\label{sec6}
In view of the result in Section \ref{sec3}, we  use $M^{(1)}_l(z;t,y)$ to recover the solution $m(t,x(y,t))$ for  $y\in[0,+\infty)$, and   $M^{(1)}_r(z;t,y)$ to recover the solution $m(t,x(y,t))$ for $y\in(-\infty, 0)$ through equations \eqref{res1}-\eqref{res2}. The main aim of this section is to give the proof of the following lemma.
\begin{lemma}\label{the1}
	Assume that  the initial data  $m_0>0 $ and $ m_0-1 \in H^{2,1}(\mathbb{R} )\cap H^{1, 2} (\mathbb{R})$.
	Then  there exists a unique  solution $m(t)  \in C([0, +\infty);  W^{1,\infty}(\mathbb{R}))$  for  the Cauchy  problem \eqref{mch1}.
\end{lemma}
The proof of Lemma \ref{the1} is approached  via a series of  propositions and lemmas.  The estimates of the functions established in the following lemma will be  used  to reconstruct solution in \eqref{res1}.
\begin{lemma}\label{lemma4.2}
 Let $ R(0) \in Y. $ For  every $t\in [0, T), \ T> 0$,  the functions $\eta(t,y)$ and $\alpha_j(t,y), j = 1, 2, 3$ defined in \eqref{rescm} with the scattering data $R(t)$ in Lemma \ref{pror},  \eqref{al1}-\eqref{al2} respectively satisfy
	\begin{align*}
		&\|\alpha_1^{-1}(t,\cdot)\|_{L^\infty},\ \|\alpha_1(t,\cdot)-1\|_{W^{2,\infty}},\ \|\alpha_j(t,\cdot)\|_{W^{2,\infty}}\leqslant C(T,\|R(0)\|_{Y}),\ j=2,3,   \quad {\rm and}\\
		&\|\eta(t,\cdot) \|_{W^{1,\infty}}\leqslant C(T,\|R(0)\|_{Y}).
	\end{align*}
Moreover, $$	|u(t,y)-1|\leqslant C(T,\|R(0)\|_{Y}).$$

\end{lemma}
\begin{proof}
	 As  shown in Lemma \ref{pror}, for any $t\in [0, T),$	
	 $R( t)\in Y$ with $ \|R( t)\|_{Y} \leqslant C(T,	\|R( 0)\|_{Y}).$
In view of  Propositions  \ref{lemmaG} and  \ref{Pro3.2}, under the scattering data ${R}(t)$, the RH problem \ref {RHP2}  has a unique solution $M^{(1)}_l(z;t,y)$ for every  $t\in [0, T)$. Except for $\|1/\alpha_1(t,\cdot)\|_{L^\infty}$, the other estimates can be obtained from  Propositions \ref{lemmai} and \ref{Pro1} and \eqref{LemmaN2} immediately. Now  consider the boundedness of $\|1/\alpha_1(t,\cdot)\|_{L^\infty}$.
	in fact, when $\alpha\neq-1$, apply Proposition \ref{lemmai} to give the boundedness of 	$f_j(t,y)$ and $\frac{ \beta}{\alpha+1}f_j(t,y)$, $j=1,2$.  This implies that $\frac{ \beta}{\alpha+1}$ is also bounded. On the other hand, det$M^{(1)}\equiv1$, namely, $\alpha_1 f_2(1+\frac{ \beta}{\alpha+1})=1$. Therefore, $1/\alpha_1=f_2(1+\frac{ \beta}{\alpha+1})$ is bounded. And when  $\alpha=1$, the boundedness of  $1/\alpha_1$ is a directly result from the boundedness of $f_1$. This completes the proof of Lemma \ref{lemma4.2}.
\end{proof}
It then follows from  Lemma \ref{pror},  Lemma \ref{lemma4.2}, and the reconstruction formulas \eqref{rescm},
 \eqref{res1} and \eqref{res2} that there exist   $u(t,y)$ and  $\eta(t,y)$ as the functions of $y$ in  $W^{1,\infty}(\mathbb{R})$ for  $t\in [0, T), $ with any $ T> 0.$
However,  the reconstruction formula of $	m$ in \eqref{rescm} is
\begin{align}
m(t,y)=(1-\eta(t,y) )^{-1}.
\end{align}
So it is hard to obtain the boundedness of $m$ from the reconstruction formula directly. Therefore, our attention is focused on  the definition of $	m$ in \eqref{mcho2} with $m=u-u_{xx}$. Note that $\partial_x=(1+2(\partial_y\alpha_1)/\alpha_1)\partial_y$, so
\begin{align*}
 u_x=&(1+2(\partial_y\alpha_1)/\alpha_1)\left(  -\partial_y(\alpha_2)\alpha_1-\alpha_1\partial_y\alpha_2-\partial_y(\alpha_3)\alpha_1^{-1}+\partial_y(\alpha_1)\alpha_3\alpha_1^{-2}\right)  ,\\
 u_{xx}=&2(\alpha_1\partial_y^2\alpha_1-\partial_y\alpha_1^2)\alpha_1^{-2}u_x+( 1+2\partial_y(\alpha_1)\alpha_1^{-1}) ^2\left(  -\partial_y^2(\alpha_2)\alpha_1-2\partial_y\alpha_1\partial_y\alpha_2-\partial_y^2(\alpha_1)\alpha_2\right) \\
&+\alpha_1^{-4}( \alpha_1+2\partial_y\alpha_1)  ^2
\left(  -\alpha_1\partial_y^2\alpha_3+2\partial_y\alpha_1\partial_y\alpha_3+\partial_y^2(\alpha_1)\alpha_3-2\partial_y(\alpha_1)^2\alpha_3\alpha_1^{-1}\right).
\end{align*}
Similarly as the above analysis, applying  \eqref{LemmaN2} and  Proposition  \ref{Pro1}  gives the boundedness of $u_{xx}$. In fact, for every $T\in\mathbb{R}^+$, $t\in[0,T)$,  under the scattering data $R( t)$ in Lemma \ref{pror},
 $u_x(t,y),$ $u_{xx}(t,y)$ defined above are bounded respectively. In addition,  their $L^\infty$-norm can be controlled by  $  C(T,\|R(0)\|_{Y})$ with
\begin{align}
	&|u_x(t,y)|,\ |u_{xx}(t,y)|\leqslant C(\|R(t)\|_{Y})\leqslant C(T,\|R(0)\|_{Y}),
\end{align}
Therefore, it is inferred that  for every  $T >0,$ $t\in[0,T)$,  $m(t,y)$ exists as a bounded function of $y$ with
\begin{align}
	&|m(t,y)-1| \leqslant C(\|R(t)\|_{Y})\leqslant C(T,\|R(0)\|_{Y}).\label{lemmam}
\end{align}
Together with Lemma \ref{lemma4.2}, it appears that for every $t\in [0, T)$, $m(t,\cdot)\in W^{1,\infty}$.

The final step of the proof of Lemma \ref{the1} is to verify the continuity of $\eta(t,y)$ with respective to $t$. Analogously, we take the case $y\in[0,+\infty)$ as an example.  The first step is to analyze the continuity about $t$ of the RH problem.
\begin{lemma}\label{contG}
Both of the operator $\mathcal{K}(t)$ defined in \eqref{K} corresponding to scattering data $ R( t  ) $ and $\left( I-\mathcal{K}(t)\right) ^{-1}$ are continuous at $t \in [0, T),$ for any $T >0. $
\end{lemma}
\begin{proof}
From the definition of $\mathcal{K}(t)$ in \eqref{K},  it is sufficient to give the details of the proof of $\mathcal{K}^{(1)}_{11}(t)$ and $\mathcal{K}^{(1)}_{12}(t)$ when $y\in[0,\infty)$. The others can be obtained analogously. For any   $t_0\in\mathbb{R}^+$ and  any $f\in L^2(\mathbb{R} )$,
\begin{align*}
	\|\mathcal{K}^{(1)}_{11}(t)f-\mathcal{K}^{(1)}_{11}(t_0)f\|_{L^2 }
	=&\Big\| \mathcal{C}_-( e^{-2\theta}f\bar{r}(z;0)e^{-\frac{2iz(z^2-1)t_0}{(z^2+1)^2}}( 1-e^{-\frac{2iz(z^2-1)(t-t_0)}{(z^2+1)^2}})  ) \Big\|_{L^2 }\\
	&\leqslant\big\|f\big\|_{L^2 } \big\|e^{\frac{2iz(z^2-1)(t-t_0)}{(z^2+1)^2}}-1\big\|_{L^\infty }	\leqslant C\big\|f\big\|_{L^2 }|t-t_0|.
\end{align*}
So $\mathcal{K}^{(1)}_{11}(t)$ is continuous on  $t \in [0, T)$. And for vector $(f_1,...,f_{2{N_0}})^T\in\mathbb{C}^{2{N_0}}$,
\begin{align}
	&\|	\mathcal{K}^{(1)}_{12}(t)(f_1,...,f_{2{N_0}})^T-\mathcal{K}^{(1)}_{12}(t_0)(f_1,...,f_{2{N_0}})^T	\|_{L^2}\nonumber\\
	&= \Big\|\sum_{j=1}^{2{N_0}}\frac{f_j\bar{c}_je^{y\text{Im}z_j}}{z-\bar{z}_j}e^{-\frac{2iz(z^2-1)t_0}{(z^2+1)^2}}( 1-e^{-\frac{2iz(z^2-1)(t-t_0)}{(z^2+1)^2}}) \Big\|_{L^2}\nonumber\\
	&\leqslant\sum_{j=1}^{2{N_0}}\Big\|\frac{f_j\bar{c}_je^{y\text{Im}z_j}}{z-\bar{z}_j}\Big\|_{L^2}\Big\| 1-e^{-\frac{2iz(z^2-1)(t-t_0)}{(z^2+1)^2}}\Big\|_{L^\infty }\leqslant C({z_j}_1^{2{N_0}})|t-t_0|\sum_{j=1}^{2{N_0}}|f_j|.
\end{align}
Since  the map: $\mathcal{K}^{(1)}\to(\mathcal{K}^{(1)})^{-1}$ is a homeomorphism in the Banach algebra,  $\left( I-\mathcal{K}(t)\right) ^{-1}$ is also continuous on $t \in [0, T)$.
\end{proof}

\begin{Proposition}\label{contm}
Suppose that  $M^{(1)}(z;t,y)$ is the solution of the  RH problem \ref{RHP2} with scattering
data $R( t  ) \in   Y$. Then  $M^{(1)}(z;t,y)$  is continuous on $t \in [0, T), $ for any $ T> 0$.
\end{Proposition}
\begin{proof}
Note that $\left( I-\mathcal{K}(t)\right) ^{-1}$ and $N_0(t)$ are both continuous on $t \in [0, T)$. Thus continuity of the function $N(z;t,y)$ on $t \in [0, T)$ follows from the expression
\begin{align*}
	N(z;t)=\left( I-\mathcal{K}(t)\right) ^{-1}(N_0(t))(z).
\end{align*}
From the definition of $N(z;t,y)$ in \eqref{dN} and
\begin{align*}
	M^{(1)}_l(z)=I+\mathcal{C}(\mu w)(z)+\sum_{j=1}^{2{N_0}} \left(\frac{\bar{c}_je^{y\text{Im}z_j}}{z-\bar{z}_j}\mu_2(\bar{z}_j),\frac{c_je^{y\text{Im}z_j}}{z-z_j}\mu_1(z_j)  \right),
\end{align*}
continuity of $	M^{(1)}_l(z;t,y)$ on $t \in [0, T)$ is immediately accomplished. This completes the proof of Proposition \ref{contm}.
\end{proof}
We are now in the position to prove Lemma \ref{the1}.
\begin {proof} [Proof of Lemma  \ref{the1}]
Since the initial data  $m_0>0\ $ with $ m_0 -1 \in H^{2,1}(\mathbb{R} )\cap H^{1, 2} (\mathbb{R})$,  applying Lemma \ref{r1} gives that the scattering data $R( 0  ) \in   Y$. Thus Lemma \ref{pror} also reveals that after time evolution,  the scattering data $R(t) \in Y$. It is then adduced from Lemma \ref{lemma4.2} and
 the estimate in \eqref{lemmam}  that   $m(t,\cdot)\in W^{1,\infty},$ for any $t\in [0, T),  \,  T > 0$.
Moreover,  it is deduced from Proposition \ref{contm} that  $M^{(1)}(z;t,y)$  is   continuous on $t \in [0, T)$. In consequence,
it follows  from the reconstruction formula in \eqref{rescm} that $\eta(t) $ and $ m(t)\in W^{1,\infty}(\mathbb{R})$ are  continuous on $t \in [0, T)$. This completes the proof of  Lemma  \ref{the1}.
\end {proof}

\section{Regularity}\label{sec7}
\quad The aim of  this section is to derive  the regularity of the   solution $m$ and show that
$m -1  \in C ([0, +\infty);  H^{2,1}(\mathbb{R} )\cap H^{1,2} (\mathbb{R} ) )$.

\subsection{Transitions  of the spectral problem}\label{subseclax2}
The first step is to establish  a new RH problem by using  the Lax pair \eqref{lax0}.  When $t=0$, consider two column function vectors
\begin{align}
	& \tilde{\Phi}_1=\left(\begin{array}{cc}
		1 & 0   \\[4pt]
		-\frac{m_0-1}{m_0} & -2\lambda
	\end{array}\right)\Psi_{1},\hspace*{0.8cm}	 \tilde{\Phi}_2=\left(\begin{array}{cc}
		-2\lambda  &  	-\frac{m_0-1}{m_0}   \\[4pt]
		0 & 1
	\end{array}\right)\Psi_{2}.
\end{align}
Then a straightforward computation
shows that they satisfy  two new spectral problems
\begin{align}
	 \tilde{\Phi}_{j,x}=- \frac{1}{4} ik \, m_0 \, \sigma_3 \tilde{\Phi}_j+P_j \tilde{\Phi}_j,\ j=1,2,
\end{align}
with
\begin{align*}
	&P_1=\frac{i(m_0-1)}{2k }\left(\begin{array}{cc}
		\eta-2 & 1   \\[4pt]
		-\eta^2-\frac{4}{m_0} & -\eta+2
	\end{array}\right) -\left(\begin{array}{cc}
		0 & 0   \\[4pt]
		\eta_x  & 0
	\end{array}\right),\\[5pt]
	&P_2=\frac{i(m_0-1)}{2k}\left(\begin{array}{cc}
		\eta -2 &  \eta^2+\frac{4}{m_0}  \\[4pt]
		-1 & -\eta+2
	\end{array}\right) -\left(\begin{array}{cc}
		0 & \eta_x   \\[4pt]
		0  & 0
	\end{array}\right),
\end{align*}
where  $\eta(0,y) =	\frac{m_0-1}{m_0} $.
Furthermore, define the functions $ \nu_j^\pm(x;k),  j = 1, 2 $ by the  transformations
\begin{align}
	\nu_j^\pm(x;k)=e^{(-1)^{j-1} p(x;k)} \tilde{\Phi}_j(x;k),
\end{align}
which  satisfy  the following Volterra  integral equations,
\begin{align}
	&	\nu_1^\pm(x;k)=e_1+\int_{\pm \infty}^{x}\left(\begin{array}{cc}
		1 & 0  \\[4pt]
		0  & e^{2 (p(x;k)-p(s;k))}
	\end{array}\right)P_1(s,k)\nu_1^\pm(s,k)ds,\\
	&\nu_2^\pm(x,k)=e_2+\int_{\pm \infty}^{x}\left(\begin{array}{cc}
		e^{-2 (p(x;k)-p(s;k))} & 0  \\[4pt]
		0  & 1
	\end{array}\right)P_2(s;k)\nu_2^\pm(s,k)ds.
\end{align}
 It then transpires from a direct calculation  that
\begin{Proposition}\label{lemma5.1}
	If  initial data  $m_0-1\in  H^{2,1}(\mathbb{R} )\cap H^{1, 2} (\mathbb{R}) $,     then  the  Jost functions $	 \nu_j^\pm(x;k), j = 1, 2 $
	admit  the following asymptotics
	\begin{align}
		&\nu_1^\pm(x;k)=e_1+ \mathcal{O}( k^{-1}),\quad |k|\to\infty,\\
		&\nu_2^\pm(x;k)=e_2+\left(\begin{array}{cc}
			2i\frac{\p_xm_0}{m_0^3}  \\[4pt]
			b_\pm(x)
		\end{array}\right)k^{-1} +\mathcal{O}( k^{-2}),\quad |k|\to\infty,
	\end{align}
	with
	$$b_\pm(x)=\frac{i}{2}\int_{\pm\infty}^x\frac{m_0(s)^2-1}{m_0(s)}ds,$$
	and they satisfy the following limits along a contour in the domains of their
	analyticity
	\begin{align}
		&	\nu_1^\pm(x,k)=-\frac{i\alpha_\pm}{k}\left(\begin{array}{cc}
			1 \\[4pt]
			2-\eta
		\end{array}\right)+\mathcal{O}(1),\quad k \to 0,\\
		&\nu_2^\pm(x,k)=\frac{i\alpha_\pm}{k}\left(\begin{array}{cc}
			2-\eta  \\[4pt]
			1
		\end{array}\right)+\mathcal{O}(1),\quad k  \to 0.
	\end{align}
	
\end{Proposition}
\subsection{A new RH problem}
We now define
\begin{align}
	&\frac{\nu^-_1}{a^*(z)}-\nu^+_1=r_2e^{ik y/2}p^l_+ ,\hspace{0.3cm}\frac{p^l_-}{a}-p^l_+=r_1e^{-ik y/2}\nu^+_2,\\
	&\nu^-_1-\frac{\nu^+_1}{a^*(z)}=\tilde{r}_2e^{ik y/2}p^r_-,\hspace{0.3cm}\frac{p^r_+}{a}=\frac{\nu^+_1}{a^*(z)}r_1e^{-ik y/2}\nu^+_2+(1-r_1r_2)p^r_-,
\end{align}
where
	\begin{align}
			&r_1(k)=\frac{z}{z^2+1}r(z),\ 	r_2(k)=\frac{z^2+1}{z}\overline{r(z)},\ r_1(k)r_2(k)=|r(z)|^2,\\
			&\tilde{r}_1(k)=\frac{z}{z^2+1}\tilde{r}(z),\ 	\tilde{r}_2(k)=\frac{z^2+1}{z}\overline{\tilde{r}(z)},\ \tilde{r}_1(k)\tilde{r}_2(k)=|r(z)|^2.
	\end{align}
Obviously, $p_\pm(k;y)$ are both well defined under variable $k=z-1/z$.
By using  these two column vectors, we construct  two functions by
\begin{align}
	&M^{(2)}_l(k;y)=\left\{ \begin{array}{ll}
		\left(\begin{array}{cc}
			\eta-2, & 1
		\end{array}\right)\left(\begin{array}{cc}\frac{ \nu_{1}^1 (k) } {a^*(k)}, & p^l_{+}(k)\end{array}\right),   &\text{as } k\in \mathbb{C}^+,\\[12pt]
		\left(\begin{array}{cc}
			\eta-2, & 1
		\end{array}\right)\left(\begin{array}{cc}  \nu_{1}^+(k),&\frac{ p^l_{-}(k)}{a(k)}\end{array}\right) , &\text{as }k\in \mathbb{C}^-,\\
	\end{array}\right. \ y\in\mathbb{R}^+,\label{M21}\\[12pt]
&M^{(2)}_r(k;y)=\left\{ \begin{array}{ll}
	\left(\begin{array}{cc}
		\eta-2, & 1
	\end{array}\right)\left(\begin{array}{cc}\nu_{1}^1 (k), & \frac{ p_{+}^r(k) } {a^*(k)}\end{array}\right),   &\text{as } k\in \mathbb{C}^+,\\[12pt]
	\left(\begin{array}{cc}
		\eta-2, & 1
	\end{array}\right)\left(\begin{array}{cc} \frac{ \nu_{1}^+(k)}{a(k)},&p_{-}^r(k)\end{array}\right)  , &\text{as }k\in \mathbb{C}^-,\\
\end{array}\right. \ y\in\mathbb{R}^-.\label{M22}
\end{align}
It is easy to see that these two functions satisfy the following  two RH problems, respectively.
\begin{RHP}\label{RHP7}
	Find a vector-valued function $M_l^{(2)}(k)=M^{(2)}_l(k;y)$ which satisfies
	
	$\bullet$ Analyticity: $M_l^{(2)}(k)$ is meromorphic in $\mathbb{C}\setminus \mathbb{R}$;

	$\bullet$ Jump condition: $M_l^{(2)}(k)$ has continuous boundary values $[M_l^{(2)}]_\pm(k)$ on $\mathbb{R}$ and
	\begin{equation}
		[M^{(2)}_l]_+(k)=[M^{(2)}_l]_-(k)\widetilde{V}_l(k),\hspace{0.5cm}	\widetilde{V}_l(k)=\left(\begin{array}{cc}
			1-|r|^2 & r_1e^{-2\theta}\\
			-r_2e^{2\theta} & 1
		\end{array}\right);  \nonumber
	\end{equation}

	$\bullet$ Asymptotic behavior:
	\begin{align}
		&M^{(2)}_l(k) =  (	\eta-2, 1 )+\mathcal{O}(k^{-1}),\hspace{0.5cm}|k| \rightarrow \infty\ in\ \mathbb{C};
	\end{align}

	$\bullet$ Residue condition: $M^{(2)}_l(k)$ has simple poles at each point in $ \{k_j,\bar{k}_j\}_{j=1}^{{N_0}} $ with
	\begin{align}
		&\res_{k=k_j}M^{(2)}_l(k)=\lim_{k\to k_j}M^{(2)}_l(k)\left(\begin{array}{cc}
			0 & 	d_je^{2y\text{Im}z_j}\\
		0 & 0
		\end{array}\right),\\
		&\res_{k=\bar{k}_j}M^{(2)}_l(k)=\lim_{k\to \bar{k}_j}M^{(2)}_l(k)\left(\begin{array}{cc}
			0 & 0\\
			\tilde{d}_je^{2y\text{Im}z_j} & 0
		\end{array}\right).
	\end{align}
\end{RHP}
 Here, $k_j=z_j-1/z_j$, $d_j=-\bar{z}_jc_j$ and $\tilde{d}_j=(2\text{Re}z_j)^2\bar{d}_j$ for $j=1,...,{N_0}$. For $j=1,...,{N_0}$, the  zeros $z_j$ and $z_{j+{N_0}}$  of $a$ in the $z$-plane  condense to one point $k_j$ in the $k$-plane. 
\begin{RHP} \label{rhp5}
	Find a vector-valued function $M_r^{(2)}(k)=M^{(2)}_r(k;y)$ which satisfies
	
	$\bullet$ Analyticity: $M_r^{(2)}(k)$ is meromorphic  in $\mathbb{C}\setminus \mathbb{R}$;

	$\bullet$ Jump condition: $M_r^{(2)}(k)$ has continuous boundary values $[M_r^{(2)}]_\pm(k)$ on $\mathbb{R}$ and
	\begin{equation}
		[M^{(2)}_r]_+(k)=[M^{(2)}_r]_-(k)\widetilde{V}_r(k),\hspace{0.5cm} \widetilde{V}_r(k)=\left(\begin{array}{cc}
			1 & \tilde{r}_1e^{-2\theta}\\
			-\tilde{r}_2e^{2\theta} & 1-|r|^2
		\end{array}\right);
	\end{equation}

	$\bullet$ Asymptotic behavior:
	\begin{align}
		&M^{(2)}_r(k) =  (	\eta-2, 1 )+\mathcal{O}(k^{-1}),\hspace{0.5cm}|k| \rightarrow \infty\ in\ \mathbb{C};
	\end{align}

	$\bullet$ Residue condition: $M^{(2)}_r(k)$ has simple poles at each point in $ \{k_j,\bar{k}_j\}_{j=1}^{{N_0}} $ with
	\begin{align}
		&\res_{k=k_j}M^{(2)}_r(k)=\lim_{k\to k_j}M^{(2)}_r(k)\left(\begin{array}{cc}
			0 & 0	\\
	\bar{d}_je^{-2y\text{Im}z_j} 	& 0
		\end{array}\right),\\
		&\res_{k=\bar{k}_j}M^{(2)}_r(k)=\lim_{k\to \bar{k}_j}M^{(2)}_r(k)\left(\begin{array}{cc}
			0 & \tilde{\bar{d}}_je^{-2y\text{Im}z_j}\\
			0 & 0
		\end{array}\right).
	\end{align}
\end{RHP}

In view of  the asymptotic expression of the Jost function in Proposition \ref{lemma5.1}, construction of $M^{(2)}_l(k)$ and $M^{(2)}_r(k)$ in \eqref{M21}-\eqref{M22}, it is thereby inferred that
\begin{align}
	&\frac{2im_x}{m^3}=\lim_{k\to \infty} k (M^{(2)}_1-\eta+2) ,\hspace{0.3cm} b_+=\frac{i}{2}\int_{+\infty}^x\frac{m(s)^2-1}{m(s)}ds=\lim_{k\to \infty} k(M^{(2)}_2-1), \label{res3}
\end{align}
where $M^{(2)}=M_l^{(2)}$ for $y\in\mathbb{R}^+$ and $M^{(2)}=M_r^{(2)}$ for $y\in\mathbb{R}^-$ with $ M^{(2)}_j$ as the $j$-element of  $M^{(2)}$.
Similar to   Section \ref{sec3}, when $y=0$, it is found that
\begin{align}
	\lim_{k\to \infty}[M^{(2)}_l(k;0)-(\eta(0)-2,1)]=	\lim_{k\to \infty}[M^{(2)}_r(k;0)-(\eta(0)-2,1)].
\end{align}
It is our purpose here to analyze the property of $M^{(2)}$. We only give the proof of case $y\in\mathbb{R}^+$, since the estimate of $y\in\mathbb{R}^-$ could be obtained using the same method. To see this, we have the following lemma derived from $r\in H^{1,2}(\mathbb{R} )\cap H^{2,1}(\mathbb{R} )$ and the symmetry of $r$.
\begin{lemma}\label{lemmar1}
	If  $r\in H^{1,2}(\mathbb{R} )\cap H^{2,1}(\mathbb{R} )$,  then
	$r_1(k)\in H^{1,3}(\mathbb{R})\cap H^{2,2}(\mathbb{R} )$,  $r_2(k)\in W^{2,2}(\mathbb{R})\cap  H^{1,1}(\mathbb{R})$,   and  there exists a  positive constant $C$ such that
	\begin{align}
		\parallel r_1 \parallel_{H^{1,3}\cap H^{2,2}},\ \parallel r_2 \parallel_{W^{2,2}\cap  H^{1,1}} \leqslant C\parallel r \parallel_{H^{1,2}\cap H^{2,1}}.
	\end{align}
\end{lemma}
For $k\in\mathbb{R}$, denote
\begin{align}
	&\tilde{w}=\tilde{w}_++\tilde{w}_-,\  \tilde{w}_+(k)=\left(\begin{array}{cc}
		0  & 0\\
		-r_2(k)e^{2\theta}	& 0
	\end{array}\right),\ \tilde{w}_-(k)=\left(\begin{array}{cc}
		0  & r_1(k)e^{-2\theta}\\
		0	& 0
	\end{array}\right),
\end{align}
and  $\mathcal{C}^k$ is the Cauchy operator on the $k$-plane.
Then  the RH problem \ref{RHP7} has a solution given by
\begin{align}
	M^{(2)}_l(k)= (	\eta-2, 1 )+\mathcal{C}^k(\tilde{\mu} \tilde{w})(k)+\sum_{j=1}^{{N_0}} \left(\frac{\tilde{d}_je^{y\text{Im}z_j}}{k-\bar{k}_j}\tilde{\mu}_2(\bar{k}_j),\frac{d_je^{y\text{Im}z_j}}{k-k_j}\tilde{\mu}_1(k_j)  \right),
\end{align}
if and only if there is a solution   $\tilde{\mu}$ of  the Fredholm integral equation  in the $k$-plane
\begin{align}
	\tilde{\mu}(k)= (	\eta-2, 1 )+\mathcal{C}^k_{\tilde{w}}	\tilde{\mu}(k) +\sum_{j=1}^{{N_0}} \left(\frac{\tilde{d}_je^{y\text{Im}z_j}}{k-\bar{k}_j}\tilde{\mu}_2(\bar{k}_j),\frac{d_je^{y\text{Im}z_j}}{k-k_j}\tilde{\mu}_1(k_j)  \right),
\end{align}
with $$\mathcal{C}^k_{\tilde{w}}f(z)=\mathcal{C}^k_+(f\tilde{w}_-)(z)+\mathcal{C}_-^k(f\tilde{w}_+)(z).$$
Introduce a  column vector function
\begin{align*}
	\tilde{N}(k)=&\left( \tilde{\mu}_{1}(k)-(\eta-2),
 \tilde{\mu}_{2}(k)-1 \right)^T ,
\end{align*}
which  satisfies
 the following equation,
\begin{align}
	\tilde{N}(k)=& \tilde{N}_0(k)+\tilde{\mathcal{K}}\tilde{N}(k),\hspace{0.3cm} \tilde{\mathcal{K}}=\left(\begin{array}{cc}
		0  & \tilde{\mathcal{K}}^{(1)}\\
		\tilde{\mathcal{K}}^{(2)}	& 0
	\end{array}\right),\label{tidN}\\
	\tilde{N}_0(k)=&\left((\eta-2)\mathcal{C}_-^k(-r_2e^{-2\theta})(z)+\sum_{j=1}^{N_0}(\eta-2)\frac{\tilde{d}_je^{\text{Im}z_jy}}{k-\bar{k}_j}, \mathcal{C}_+^k(r_1e^{2\theta})(z)+\sum_{j=1}^{N_0}\frac{d_je^{\text{Im}z_jy}}{k-k_j} \right)^T.\label{N0}
\end{align}
where the definitions of $\tilde{\mathcal{K}^{(1)}}$ and $\tilde{\mathcal{K}^{(2)}}$ are  similar to \eqref{A11}-\eqref{B22}.  We now verify that $	\tilde{N}\in L^2(\mathbb{R})$. In fact,
 it is shown  in Proposition \ref{lemmaG} that   $I-\mathcal{K}$ and $(I-\mathcal{K})^{-1}$ are two bounded linear operators on: $L^2(\mathbb{R} )\to L^2(\mathbb{R} )$ of the $z$-plane.  It suffices to derive  the estimate of bound for the  integral operators on the $k$-plane.

 In the following proposition, the superscript for the $z$-plane and the $k$-plane is used on  two  Cauchy projection operators with $\mathcal{C}_\pm^k$, $\mathcal{C}^k$ and $\mathcal{C}_\pm^z$, $\mathcal{C}^z$.
\begin{Proposition}\label{lemmaeo}
	If $f\in L^p(\mathbb{R} )$, $1\leqslant p<\infty$ and $f(z)=f(-\frac{1}{z})$, then
	\begin{align}
		\mathcal{C}_\pm^{z,(e)}f(z)=\mathcal{C}_\pm^kf(k),
	\end{align}
	where
	\begin{align}
		\mathcal{C}_\pm^{z,(e)}f(z)=\frac{1}{2\pi i}\left(\int_{-\infty}^{-1}+\int_{-1}^{0}+\int_{0}^{1}+\int_{1}^{\infty} \right) \frac{f(s)ds}{s-(z\pm0i)}.
	\end{align}
	If $g\in L^p(\mathbb{R} )$, $1\leqslant p<\infty$ and $g(z)=-g(-\frac{1}{z})$, then
	\begin{align}
		\mathcal{C}_\pm^{z,(o)}g(z)=\mathcal{C}_\pm^kg(k),
	\end{align}
	where
	\begin{align}
		\mathcal{C}_\pm^{z,(o)}g(z)=\frac{1}{2\pi i}\left(\int_{-1}^{-\infty}+\int_{-1}^{0}+\int_{0}^{1}+\int_{\infty}^{1} \right) \frac{g(s)ds}{s-(z\pm0i)}.
	\end{align}
	In addition,
	\begin{align}
		&\left( 	z+\frac{1}{z}\right) \mathcal{C}_\pm^{z,(e)}f (z)=\mathcal{C}_\pm^{z,(e)} \left(( 	z+\frac{1}{z})f \right)(z),\\	&\left( 	z+\frac{1}{z}\right)\mathcal{C}_\pm^{z,(o)}g (z)=	\mathcal{C}_\pm^{z,(o)} \left(( 	z+\frac{1}{z})g\right)(z).
	\end{align}
\end{Proposition}
This proposition is a directly consequence of
\begin{align*}
	\frac{1}{s-z}+\dfrac{1}{s^2}\cdot\frac{1}{-\frac{1}{s}-z}=\frac{1+\frac{1}{s^2}}{s-\frac{1}{s}-z+\frac{1}{z}}.
\end{align*}
The following proposition about $\tilde{\mathcal{K}}$ is deduced in the same way as Proposition \ref{lemmaG}.
\begin{Proposition}\label{pro5.3}
	The operator $I-\tilde{\mathcal{K}}$ is invertible on $L^2(\mathbb{R})$ and $(I-\tilde{\mathcal{K}})^{-1}$ is also a bounded operator.
	 Moreover, there exists a positive constant $C( \|R(0)\|_{Y})$  such that
	\begin{align}
		\|(I-\tilde{\mathcal{K}})^{-1}\|\leqslant C( \|R(0)\|_{Y}).
	\end{align}
 Therefore,  $\tilde{N}$ is solvable by
 \begin{align}
 	\tilde{N}(k)=(I-\tilde{\mathcal{K}})^{-1}(\tilde{N}_0)(k),
 \end{align}
with
\begin{align}
	\|\tilde{N}\|_{L^2}\leqslant C(\|R(0)\|_{Y}\|\tilde{N}_0\|_{L^2}.
\end{align}
\end{Proposition}
\begin{proof}
	Same as in Proposition \ref{lemmaG}, $\tilde{\mathcal{K}}^2$ is a compact operator, and $I-\tilde{\mathcal{K}}$ is injective.
	In fact, if a vector function	
$$\tilde{f}(k)=(f_1(k),f_2(k))$$
 is a solution of equation $\tilde{\mathcal{K}}f=f$, it then follows from Proposition \ref{lemmaeo}  that
 $$f(z)=((z+1/z)^{-1}f_1(z-1/z),f_2(z-1/z))$$
  is a solution of $\mathcal{K}f=f$. Thus the kernel of $\tilde{\mathcal{K}}$ must be trivial. So it is inferred that $(I-\tilde{\mathcal{K}})^{-1}$ exists as a  bounded operator.
	To obtain $\|\tilde{\mathcal{K}}^2\|\to0$ as $y\to\infty$, similarly it only needs to analyze the integral kernel of  $\mathcal{F}\tilde{K}^{(2)}_{1}\tilde{K}_{1}^{(1)}\mathcal{F}^{-1}$: for $f\in L^2(\mathbb{R})$,
	 \begin{align*}
	 	(\tilde{K}^{(2)}_{1}\tilde{K}_{1}^{(1)}\check{f})^\wedge (z) =\int_{\mathbb{R}}\tilde{K}(z,\lambda)f(\lambda)d\lambda,
	 \end{align*}
	 where
	 \begin{align*}
	 	\tilde{K}(z,\lambda)=X_+(z)\int_{-y}^{-\infty}	\hat{r}_2(z-s)\hat{r}_1(s-\lambda)ds.
	 \end{align*}
 Then it is concluded that as $y\to\infty$,
 \begin{align*}
 	\|	\tilde{K}\|_{L^2\times L^2}\lesssim \langle y\rangle ^{-1}\|	\hat{r}_2\|_{L^{2}}\|	\hat{r}_1\|_{L^{2,2}}\to 0.
 \end{align*}
	This completes the proof of Proposition \ref{pro5.3}.	
\end{proof}
In view of Proposition \ref{pro5.3}, it  also gives solvability of  $\tilde{\mu}$ for the RH problem \ref{RHP7}.
\begin{lemma}\label{corm3}
	For $R \in   Y$, $y\in\mathbb{R}^+$, there exists  a  unique  solution for the RH problem \ref{RHP7} with
	\begin{align}
		M^{(2)}_l(k)=(\eta-2,1)+\mathcal{C}^k(\tilde{\mu} \tilde{w})(k)+\sum_{j=1}^{{N_0}} \left(\frac{\tilde{d}_je^{y\text{Im}z_j}}{k-\bar{k}_j}\tilde{\mu}_2(\bar{k}_j),\frac{d_je^{y\text{Im}z_j}}{k-k_j}\tilde{\mu}_1(k_j)  \right) .
	\end{align}
	Moreover,
	\begin{align*}
		\parallel 	\tilde{\mu}-(\eta-2,1) \parallel_{L^2 }+\sum_{j=1}^{{N_0}}\big|\tilde{\mu}_1(k_j)-(\eta-2)\big|+\sum_{j=1}^{{N_0}}\big|\tilde{\mu}_2(\bar{k}_j)-1\big|\leqslant C(\|R\|_{Y}) \|\tilde{N}_0\|_{L^2}.
	\end{align*}

\end{lemma}
The  estimates in the following lemma   are decisively  arriving at the $y$-integrability.
\begin{lemma}\label{lemma13}
	For any $f,g \in W^{1,2}(\mathbb{R})$, it appears that
	\begin{align}
		&\sup_{y\in\mathbb{R}^+} \parallel  \mathcal{C}_\pm(fe^{\mp i\cdot y/2})(\cdot)\parallel_{L^\infty}\leqslant C 	\parallel f  \parallel_{W^{1,2}},\\
		&\sup_{y\in\mathbb{R}^+} \parallel \langle y\rangle \mathcal{C}_\pm(fe^{\mp i\cdot y/2})(\cdot)\parallel_{L^2}\leqslant C	\parallel f   \parallel_{W^{1,2}},\\
		&\sup_{y\in\mathbb{R}^+} \langle y\rangle^2\Big\vert \int_{\mathbb{R}}f(s)e^{\pm is y/2}\mathcal{C}_\pm(-ge^{\mp ik y/2}) (s) ds\Big\vert\leqslant C 	\parallel g  \parallel_{W^{1,2}}\parallel f  \parallel_{W^{1,2}},\label{513}
		\end{align}
	where $\langle y\rangle=(1+y^2)^{1/2}$ and $C$ is a positive constant. Moreover, when $f,g\in W^{2,2}$
		\begin{align}
		&\sup_{y\in\mathbb{R}^+} \parallel \langle y\rangle^2 \mathcal{C}_\pm(fe^{i\cdot y/2})(\cdot)\parallel_{L^2}\leqslant C	\parallel f   \parallel_{W^{2,2}},\\
		&\sup_{y\in\mathbb{R}^+} \langle y\rangle^4\Big\vert \int_{\mathbb{R}^+}g(s)e^{-is y/2}\mathcal{C}_+(-fe^{ik y/2})(s) ds\Big\vert\leqslant C 	\parallel f  \parallel_{W^{2,2}}\parallel g \parallel_{W^{2,2}},
	\end{align}
\end{lemma}
\begin{proof}
	Take $\mathcal{C}_-$ as an example. It is noted  that
	\begin{align*}
		\mathcal{C}_-(fe^{ik y/2})(k)=- (X_-(fe^{i\cdot y/2})^\wedge)^\vee (k).
	\end{align*}
So
	\begin{align*}
		\parallel  \mathcal{C}_\pm(fe^{\mp i\cdot y/2})(\cdot)\parallel_{L^\infty}\leqslant \parallel  X_\pm(fe^{i\cdot y/2})^\wedge\parallel _{L^1} \leqslant \parallel\hat{f}\parallel _{L^{2,1}}.
	\end{align*}
On the other hand,
\begin{align*}
		&\parallel  \mathcal{C}_\pm(fe^{\mp i\cdot y/2})(\cdot)\parallel_{L^2}=\| X_\pm(fe^{i\cdot y/2})^\wedge\parallel_{L^2}=\left( \int_{-\infty}^0\big|\hat{f}(z-\frac{y}{4\pi})\big|^2dz\right) ^{1/2}\nonumber\\
		&=\left( \int_{-\infty}^{-\frac{y}{4\pi}}\big|\hat{f}(z)\big|^2dz\right) ^{1/2}\lesssim \langle y\rangle^{-1}\left( \int_{-\infty}^{-\frac{y}{4\pi}}\big|\langle z\rangle\hat{f}(z)\big|^2dz\right) ^{1/2}\leqslant \langle y\rangle^{-1} \parallel f \parallel_{W^{1,2}} .
\end{align*}
	The last step  follows from that $\langle y\rangle\lesssim \langle s\rangle$ for $y\geq 0$. And for \eqref{513}, we have
	\begin{align*}
		 &\int_{\mathbb{R}}f(s)e^{is y/2}\mathcal{C}_+(-ge^{-ik y/2} )(s) ds=-\int_{\mathbb{R}}\int_{\mathbb{R}}f(s)e^{is y/2}  X_+(k)\hat{g}(k+\frac{y}{4\pi})e^{2\pi iks}dkds\\
		 &=-\int_{\frac{y}{4\pi}}^{+\infty}\check{f}(k)\hat{g}(k)ds.
	\end{align*}
Then \eqref{513} is accomplished in the same calculation.
	 And when $f\in W^{2,2}$, the above inequality becomes
	\begin{align*}
		&\parallel  \mathcal{C}_\pm(fe^{\mp i\cdot y/2})(\cdot)\parallel_{L^2}\lesssim \langle y\rangle^{-2}\left( \int_{-\infty}^{-\frac{y}{4\pi}}\big|\langle z\rangle^2\hat{f}(z)\big|^2dz\right) ^{1/2}\leqslant \langle y\rangle^{-2} \parallel f \parallel_{W^{2,2}},\\
		&\big|\int_{\mathbb{R}}f(s)e^{is y/2}\mathcal{C}_+(-ge^{-ik y/2})(s)ds\big|\leqslant \langle y\rangle^{-4} \parallel f \parallel_{W^{2,2}}\parallel g \parallel_{W^{2,2}},
	\end{align*}
	thereby concluding the proof of Lemma \ref{lemma13}.
\end{proof}
Applying Lemma \ref{lemma13} finally leads to for $j=0,1,2$,
\begin{align*}
	\langle y\rangle^{j}	\|\tilde{N}_0\|_{L^2}\leqslant C( \|R\|_{Y}).
\end{align*}
Combining with Lemma \ref{corm3}, it is deduced that
\begin{align}
		\langle y\rangle^{j}\left[ \parallel 	\tilde{\mu}-(\eta-2,1) \parallel_{L^2 }+\sum_{j=1}^{{N_0}}\big|\tilde{\mu}_1(k_j)-(\eta-2)\big|+\sum_{j=1}^{{N_0}}\big|\tilde{\mu}_2(\bar{k}_j)-1\big|\right] \leqslant C( \|R\|_{Y}).\label{lemma14}
\end{align}
Consider the $y$-derivative of $	M^{(2)}_l$. Invoking Lemma \ref{corm3}, it is adduced that
\begin{align*}
\partial_y	M^{(2)}_l(k)=&	(\partial_y\eta,0)+\mathcal{C}^k(\partial_y\tilde{\mu} \tilde{w})(k)+\mathcal{C}^k(\tilde{\mu} \partial_y\tilde{w})(k)+\sum_{j=1}^{{N_0}} \left(\frac{\tilde{d}_je^{y\text{Im}z_j}}{k-\bar{k}_j}\partial_y\tilde{\mu}_2(\bar{k}_j),\frac{d_je^{y\text{Im}z_j}}{k-k_j}\partial_y\tilde{\mu}_1(k_j)  \right)\nonumber\\
&+\sum_{j=1}^{{N_0}} \left(\frac{\text{Im}z_j\tilde{d}_je^{y\text{Im}z_j}}{k-\bar{k}_j}\tilde{\mu}_2(\bar{k}_j),\frac{\text{Im}z_jd_je^{y\text{Im}z_j}}{k-k_j}\tilde{\mu}_1(k_j)  \right),
\end{align*}
where  $\partial_y\tilde{\mu}$ satisfies  that
\begin{align*}
	\partial_y\tilde{N}(k)=&\left( \partial_y\tilde{\mu}_{1}(k)-\partial_y\eta,\partial_y\tilde{\mu}_{1}(k_1)-\partial_y\eta,...,\partial_y\tilde{\mu}_{1}(k_{{N_0}})-\partial_y\eta\right) ^T\nonumber\\
	&\otimes\left(  \partial_y\tilde{\mu}_{2}(k),\partial_y\tilde{\mu}_{2}(\bar{k}_1),...,\partial_y\tilde{\mu}_{2}(\bar{k}_{{N_0}}) \right)^T ,\\
	\partial_y\tilde{N}(k)=& \partial_y\tilde{N}_0(k)+(\partial_y\tilde{K})\tilde{N}(k)+\tilde{K}(\partial_y\tilde{N})(k),\hspace{0.3cm} \partial_y\tilde{K}=\left(\begin{array}{cc}
		0  & \partial_y\tilde{A}\\
	\partial_y	\tilde{B}	& 0
	\end{array}\right),\\
	\partial_y\tilde{N}_0(k)=&\mathcal{C}_-^k( -(\frac{ik(\eta-2)}{2}+\partial_y\eta)r_2e^{-2\theta}) (z)\\
	&\otimes\left(((\eta-2)\text{Im}z_1+\partial_y\eta)\frac{\tilde{d}_1e^{\text{Im}z_1y}}{k-\bar{k}_1},...,((\eta-2)\text{Im}z_{{N_0}}+\partial_y\eta)\frac{\tilde{d}_{{N_0}}e^{\text{Im}z_{{N_0}}y}}{k-\bar{k}_{{N_0}}} \right)^T\nonumber\\
	&\otimes\left(\mathcal{C}_+^k(-ikr_1e^{2\theta})(z),\frac{\text{Im}z_1d_1e^{\text{Im}z_1y}}{k-k_1},...,\frac{\text{Im}z_{{N_0}}d_{{N_0}}e^{\text{Im}z_{{N_0}}y}}{k-k_{{N_0}}} \right)^T.
\end{align*}
It in turn  implies  from $\|\partial_y\tilde{K}\|\lesssim  \|R(0)\|_{Y}$ that $\partial_y\tilde{N}_0+(\partial_y\tilde{K})\tilde{N} \in L^2 $ and
\begin{align*}
\partial_y\tilde{N}(k)=&\left( I-\tilde{K}\right)^{-1}(\partial_y\tilde{N}_0+(\partial_y\tilde{K})\tilde{N} )(k),
\end{align*}
where
\begin{align*}
		\langle y\rangle^{j}\|\partial_y\tilde{N}_0+(\partial_y\tilde{K})\tilde{N}\|_{L^2}\leqslant C(  \|R\|_{Y}),
\end{align*}
for $j=0,1$.
Therefore, same as the above estimates,   the following norm of the variable $k$ is bounded with
\begin{align}
		\langle y\rangle^{j} (\parallel 	\partial_y\tilde{\mu}-(\partial_y\eta,0) \parallel_{L^2 }+\sum_{j=1}^{{N_0}}\big|\partial_y\tilde{\mu}_1(k_j)\partial_y\eta\big|+\sum_{j=1}^{{N_0}}\big|\partial_y\tilde{\mu}_2(\bar{k}_j)\big|) \leqslant C( \|R\|_{Y}).
\end{align}
And for $y\in\mathbb{R}^-$, $M^{(2)}_r$ admits the  properties same as $M^{(2)}_l$.

  Because the estimate in the above section is under the scaling transformation of $y$, the  equivalency between the integral norm in $\mathbb{R}_x$ and  $\mathbb{R}_y$ is first given to arrive at the $L^2$-estimate with respect to  the variable $x$.
\begin{lemma}\label{lemmaytox}
	For any function $h(x(\cdot))\in L^p(\mathbb{R}_y)$, $1\leqslant p\leqslant \infty$, there appears that $h\in L^p(\mathbb{R}_x)$.
\end{lemma}
\begin{proof}
	When $p=\infty$, the result is obvious. On the other hand, it follows from \eqref{res2} that $dx=\left( 1+ 2\frac{ \partial_ y\alpha_1}{\alpha_1}\right) dy.$
	For $1\leqslant p<\infty$, from Lemma \ref{lemma4.2} it is ascertained that
	\begin{align*}
		y+2-2\|1/\alpha_1\|_{L^\infty}\leqslant x\leqslant y+2\|\alpha_1\|_{L^\infty}.
	\end{align*}
	Therefore,
	\begin{align}
		\int_{\mathbb{R}} |h(x)	|dx\leqslant 	\int_{\mathbb{R}} |h(x(y))	|\left( 1+ \Big|2\frac{ \frac{\partial }{\partial y}\alpha_1}{\alpha_1}\Big|\right) dy<\infty.\label{eew}
	\end{align}
\end{proof}
We now establish the  $L^2$-estimate on  the solution  $m$   by using the reconstruction formula in \eqref{res3}.

\begin{lemma}\label{lemma55}
  Assume that $ m_0(x)  > 0, \forall  x\in \mathbb{R} $ such that $m_0 -1\in  H^{2,1}(\mathbb{R})\cap H^{1, 2} (\mathbb{R})$. Then for every $T > 0, t\in[0,T)$,
 we have the following  estimates
	\begin{align}
		&\parallel m_x \parallel_{L^{2,2}(\mathbb{R})},\ \parallel m_{xx} \parallel_{L^{2,1}(\mathbb{R})},\  \parallel m-1\parallel_{L^{2,2}(\mathbb{R})}  \leqslant
 C(T,\|R(0)\|_{Y}). \label{wwdj}
	\end{align}
\end{lemma}

\begin{proof}
It is known from  Lemma \ref{lemmaytox}  that the estimates in \eqref{wwdj} are equivalent to
  those of integral norms  on the variable $y$.
 Applying Lemma \ref{r1}  and  Lemma \ref{pror} gives  that    $R( t ) \in   Y$ for every $t \in [0, T)$.    Then
  from the RH problem \ref{RHP7} and RH problem \ref{rhp5},  their  solutions $M^{(2)}_l(z;t,y)$   on  $y\in\mathbb{R}^+$ and $M^{(2)}_r(z;t,y)$ on $y\in\mathbb{R}^-$
    exist respectively.
    Resembling the procedure in Section \ref{subsec3.3}, the reconstruction formula \eqref{res3} is also practicable for every $t\in [0, T)$.

    The first step is to prove $m_x,\ m_{xx},\ b_+\in L^{2,2}(\mathbb{R})$. We only give the details of the case $L^{2,2}(\mathbb{R}^+),$ since the case  $y\in\mathbb{R}^-$ can be shown in a similar way.
For  $y\in\mathbb{R}^+$,
by using Lemma \ref{corm3} and    \eqref{res3},  it holds  that
	\begin{align}
		m_x&=\lim_{k\to \infty}k (M_{l,1}^{(2)}-(\eta-2))\nonumber  \\
&= \frac{1}{2\pi i} \int_{\mathbb{R}} (\tilde{\mu}_2(s)-1 ) r_2(s)e^{2\theta}ds
+ \frac{1}{2\pi i} \int_{\mathbb{R}}   r_2(s)e^{2\theta}ds
 +\sum_{j=1}^{{N_0}}\tilde{d}_je^{y\text{Im}z_j}\tilde{\mu}_2(\bar{k}_j).\label{resc1}
	\end{align}
For the first integral, it is inferred from \eqref{tidN} that
	\begin{align*}
		&\int_{\mathbb{R}} (\tilde{\mu}_2(s)-1) r_2(s)e^{2\theta}ds  =  \int_{\mathbb{R}}r_2(s)e^{is y/2}\mathcal{C}_+(-r_1e^{ik y/2}) (s) ds\\
		&+\int_{\mathbb{R}}r_2(s)e^{is y/2}\mathcal{C}_+((\tilde{\mu}_2(s;y)-1)r_1e^{-ik y/2})(s)  ds +\int_{\mathbb{R}}\frac{d_je^{y\text{Im}z_j}}{s-k_j}\tilde{\mu}_1(k_j)r_2(k)e^{2\theta}ds \\
		&\triangleq I_1(y)+I_2(y)+I_3(y).
	\end{align*}
	 Note that
	 \begin{align}
	 	|I_3|\leqslant |d_j|e^{y\text{Im}z_j}|\tilde{\mu}_1(k_j)|\int_{\mathbb{R}}\frac{|r_2(k)|}{\sqrt{s^2+\text{Im}k_j^2}}ds \leqslant  C(T,\|R(0)\|_{Y})e^{y\text{Im}z_j},
	 \end{align}
 which leads to $I_3(\cdot)\in L^{2,2}(\mathbb{R}^+)$ immediately.
	   It also implies from  Lemma \ref{lemma13}  that  $I_1(\cdot)\in L^{2,2}(\mathbb{R}^+)$.
From  Lemma \ref{lemma13},  Lemma \ref{lemmar1} and  \eqref{lemma14},  applying the   H\"older inequality yields
	\begin{align*}
		\Big|\langle y\rangle^4 I_2(y)\Big|&=\Big|\langle y\rangle^2\int_{\mathbb{R}}(\tilde{\mu}_2(s;y)-1)r_1(s)e^{is y/2}\mathcal{C}_+(r_2e^{-ik y/2})(s) ds\Big|\\
		&\leqslant \parallel \langle y\rangle^2 (\tilde{\mu}_2(\cdot;y)-1)r_1(\cdot) \parallel_{L^2}  \Big{\Vert} \langle y\rangle^2 \mathcal{C}_+(r_2e^{-i\cdot y/2}) (\cdot)\Big{\Vert}_{L^2}\leqslant  C(T,\|R(0)\|_{Y}),
	\end{align*}
	which in turn implies that  $I_2(\cdot)\in L^{2,2}(\mathbb{R}^+)$.
The second  integral in (\ref{resc1}) can be controlled by $\| r \|_{H^{1,2}\cap H^{2,1}}$  with
 the Fourier transformation and the third   integral is  bounded.
 It is then concluded that  $m_x\in L^{2,2}(\mathbb{R}^+)$  and
	\begin{align*}
		 \| m_x\|_{L^{2,2}(\mathbb{R}^+)}
		\leq  C(T,\|R(0)\|_{Y}).
	\end{align*}

 Taking $y$-derivative of the first reconstruction formula \eqref{res3} leads to
	\begin{align*}
		&\partial_y	\lim_{k\to \infty}k (M_{l,1}^{(2)}-(\eta-2) )= -\frac{1}{4\pi }\int_{\mathbb{R}} s(\tilde{\mu}_2(s)-1) r_2(s)e^{2\theta}ds+\frac{1}{2\pi i}\int_{\mathbb{R}} \partial_y\tilde{\mu}_2(s) r_2(s)e^{2\theta}ds\\
		&-\frac{1}{4\pi }\int_{\mathbb{R}} sr_2(s)e^{2\theta}ds+\sum_{j=1}^{{N_0}}\tilde{d}_je^{y\text{Im}z_j}\left( \text{Im}z_j\tilde{\mu}_2(\bar{k}_j)+	\partial_y\tilde{\mu}_2(\bar{k}_j)\right).
	\end{align*}
	Similar to the above calculation, it appears from  $ kr_2\in W^{1,2}(\mathbb{R})$ that
	\begin{align}
		\Big{|}	\langle y\rangle^2 \partial_y	 \lim_{k\to \infty}k (M_{l,1}^{(2)}-(\eta-2)	)\Big{|}	\leqslant  C(T,\|R(0)\|_{Y}).
	\end{align}
Then from
\begin{align*}
	m_{xx}=\alpha_1^{-1}( \alpha_1+ 2 \partial_ y\alpha_1) \partial_ym_x,
\end{align*}
it is found  that  $m_{xx}\in L^{2,1}(\mathbb{R}^+)$ with
	\begin{align*}
		 \| m_{xx} \|_{L^{2,2}(\mathbb{R}^+)}
		\leq  C(T,\|R(0)\|_{Y}).
	\end{align*}
	
On account of   the  reconstruction formula \eqref{res3} and
	  Lemma \ref{corm3}, it is deduced  that
	\begin{align}
		b_+=&-\frac{1}{2\pi i}\int_{\mathbb{R}} (\tilde{\mu}_1(k)-(\eta-2)) r_1(k)e^{-2\theta}ds\nonumber\\
&-\frac{1}{2\pi i}\int_{\mathbb{R}}(\eta-2) r_1(k)e^{-2\theta}ds +\sum_{j=1}^{{N_0}} d_je^{y\text{Im}z_j}\tilde{\mu}_1(k_j).\label{resc2}
	\end{align}
	The estimate of  $	|b_+|\leqslant C(T, \|R(0)\|_{Y})$ is obviously a  result from Lemmas \ref{lemmar1} and \ref{corm3}.
	This thus  implies that $b_+$ is uniformly bounded with respect to $y$.
	For its $L^2$-integrability, similar to  the above step, combining the jump condition and the integral equation of $M^{(2)}$, it is deduced that
	\begin{align}
		\parallel b_+ \parallel_{L^{2,2}(\mathbb{R}^+)}\leqslant C(T,\|R(0)\|_{Y}).
	\end{align}
	The next step is to obtain  $m-1 \in L^{2,2}. $  Since
	\begin{align*}
		\partial_x b_+=\frac{m^2-1}{m},\  \quad {\rm and} \quad 	\partial_x^2 b_+=\frac{m^2+1}{m^2}m_x.
	\end{align*}
Together with the boundedness of $\eta=1-\frac{1}{m}$ shown in Lemma \ref{lemma4.2} and boundedness of $m$ shown in Lemma \ref{the1}, it  is accomplished that $\partial_x^2b_+\in L^{2}(\mathbb{R})$.
	Then applying the  Sobolev-Gagliardo-Nirenberg inequality yields   $	\frac{\partial}{\partial x}b_+\in L^{2}(\mathbb{R})$, which means $m-1\in L^{2}(\mathbb{R})$. To arrive at the weighted integrable property, one needs to consider  the $x$-derivative for the  function $xb_+$,
	\begin{align*}
		\partial_x (xb_+)=x\frac{m^2-1}{m}+b_+,\ 	\partial_x^2 (xb_+)=2\frac{m^2-1}{m}+x\frac{m^2+1}{m^2}m_x.
	\end{align*}
	Then $xb_+$, $\partial_x^2 (xb_+)\in L^{2}(\mathbb{R})$.
	Analogously, by the  Sobolev-Gagliardo-Nirenberg inequality, it is ascertained that $\partial_x (xb_+)\in L^{2}(\mathbb{R})$, which implies $x\frac{m^2-1}{m}\in L^{2}(\mathbb{R})$. Again we analyze  $x^2b_+$,
	\begin{align*}
		\partial_x (x^2b_+) =x^2\frac{m^2-1}{m}+2xb_+,\ 	\partial_x^2 (x^2b_+) =2b_++4x\frac{m^2-1}{m}+x^2\frac{m^2+1}{m^2}m_x,
	\end{align*}
	and obtain that $x\frac{m^2-1}{m}\in L^{2}(\mathbb{R})$. Therefore, $m(t,\cdot)-1\in L^{2,2}(\mathbb{R}) $ with
	\begin{align*}
		 \| m(t,\cdot)-1 \|_{L^{2,2}(\mathbb{R})}
		\leq  C(T,\|R(0)\|_{Y}).
	\end{align*}
 This completes the proof of Lemma \ref{lemma55}.
\end{proof}
\subsection{Proof of the main theorem }
We are now in the position to give a proof of  Theorem  \ref{last}.
\begin {proof} [Proof of Theorem \ref{last}]
 First, the estimate (\ref{wwdj})  in  Lemma \ref{lemma55} implies that
\begin{align}
m(t,\cdot)-1\in H^{2,1}(\mathbb{R})\cap H^{1,2}(\mathbb{R} ), \;   \forall t >0. \label{estmate5}
\end{align}

Next, the continuity of $m$ on $t > 0 $ follows from the same step in Lemma \ref{the1}.
In fact,  the operator $\tilde{\mathcal{K}}(t)$   defined by \eqref{tidN} and function $\tilde{N}_0$ defined by  \eqref{N0} corresponding to scattering data $R(t)  $
are  continuous on $t >0$. It also gives that $M^{(2)}_l $ and $M^{(2)}_r $ are   continuous on $ t > 0$.
This in turn implies that all terms on the right of \eqref{resc1} and \eqref{resc2} are continuous on $t > 0$.
This is  concluded that the  global solution $m$  in the sense of norm  $\| m(t,\cdot)-1\|_{H^{2,1} \cap H^{1,2} } $  is   continuous  on any $t > 0$ , which together    (\ref{estmate5})
implies that
$$m -1\in C([0, +\infty);  H^{2,1}(\mathbb{R})\cap H^{1, 2} (\mathbb{R})).$$
This  completes the proof of Theorem \ref{last}.
\end {proof}


\noindent\textbf{Acknowledgements}
The work of Fan and Yang is partially  supported by  the National Science
Foundation of China under grants 12271104, 51879045 and 12247182. The work of Liu is partially supported by the Simons Foundation under grant 499875.






\end{document}